%% file: thesis.tex
\documentclass[a4paper,12pt,twoside]{report}
\usepackage[left=2.5cm,right=2.5cm,top=2cm,bottom=3cm]{geometry}
\usepackage[toc, page]{appendix}

\include{thesis.preamble}

\begin{document}

\title{\LARGE {\bf Discrete functional inequalities on lattice graphs}
 \vspace*{6mm}
}

\author{Shubham Gupta}

\maketitle
\preface
\input{declaration}

\input{copyright}

\input{abstract/abstract}

\input{acknowledgements/acknowledgements}


    \newpage
    \pagestyle{uheadings}
    \tableofcontents
    \pagestyle{uheadings}
    \cleardoublepage
    \pagestyle{uheadings}
    \input{Table_of_symbols/table_of_symbols}

    \pagestyle{plain}
    \pagestyle{uheadings}
    \listoffigures
    \addcontentsline{toc}{chapter}{List of Figures}
    \pagestyle{plain}
    \pagestyle{uheadings}
    \pagenumbering{arabic}
    \onehalfspacing

\input{Introduction/Introduction}
\input{1D_Hardy_inequality/1D_Hardy_inequality}

\input{higher_Hardy_inequality/higher_Hardy_inequality}

\input{1D_Rearrangement_inequality/1D_Rearrangement_inequality}

\input{higher_Rearrangement_inequality/higher_Rearrangement_inequality}




\addcontentsline{toc}{chapter}{Bibliography}


\input{appendix}

\end{document}

%% file: thesis.preamble.tex
\usepackage{graphicx}
\usepackage{verbatim}
\usepackage{latexsym}
\usepackage{mathchars}
\usepackage{setspace}
\usepackage{amsmath,amsfonts,amsthm,amssymb,amsxtra}
\usepackage{bbm} 
\usepackage{hyperref}	

\usepackage{graphicx}

\usepackage{tikz, caption}

\usetikzlibrary{shapes}

\input{blocked.sty}
\input{uhead.sty}
\input{boxit.sty}
\input{icthesis.sty}

\newcommand{\NN}{{\sf I\kern-0.14emN}}   
\newcommand{\ZZ}{{\sf Z\kern-0.45emZ}}   
\newcommand{\QQQ}{{\sf C\kern-0.48emQ}}   
\newcommand{\RR}{{\sf I\kern-0.14emR}}   

\newcommand{\N}{\mathbb{N}}

\renewcommand{\phi}{\varphi}
\newcommand{\R}{\mathbb{R}}

\newcommand{\Z}{\mathbb{Z}}
\newcommand{\C}{\mathbb{C}}

\newcommand{\syncc}{~\stackrel{\textstyle \rhd\kern-0.57em\lhd}{\scriptstyle L}~}

\newtheorem{theorem}{Theorem}[chapter]
\newtheorem{corollary}{Corollary}[chapter]
\newtheorem{lemma}{Lemma}[chapter]

\theoremstyle{definition}
\newtheorem{definition}{Definition}[chapter]
\newtheorem{remark}{Remark}[chapter]

%% file: declaration.tex
\chapter*{Declaration}
I certify that this thesis, and the research to which it refers, are the product of my own work,
and that any ideas or quotations from the work of other people, published or otherwise, are
fully acknowledged in accordance with the standard referencing practices of the discipline.  \\ \vspace{49pt}

\hfill \emph{Shubham Gupta}

%% file: copyright.tex
\chapter*{Copyright}
The copyright of this thesis rests with the author. Unless otherwise indicated, its contents are licensed under a Creative Commons Attribution-Non Commercial-No Derivatives 4.0 International Licence (CC BY-NC-ND). Under this licence, you may copy and redistribute the material in any medium or format on the condition that; you credit the author, do not use it for commercial purposes and do not distribute modified versions of the work. When reusing or sharing this work, ensure you make the licence terms clear to others by naming the licence and linking to the licence text. Please seek permission from the copyright holder for uses of this work that are not included in this licence or permitted under UK Copyright Law.

%% file: abstract/abstract.tex
\chapter*{Abstract}
\addcontentsline{toc}{chapter}{Abstract}

In this thesis, we study problems at the interface of analysis and discrete mathematics. We discuss analogues of well known \emph{Hardy-type} inequalities and \emph{Rearrangement} inequalities on the lattice graphs $\Z^d$, with a particular focus on behaviour of sharp constants and optimizers.

In the first half of the thesis, we analyse Hardy inequalities on $\Z^d$, first for $d=1$ and then for $d \geq 3$. We prove a sharp weighted Hardy inequality on integers with power weights of the form $n^\alpha$. This is done via two different methods, namely \emph{super-solution} and \emph{Fourier} method. We also use Fourier method to prove a weighted Hardy type inequality for higher order operators. After discussing the one dimensional case, we study the Hardy inequality in higher dimensions ($d \geq 3$). In particular, we compute the asymptotic behaviour of the sharp constant in the discrete Hardy inequality, as $d \rightarrow \infty$. This is done by converting the inequality into a continuous Hardy-type inequality on a torus for functions having zero average. These continuous inequalities are new and interesting in themselves.

In the second half, we focus our attention on analogues of Rearrangement inequalities on lattice graphs. We begin by analysing the situation in dimension one. We define various notions of rearrangements and prove the corresponding \emph{Polya-Szeg\H{o}} inequality. These inequalities are also applied to prove some weighted Hardy inequalities on integers. Finally, we study Rearrangement inequalities (Polya-Szeg\H{o}) on general graphs, with a particular focus on lattice graphs $\Z^d$, for $d \geq 2$. We develop a framework to study these inequalities, using which we derive concrete results in dimension two. In particular, these results develop connections between Polya-Szeg\H{o} inequality and various isoperimetric inequalities on graphs.

%% file: acknowledgements/acknowledgements.tex
\chapter*{Acknowledgements}
\addcontentsline{toc}{chapter}{Acknowledgements}

First, I would like to express my gratitude to my advisor Ari Laptev for his unwavering support and encouragement throughout my PhD. He has been a source of many exciting mathematical problems, and discussing them with him has always been a very joyful experience. I am especially thankful to him for giving me enough freedom to develop and pursue my own mathematical taste, even when it did not align with his expertise. Thank you Ari! 

In addition, I am indebted to my teachers Nandini Nilakantan, T. Muthukumar, Prosenjit Roy and Adi Adimurthi for developing my interest in mathematics through excellent lectures and for guiding me during my undergraduate studies. 

I have thoroughly enjoyed my time as a PhD student at Imperial College London. A part of the reason is many good friends I made here and I would like to thank all of them: Larry, Xiangfeng, Kamilla, William, Adam, Max, Ali and Shreya. I would also like to thank our administrator Aga for her help in various bureaucratic matters as well as donors of President's PhD scholarship for funding my research.

I am also grateful for the company and friendship of various people in UK who made my life less stressful: Bishal, Ilina, Vijay, Rutvij. I would also like to thank my friends in India: Panu, Harsh, Prashumn, Vikas, Dheeraj, Thoury and Divyat for always being there for me, especially during the tough and isolating times of Covid. 

Finally, I am thankful to all my family members for supporting me throughout this journey. 

%% file: Table_of_symbols/table_of_symbols.tex
\chapter*{List of Symbols}
\addcontentsline{toc}{chapter}{List of Symbols}

\begin{tabular}{cp{0.78\textwidth}}
\\ \vspace{9pt}
  $\N_0 = \Z^+$ &\hfill The set of non-negative integers \\ \vspace{9pt}
  
  $\text{Re} \hspace{2pt}a$ & \hfill The real part of a complex number $a$ \\ \vspace{9pt}
  
  $\text{Im} \hspace{2pt} a$ & \hfill The imaginary part of a complex number $a$ \\ \vspace{9pt} 
  
  $\lfloor x \rfloor$ &\hfill The largest integer less than or equal to a real number $x$ \\ \vspace{9pt} 
  
  $\lceil x \rceil$ &\hfill The smallest integer greater than or equal to a real number $x$ \\ \vspace{9pt} 
  
  $|\cdot|$ &\hfill $\ell^2$ norm on $\R^d$\\ \vspace{9pt}
  
  $\Gamma(x)$ &\hfill The Gamma function \\ \vspace{9pt}
  
  ${x \choose y}:= \frac{\Gamma(x+1)}{\Gamma(x-y+1)\Gamma(y)}$ &\hfill Binomial coefficient\\ \vspace{9pt}
  
  $d^k := d^k/dx^k$ & \hfill $k^{th}$ derivative of the function\\ \vspace{9pt} 
 $C^\infty$ &\hfill The set of infinitely differentiable functions \\ \vspace{9pt}
  $C_0^\infty$ &\hfill The set of infinitely differentiable functions with compact support\\ \vspace{9pt}
  $C_c$ &\hfill The set of finitely supported functions \\ \vspace{9pt}
  $\ell^2$ &\hfill The set of sequences whose square of the absolute value is summable\\ \vspace{9pt}
  $L^2$ &\hfill The set of functions whose square of the absolute value is integrable\\ \vspace{9pt}

\end{tabular}

%% file: Introduction/Introduction.tex
\chapter{Introduction}\label{ch:intro}
It is well known that functional inequalities form a backbone of analysis on continuous spaces (Euclidean space or manifolds in general). They are used in various different areas of analysis: calculus of variations, existence and regularity theory of partial differential equations, mathematical physics, spectral geometry and many more. However, functional inequalities on discrete spaces, such as graphs, are not very well studied. Till date, very little is known about the discrete analogues of various important functional inequalities. In fact, even the most basic questions about the behaviour and structure of optimizers of these inequalities remain unanswered. For example, existence of optimizers of a classical inequality of Sobolev on the integer lattice was established recently in 2021 \cite{bobo}. 

A lot of interest has been emerging in the community to understand various geometric, functional and spectral inequalities on discrete spaces, as well as to understand connections between them. While some of these investigations are purely curiosity driven, they often end up having deep implications in both pure and applied mathematics. For instance, in \cite{alon1986eigenvalues}, Alon proved a discrete analogue of a famous result of Cheeger called \emph{Cheeger inequality}. This inequality relates the first non-trivial eigenvalue of the Laplacian operator with the connectivity properties of a graph. The result has been used in designing approximate algorithms for the \emph{graph bisection problem} \cite{pothen1990partitioning, chan1997optimality}, which has found various practical applications such as, designing layouts of electronic circuits on computer boards \cite{corrigan1979placement, friedman1985theory}. 

Despite independent interest, studying discrete problems has been fruitful in answering questions in the continuous setting; as often, discrete problems can be seen as an approximation of continuous ones \cite{approx_ciaurri, lukas_approx}. In \cite{lukas_approx}, Schimmer approximated Schr\"odinger operators on $\R$ with discrete Jacobi operators on $\Z$ to prove some sharp \emph{Lieb-Thirring inequalities}: they give an upper bound on the negative eigenvalues of Schr\"odinger operator in terms of a norm of its potential. Thus, it is a worthwhile effort to develop tools and techniques for a better understanding of analysis on discrete spaces. This thesis is a  contribution in that direction.

In this thesis, we study discrete analogues of two well known classes of functional inequalities, namely, \emph{Hardy-type inequalities} and \emph{Rearrangement inequalities}, often focusing on the nature of optimizers and \emph{sharp} constants. In the next two sections, we give a brief overview of these inequalities in the continuum, only focusing on the components that we extend later to the discrete setting. For thorough literature on the subject we recommend Balinsky-Evans-Lewis \cite{balinsky} for Hardy-type inequalities, and Baernstein \cite{baernstein} for Rearrangement inequalities. We finally end this chapter summarizing the content of the rest of the thesis. 

\section{Overview of Hardy-type inequalities}
Hardy inequality in its most basic form is one of the first mathematical formulation of \emph{Heinsenberg's uncertainty principle}, which says that the position and momentum of a particle cannot be measured simultaneously. It can be written as:
\begin{equation}\label{1.1}
    \frac{(d-2)^2}{4}\int_{\R^d} |u(x)|^2 |x|^{-2} dx \leq  \int_{\R^d} |\nabla u(x)|^2 dx,  
\end{equation}
for $d \geq 3$ and $u \in C_0^\infty(\R^d)$, the space of smooth and compactly supported functions. The constant $(d-2)^2/4$ in LHS of \eqref{1.1} is \emph{sharp}, that is, \eqref{1.1} fails to hold true for a strictly bigger constant. Hardy inequalities have been quite useful in the study of spectral theory of \emph{Schr\"odinger operators} $-\Delta-V$. In \cite{agmon_decay, pinchover_agmon}, authors used an \emph{optimal} Hardy inequality to study the localization properties of eigenfunctions of Schr\"odinger operators, which have found applications in the study of metal-insulator transitions, disordered photonic crystals and other physical phenomena \cite{mott1972metal, schwartz2007transport}. In \cite{Ari_calogero} Laptev, Read and Schimmer used Hardy's inequality to prove \emph{Calogero inequality}, which is linked with a one-parameter family of spectral inequalities called \emph{Lieb-Thirring inequality} \cite{lieb1997inequalities, lieb2001bound}. This inequality gives an upper bound on the sum of absolute values of negative eigenvalues of the Schr\"odinger operator $-\Delta -V$ in terms of the integral of the potential $V$. Lieb and Thirring \cite{lieb2001bound} used this spectral inequality to give a proof of \emph{stability of matter}, where the constant in the Lieb-Thirring inequality enters their stability estimate. The value of the sharp constant in Lieb-Thirring inequality is a famous open problem in spectral theory of Schr\"odinger operators. For many other spectral theoretic applications of Hardy-type inequalities, we cite \cite{frank2022improved, rupert_davies, hoffmann2021hardy}. 

Due to its vast applicability, Hardy inequality \eqref{1.1} has received enormous attention in the past 100 years. It has been generalized in various directions, and to delve into these developments is beyond the scope of this thesis. We will discuss one particular extension of \eqref{1.1} to higher order operators, which would be of interest later. In 1954 Rellich \cite{rellich} proved the following second order version of Hardy inequality \eqref{1.2} which bears his name:
\begin{equation}\label{1.2}
     \frac{d^2(d-4)^2}{16}\int_{\mathbb{R}^d} \frac{|u(x)|^2}{|x|^4} dx \leq \int_{\mathbb{R}^d} |\Delta u|^2 dx.
\end{equation}
It was later proved by Schmincke \cite{schmincke} and Bennett \cite{bennett} using methods different from those of Rellich. The ideas of these papers were later extended by Davies and Hinz \cite{davies} to prove the weighted $L^p(\mathbb{R}^d)$ versions of inequality \eqref{1.2} for $p > 1$. They applied these weighted inequalities inductively to prove a Hardy-type inequality for the higher order operators $\Delta^m$, $\nabla (\Delta^m)$:
\begin{theorem}[E.B. Davies and A.M. Hinz \cite{davies}]\label{thm:DH}
Let $m \in \mathbb{N}_0$,  $d \in \mathbb{N}$ and $u \in C_0^\infty(\R^d \setminus \{0\})$
\begin{enumerate}
    \item If $p>1$ and $d > 2mp$ then 
    \begin{equation}\label{1.3}
        \int_{\mathbb{R}^d}|\Delta^m u|^p dx \geq C_1(m, p, d)\int_{\mathbb{R}^d} \frac{|u(x)|^p}{|x|^{2mp}} dx,
    \end{equation}
    where 
    $$C_1(m, p, d) := p^{-2mp}\Big(\prod_{k=1}^m (d-2kp)(2(k-1)p + (p-1)d)\Big)^p.$$
    \item If $p>1$ and $ d > (2m+1)p$ then
    \begin{equation}\label{1.5}
        \int_{\mathbb{R}^d}|\nabla(\Delta^m u)|^p dx \geq C_2(m, p, d) \int_{\mathbb{R}^d}\frac{|u(x)|^p}{|x|^{(2m+1)p}} dx,
    \end{equation}
    where 
    \begin{equation}\label{1.6}
        C_2(m, p, d):= \Big(\frac{p^{2m+1}}{d-p}\Big)^{-p}\Big(\prod_{k=1}^m (d-(2k+1)p)((2k-1)p+(p-1)d)\Big)^p.
    \end{equation}
\end{enumerate}
The constants in inequalities \eqref{1.3} and \eqref{1.5} are sharp.
\end{theorem}
The initial part of this thesis (Chapters \ref{ch:1d-hardy} and \ref{ch:higher-hardy}) concerns a discrete analogue of inequality \eqref{1.1} and Theorem \ref{thm:DH} on the lattice graph $\Z^d$, with a particular focus on the behaviour of the sharp constants. 

\section{Overview of Rearrangement inequalities}

On the Euclidean space $\R^d$, a well known rearrangement procedure converts an arbitrary function $f$ into a spherically symmetric function by rearranging its level sets ($\{x: f(x)>t\}$) in such a way that they are preserved in volume while being radially centered around origin. In particular, the rearrangement procedure converts a $d$-dimensional function into a one-dimensional function, thereby reducing the complexity of the problem at hand. This feature has been quite useful in solving problems from several areas of analysis. For instance, rearrangement theory has  been used to prove sharp functional inequalities \cite{alvino1, alvino2, lieb1, moser, talenti}, to study variational problems \cite{lenzmann, lieb2, lieb3}, and to study problems in spectral geometry \cite{henrot}.

Formally, let $f: \R^d \rightarrow \R$ be a measurable function \emph{vanishing at infinity} (i.e. level sets of $|f|$ has finite measure, in other words $|\{x: |f(x)|>t\}| < \infty$ for all $t >0$). Then the \emph{symmetric decreasing rearrangement of f} is defined by
\begin{equation}\label{1.6}
    f^*(x) := \int_0^\infty \chi_{\{|f|>t\}^*}(x) dt,
\end{equation}
where given $\Omega \subseteq \R^d$, $\Omega^*$ is the open ball centered at origin whose measure coincides with $\Omega$. Comparing \eqref{1.6} with the layer-cake formula for $f^*$:
\begin{equation}
    f^*(x) = \int_0^\infty \chi_{\{f^*>t\}}(x) dt,
\end{equation}
one can conclude that the level sets of $f^*$ are obtained by rearranging the level sets of $|f|$:
$$ \{x: f^*(x) > t\} = \{x: |f(x)|>t\}^*.$$
One can easily deduce the following properties of $f^*$:
\begin{itemize}
    \item $f^*$ is a radially decreasing function, that is, 
    \begin{align*}
        &f^*(x) = f^*(y) \hspace{9pt} \text{if} \hspace{5pt} |x|=|y| \\
        &f^*(x) \geq f^*(y) \hspace{9pt} \text{if} \hspace{5pt} |x|< |y|.
    \end{align*}
    \item Level sets of $|f|$ and $f^*$ have the same volume:
    $$|\{x: f^*(x)>t\}| = |\{x: |f(x)|>t\}|.$$ This further implies that the $L^p$ norm is preserved under rearrangement, that is, $\|f\|_{L^p} = \|f^*\|_{L^p}$ for all $p>0$.
    \item Perimeter of the level sets decreases under rearrangement: $$\text{Per}(\{x: f^*(x)>t\}) \leq \text{Per}(\{x: |f(x)|>t\}).$$ This is a direct consequence of a classical isoperimetric inequality (amongst sets of fixed volume ball has the smallest perimeter). 
\end{itemize}
There are various inequalities comparing/connecting $f^*$ with $f$, and one of them is celebrated \emph{Polya-Szeg\H{o}} inequality, which  states that function becomes ``smoother" under rearrangement. More precisely,  
\begin{equation}\label{1.9}
    \int_{\R^d} |\nabla f^*|^p dx \leq \int_{\R^d} |\nabla f|^p dx,
\end{equation}
for all $p \geq 1$. Inequality \eqref{1.9} can be regarded as a functional version of the isoperimetric inequality and, as such, has many applications in analysis, see Baernstein \cite{baernstein}, Lieb-Loss \cite{liebloss} and Polya-Szeg\H{o} \cite{polya-szego}. 

This important analytic tool, namely rearrangement theory is missing when the underlying space is a graph. In the later part of the thesis (Chapter \ref{ch:1d-rearrangement} and \ref{ch:higher-rearrangement}), we study to what extent rearrangement inequalities like \eqref{1.9} can hold true in the context of a graph, primarily focusing on the lattice graph $\Z^d$. We hope that by developing a theory of discrete rearrangements, one might be able to answer many open questions in the field of \emph{analysis on graphs}.       

\section{Structure of the thesis}
In Chapter \ref{ch:1d-hardy} we study a discrete analogue of Hardy inequality \eqref{1.1} with power weights $n^\alpha$ on the one dimensional integer lattice $\Z$. The motivation of studying the 1-d weighted inequality comes from a `possible' connection with its higher dimensional version. We prove the inequality with sharp constant using two different methods: \emph{Super-solution} and \emph{Fourier transform} method. Super-solution method gives the inequality for $\alpha \in [0,1) \cup [5, \infty)$ and Fourier transform method gives the sharp inequality when $\alpha$ is an even natural number. The results are based on \cite{gupta1} and \cite{gupta2}.  

In Chapter \ref{ch:higher-hardy} we analyse a discrete analogue of Hardy inequality \eqref{1.1} on the lattice graph $\Z^d$. We compute the asymptotic behaviour of the sharp constant as the space dimension $d$ goes to infinity. We also compute the large dimension behaviour of sharp constants in the higher order discrete Hardy inequality: a discrete analogue of Theorem \ref{thm:DH}. In the process, we prove some new Hardy-type inequalities on a torus for functions with zero average. The problem of finding the exact values of sharp constants in these inequalities is completely open. The content is based on \cite{gupta4}.

In Chapter \ref{ch:1d-rearrangement} we look at various different notions of rearrangements on the one dimensional lattice and prove the corresponding discrete Polya-Szeg\H{o} inequalities. We finally use them to prove a weighted discrete Hardy inequality on integers. The results can be found in \cite{gupta3}.

Finally in Chapter \ref{ch:higher-rearrangement} we develop an approach to study rearrangement inequalities on general graphs, in particular on lattice graphs. Apart from general results, we present some concrete results for the two dimensional lattice graph $\Z^2$. The results in this chapter develop a connection between Polya-Szeg\H{o} inequality and various isoperimetric inequalities on graphs. These results are obtained jointly work with Stefan Steinerberger \cite{gupta5}.

%% file: 1D_Hardy_inequality/1D_Hardy_inequality.tex
\chapter{Hardy inequality: 1D case}\label{ch:1d-hardy}

\section{Introduction}
Historically, the original Hardy inequality appeared for the first time in G.H. Hardy's proof of Hilbert's theorem \cite{hardy}. It was later developed and improved during the period 1906-1928: \cite{kufner} contains many stories and contributions of other mathematicians such as E. Landau, G. Polya, I. Schur and M. Riesz in the development of Hardy inequality. In its modern form, Hardy inequality can be written as:
\begin{equation}\label{2.1}
    \sum_{n=1}^\infty |u(n)-u(n-1)|^2 \geq 1/4 \sum_{n=1}^\infty \frac{|u(n)|^2}{|n|^2},
\end{equation}
for finitely supported functions $u: \N_0 \rightarrow \R$ with $u(0) = 0$. The constant appearing in RHS of \eqref{2.1} is the best possible. A continuous analog of \eqref{2.1} was developed alongside, and on the \emph{half-line} $(0, \infty)$, reads as:
\begin{equation}\label{2.2}
    \int_0^\infty |u'(x)|^2 dx \geq 1/4 \int_0^\infty \frac{|u(x)|^2}{|x|^2} dx,
\end{equation}
for smooth and compactly supported functions $u: (0, \infty) \rightarrow \R$. The constant \eqref{2.2} is sharp as well. Since its discovery inequality \eqref{2.2} has been extended in various directions: $d$-dimensional space, wide range of operators, more general domains, various other classes of functions, and many more. These extensions have turned out to be quite fruitful, as discussed in Chapter \ref{ch:intro}. However, its discrete counterpart \eqref{2.1} has not received much attention until recently. In this and the next chapter we study an extension of \eqref{2.1} in three directions: weighted case, higher dimensional case, higher order operator case. More precisely, we ask:

Question 1: What is the sharp constant in \eqref{2.1} when we allow \emph{power weights} $n^\alpha$, for a real parameter $\alpha$? In other words, what is the value of sharp constant $c(\alpha)$ in 
\begin{equation}\label{2.3}
    \sum_{n=1}^\infty |u(n)-u(n-1)|^2 n^\alpha \geq c(\alpha) \sum_{n=1}^\infty |u(n)|^2 n^{\alpha-2}?
\end{equation}

Question 2: What is the value of the sharp constant in a $d$-dimensional analogue of \eqref{2.1}? Let $u: \Z^d \rightarrow \R$ be a finitely supported function, then what is the best possible $c(d)$ in 
\begin{equation}\label{2.4}
    \sum_{n \in \Z^d} \sum_{j=1}^d |u(n)-u(n-e_j)|^2 \geq c(d) \sum_{n \in \Z^d} \frac{|u(n)|^2}{|n|^2},
\end{equation}
with $e_j$ being the $j^{th}$ basis element of $\R^d$?

Question 3: What are the analogues of inequalities \eqref{2.3} and \eqref{2.4}, when the first order difference operator is replaced by higher order difference operators? 

In the continuum, it is not very hard to derive $d$-dimensional Hardy inequality \eqref{1.1} from a weighted one dimensional integral inequality of the type \eqref{2.3}. The key is to express the integrals in polar coordinates. This connection is missing in the discrete setting, since there does not seem to be a `good candidate' for polar coordinates on $\Z^d$ for this purpose \footnote{In Appendix \ref{appendix:A} we defined polar coordinates on $\Z^2$ and used it prove discrete Hardy inequality for antisymmetric functions.}. To that end, we ask: Is it possible to prove Hardy inequality \eqref{2.4} using one dimensional inequalities of the type \eqref{2.3}? Answering this question was our original motivation to study the one dimensional weighted inequality \eqref{2.3}. The connection is still unknown!

Unfortunately the questions asked above are not understood completely. A lot of questions around the discrete Hardy inequalities in higher dimensions remain unanswered. One of the major hurdles in answering these questions is that calculus breaks down in the discrete setting. This makes it difficult to transfer the proofs/techniques from continuum to the discrete setting. Some of the difficulties are often of a technical nature, which can be worked around. This happens to be the case in most of the one-dimensional discrete Hardy inequalities known so far. However, some of the difficulties are of a more fundamental nature, that often appear in the study of higher dimensional Hardy inequalities \footnote{See Appendix \ref{appendix:B} for a discussion on hurdles appearing in higher dimensions.}. This makes it harder to understand the higher dimensional case. As we shall see later, there turns out to be a fundamental difference between continuous Hardy inequality \eqref{1.1} and discrete Hardy inequality \eqref{2.4} in higher dimensions.

In the rest of this section, we briefly go through some of previously known results and then describe our contributions.

\subsection{Previous Results}
There is a lot of literature on one-dimensional discrete Hardy-type inequalities, looking into various generalizations of \eqref{2.1}: allowing weights, taking general $\ell^p-\ell^q$ inequality instead of $\ell^2-\ell^2$ inequality, etc (see \cite{kufnerbook} and references therein). However, inequality \eqref{2.3} was only known for $\alpha \in [0,1)$ until recently (see \cite{lefevre} for a short proof). A lot of work has also been done in improving the classical Hardy weight $1/(4n^2)$ in \eqref{2.1}. It was first improved by Keller, Pinchover and Pogorzelski \cite{kellerimproved}. Later an explicit remainder term was computed in the improved inequality by Krej\v ci\v r\' ik and \v Stampach \cite{krejcirik1}. Recently, Krej\v ci\v r\' ik, Laptev and \v Stampach \cite{krejcirik2} extended the result proved in \cite{krejcirik1} to obtain various new `critical' Hardy weights. Finally, we cite a recent paper \cite{suragon}, where an improvement of \eqref{2.1} was recorded with a different flavour. 

To our best knowledge, the $d$-dimensional Hardy inequality appeared for the first time in papers \cite{solomyak1, solomyak2} with the aim of studying spectral properties of discrete Schr\"odinger operator. But their approach did not give an explicit constant in \eqref{2.4}. In \cite{laptev}, Kapitanski and Laptev computed an explicit constant in \eqref{2.4}. However, their constant did not scale with dimension as $d \rightarrow \infty$. Research in Hardy-type inequalities gained a lot of momentum after the work of Keller, Pinchover and Pogorzelski \cite{keller1}. They studied Hardy-type inequalities in the general setting of graph, focusing on \emph{optimal Hardy weights}\footnote{Optimal Hardy weights were used by Keller and Pogorzelski in \cite{keller3} to study the decay properties of eigenfunctions of discrete Schr\"odinger operators. This is a discrete analogue of a celebrated result in mathematical physics due to Agmon. Similar decay properties were also studied by Filoche, Mayboroda and Tao \cite{tao} for a fairly big class of finite real matrices.}. They compute an explicit `optimal' Hardy weight for a general class of graphs, in terms of their Green function. However, their theory when applied to the standard graph on $\Z^d$ does not give the classical weight $|n|^{-2}$. Rather, it gives a weight which grows as $|n|^{-2}$, as $|n| \rightarrow \infty$. 

Discrete Hardy inequalities for higher order difference operators are better understood in dimension one than higher dimensions (similar to the first order case). In \cite{gerhat} Gerhat, Krej\v ci\v r\' ik and \v Stampach obtained an discrete Hardy inequality on non-negative integers for second order difference operator with the sharp constant. Recently, Huang and Ye \cite{huang} extended the inequality to operators of arbitrary order. In \cite{keller2} Keller, Pinchover and Pogozelski studied Hardy inequality for second order difference operator in the general setting of infinite graphs. However, similar to the first order case, their methods do not give the classical results on $\Z^d$.

\subsection{Our Contributions}
In \cite{gupta1}, we proved the weighted Hardy inequality \eqref{2.3} for $\alpha \in [0,1] \cup [5, \infty)$ with the optimal constant. This was done by adapting the \emph{super-solution method} from the continuum to prove Hardy type inequalities. We further improved \eqref{2.3}, when $\alpha \in [1/3,1) \cup \{0\}$, by adding non-negative terms in the RHS of \eqref{2.3}. We discuss these findings in Section \ref{sec: supersolution method}.

In \cite{gupta2}, we used \emph{Fourier transform} method to prove weighted Hardy inequalities and their improvements for non-negative even integers $\alpha$. We also derive explicit inequalities for higher order difference operators, as well as their weighted versions. This method also gave as a side product, a completely analytic proof of a non-trivial combinatorial identity. Its appearance in the context of discrete Hardy-type inequalities, and its proof via a functional identity is quite surprising to us. We elaborate on these results in Section \ref{sec: fourier method}.

In \cite{gupta4}, we studied the $d$-dimensional Hardy inequality \eqref{2.4} as well its higher order versions. In particular, we investigated the large dimension behaviour of sharp constants in these inequalities. We proved that the sharp constant $c(d)$ in Hardy inequality \eqref{2.4} grows linearly as $d \rightarrow \infty$. This proves that there are some fundamental differences between continuous and discrete Hardy inequalities, since the constant in the continuous Hardy inequality \eqref{1.1} grows as $d^2$, as $d$ goes to infinity. The asymptotic behaviour is proved by converting discrete the Hardy inequality into a continuous Hardy inequality on a torus. We present these results in Chapter \ref{ch:higher-hardy}. \newpage

\section{Supersolution method}\label{sec: supersolution method}
In this section we prove weighted Hardy inequality \eqref{2.3} for some specific values of $\alpha$. In fact, we prove the following two parameter family of weighted Hardy inequalities: If $\alpha, \beta \in \mathbb{R}$ then
\begin{equation}\label{2.5}
    \sum_{n=1}^\infty |u(n)-u(n-1)|^2 n^\alpha \geq \sum_{n=1}^\infty w_{\alpha,\beta}(n) |u(n)|^2,
\end{equation}
where 
\begin{equation}\label{2.6}
    w_{\alpha, \beta}(n) := n^\alpha \Bigg[ 1 + \Big(1+\frac{1}{n}\Big)^\alpha - \Big(1-\frac{1}{n}\Big)^\beta - \Big(1+\frac{1}{n}\Big)^{\alpha+\beta}\Bigg],
\end{equation}
for $n \geq 2$ and $w_{\alpha,\beta}(1) := 1+2^\alpha - 2^{\alpha+\beta}$. 

As will be shown, \eqref{2.5} contains the following power weights Hardy Inequalities as a special case:
\begin{equation}\label{2.7}
    \sum_{n=1}^\infty |u(n)-u(n-1)|^2 n^\alpha \geq \frac{(\alpha-1)^2}{4}\sum_{n=1}^\infty \frac{|u(n)|^2}{n^2} n^\alpha,
\end{equation}
whenever $\alpha \in[0,1)$ or $ \alpha \in [5,\infty)$,

and we have an improvement of \eqref{2.7} for $\alpha \in [1/3,1) \cup \{0\}$
\begin{equation}\label{2.8}
    \sum_{n=1}^\infty |u(n)-u(n-1)|^2 n^\alpha \geq \frac{(\alpha-1)^2}{4} \sum_{n=1}^\infty \frac{|u(n)|^2}{n^2} n^\alpha + \sum_{k=3}^\infty b_k(\alpha) \sum_{n=2}^\infty \frac{|u(n)|^2}{n^k}n^\alpha,
\end{equation}
where the non-negative constants $b_k(\alpha)$ are given by
\begin{equation}\label{2.9}
    b_k(\alpha) := {\alpha \choose k} - (-1)^k {(1-\alpha)/2 \choose k} -{(1+\alpha)/2 \choose k}.
\end{equation}

\begin{remark}\label{rem2.1}
Inequality \eqref{2.7} is derived from \eqref{2.5} by taking $\beta = (1-\alpha)/2$, and estimating $w_{\alpha,\beta}$ from below by $\frac{(\alpha-1)^2}{4} n^{\alpha-2}$. We would like to point out that this lower estimate on $w_{\alpha, \beta}$ fails to hold true when $\alpha <0$ or $\alpha \in(1,4)$ (this will be proved in Subsection \ref{subsec:limitations(supersolution)}). Due to this reason we fail to prove \eqref{2.7} for all non-negative $\alpha$. 
\end{remark}

\begin{remark}\label{rem2.2}
We would like to mention that \eqref{2.8} is true for all $\alpha \in  [0,1) \cup [5, \infty)$. However we conjecture that the constant $b_k(\alpha)$ is not non-negative for all $k \geq 3$ when $\alpha$ lies outside $[1/3,1) \cup\{0\}$, that is, when $\alpha \in (0,1/3) \cup [5, \infty)$ (it will be partially proved in Subsection \ref{subsec:limitations(supersolution)}).
\end{remark}

The method used to prove these results is an adaptation of a well known method of proving Hardy-type inequalities in the continuum, often referred to as supersolution method (see \cite{cazacu} for details). Let us sketch briefly the idea behind this method. The standard Hardy inequality in the continuous setting states 
\begin{equation}\label{2.10}
    \int_{\mathbb{R}^d} |\nabla u|^2 dx \geq \frac{(d-2)^2}{4} \int_{\mathbb{R}^d} \frac{|u(x)|^2}{|x|^2} dx,
\end{equation}
for all $u \in C_0^\infty(\mathbb{R}^d)$ and $d \geq 3$. The super-solution method to prove \eqref{2.6} roughly goes as follows. Let $u = \varphi \psi$. Then 
\begin{align*}
    |\nabla u|^2 = \psi^2 |\nabla \varphi|^2 + \varphi^2 |\nabla \psi|^2 + 2 \nabla \varphi \cdot \nabla \psi \varphi \psi.
\end{align*}
Applying integration by parts we obtain
\begin{align*}
    \int |\nabla u|^2 & = \int \psi^2 |\nabla\varphi|^2 + \int \phi^2 |\nabla \psi|^2 + 1/2\int \nabla(\varphi^2) \cdot \nabla (\psi^2)  \\
    &= \int \varphi^2 |\nabla \psi|^2 - \int \varphi \psi^2 \Delta \varphi  \geq \int \frac{-\Delta \varphi}{\varphi}|u|^2. 
\end{align*}
If $\varphi$ satisfies $\frac{-\Delta \varphi}{\varphi} \geq w$ then we have
\begin{equation}\label{2.11}
    \int |\nabla u|^2 dx \geq \int  w(x) |u|^2 dx.
\end{equation}
Therefore proving \eqref{2.10} boils down to a much simpler task of finding a solution of $ - \Delta \varphi - w \varphi \geq 0$ with $w = \frac{c}{|x|^2}$. This idea of connecting Hardy-type inequalities with solution of differential equations has been exploited a lot in the literature to prove various weighted version and improvements of first-order inequalities of the form \eqref{2.10}(\cite{bessel1}, \cite{bessel2}).    

This section is organized as follows. In Subsection \ref{subsec:main results(supersolution)} we formally state the main results. In Subsection \ref{subsec:supersolution method} we derive a discrete analogue of supersolution method, using which we prove \eqref{2.5}. In Subsection \ref{subsec:proof of cor(supersolution)} we derive the inequalities \eqref{2.7} and \eqref{2.8} from the \eqref{2.5}. Finally in Subsection \ref{subsec:limitations(supersolution)} we comment a bit about the limitation of the method: proving the results mentioned in the Remarks \ref{rem2.1} and \ref{rem2.2}. 

\subsection{Main Results}\label{subsec:main results(supersolution)}
The first main result is the following two-parameter family of discrete weighted Hardy inequalities:
\begin{theorem}\label{thm2.1}
If $\alpha, \beta \in \mathbb{R}$, then
\begin{equation}\label{2.12}
    \sum_{n=1}^\infty |u(n)-u(n-1)|^2 n^\alpha \geq \sum_{n=1}^\infty w_{\alpha,\beta}(n) |u(n)|^2,
\end{equation}
for $u \in C_c(\mathbb{N}_0)$ and $u(0)=0$, \\ 
where 
\begin{equation}\label{2.13}
    w_{\alpha, \beta}(n) := n^\alpha \Bigg[ 1 + \Big(1+\frac{1}{n}\Big)^\alpha - \Big(1-\frac{1}{n}\Big)^\beta - \Big(1+\frac{1}{n}\Big)^{\alpha+\beta}\Bigg],
\end{equation}
for $n \geq 2$ and $w_{\alpha, \beta}(1) := 1+2^\alpha - 2^{\alpha+\beta}$.
\end{theorem}

\begin{remark}
We would like to mention that inequality \eqref{2.12} is a generalization of improved Hardy inequality in \cite{keller4}. We recover the inequality in \cite{keller4}, by taking $\alpha=0$ and $\beta =1/2$ in inequality \eqref{2.12}. See \cite{} for an explicit remainder term in \eqref{2.12} corresponding to $\alpha=0$ and $\beta = 1/2$.
\end{remark}

As a special case of Theorem \ref{thm2.1}, we obtain the following power weight discrete Hardy inequality:

\begin{corollary}\label{cor2.1}
Let $\alpha \in [0,1) \cup [5,\infty)$. Then for all $u \in C_c(\mathbb{N}_0)$ with $u(0)=0$ we have \begin{equation}\label{2.14}
    \sum_{n=1}^\infty |u(n)-u(n-1)|^2 n^\alpha \geq \frac{(\alpha-1)^2}{4} \sum_{n=1}^\infty \frac{|u(n)|^2}{n^2} n^\alpha.
\end{equation}
Moreover the constant in \eqref{2.14} is sharp; that is, if we replace $(\alpha-1)^2/4$ with a strictly bigger constant then inequality \eqref{2.14} will not be true. 
\end{corollary}

\begin{remark}
Note that inequality \eqref{2.14} with $\alpha =0$ yields classical discrete Hardy inequality \eqref{2.1}. Recently, authors in \cite{huang} proved \eqref{2.14} for all non-negative numbers $\alpha$.
\end{remark}

Inequality \eqref{2.12} also yields the following improvement of \eqref{2.14} when $\alpha \in [1/3,1)\cup \{0\}$.
\begin{corollary}\label{cor2.2}
If $\alpha \in [1/3,1) \cup\{0\}$ then 
\begin{equation}\label{2.15}
    \sum_{n=1}^\infty |u(n)-u(n-1)|^2 n^\alpha \geq \frac{(\alpha-1)^2}{4} \sum_{n=1}^\infty \frac{|u(n)|^2}{n^2} n^\alpha + \sum_{k=3}^\infty b_k(\alpha) \sum_{n=2}^\infty \frac{|u(n)|^2}{n^k}n^\alpha,
\end{equation}
for all $u \in C_c(\mathbb{N}_0)$ with $u(0)=0$,

where the non-negative coefficients $b_k(\alpha)$ are given by
\begin{equation}\label{2.16}
    b_k(\alpha) := {\alpha \choose k} - (-1)^k {(1-\alpha)/2 \choose k} -{(1+\alpha)/2 \choose k},
\end{equation}
and ${\gamma \choose r}$ is the binomial coefficient for real parameters $\gamma$ and $r$.
\end{corollary}

\begin{remark}
Inequality \eqref{2.15} for $\alpha =0$ follows from the improved Hardy inequality proved in \cite{keller4}. 
\end{remark}

\subsection{Discrete Supersolution method}\label{subsec:supersolution method}

\begin{definition}
Let $\varphi$ be a real-valued function on $\mathbb{N}_0$. Then the combinatorial Laplacian $\Delta$ is defined as 
\begin{align*}
    \Delta \varphi(n):= 
    \begin{cases}
        \varphi(n) - \varphi(n-1) + \varphi(n)-\varphi(n+1) \hspace{19pt} \text{for } \hspace{5pt}  n \geq 1.\\
        \varphi(n)-\varphi(n+1) \hspace{117pt} \text{for } \hspace{5pt}  n=0.
    \end{cases}
\end{align*}
\end{definition}

\begin{lemma} \label{lem2.1} Let $v$ and $w$ be non-negative functions on $\mathbb{N}$. Assume $\exists$ function $\varphi : \mathbb{N}_0 \rightarrow [0, \infty)$ which is positive on $\mathbb{N}$ such that
\begin{equation}\label{2.17}
    \Big(\Delta \varphi(n) v(n) - (\varphi(n+1)-\varphi(n))(v(n+1)-v(n))\Big) \geq w(n)\varphi(n), 
\end{equation}
for all $n \in \mathbb{N}$. Then following inequality holds true
\begin{equation}\label{2.18}
    \sum_{n=1}^\infty |u(n)-u(n-1)|^2 v(n) \geq \sum_{n=1}^\infty w(n)|u(n)|^2,
\end{equation}
for $u \in C_c(\mathbb{N}_0)$ and $u(0)= 0$. \footnote{Lemma \ref{lem2.1} holds true in a general setting of infinite graphs. See Appendix \ref{appendix:B} for details.}
\end{lemma}

\begin{proof}
It can be easily seen that for $a \in \mathbb{R}$ and $t \geq 0$ we have
\begin{equation}\label{2.19}
    (a-t)^2 \geq (1-t)(a^2 -t).
\end{equation}
Let $\psi(n) := \frac{u(n)}{\varphi(n)} $ on $\mathbb{N}$ and $\psi(0):=0$. Assuming $\psi(m) \neq 0$ and applying \eqref{3.3} for $a = \psi(n)/\psi(m)$ and $t = \varphi(m)/\varphi(n)$ we get
\begin{equation}\label{2.20}
    |\varphi(n)\psi(n) - \varphi(m)\psi(m)|^2 \geq (\varphi(n)-\varphi(m))(\psi(n)^2 \varphi(n) - \psi(m)^2 \varphi(m)).
\end{equation}
Since $\varphi(n) \geq \varphi(n)-\varphi(m)$, the above inequality is true even when $\psi(m)=0$.
Using \eqref{2.20} and \eqref{2.17} we obtain
\begin{align*}
    \sum_{n=1}^\infty |u(n)-u(n-1)|^2v(n)&=
    \sum_{n=1}^\infty |\varphi(n)\psi(n)-\varphi(n-1)\psi(n-1)|^2 v(n)\\
    &\geq \sum_{n=1}^\infty \Big(\varphi(n) - \varphi(n-1)\Big)\Big(\psi(n)^2 \varphi(n) - \psi(n-1)^2 \varphi(n-1)\Big)v(n)\\
    & = \sum_{n=1}^\infty \Big(\frac{\Delta \varphi}{\varphi} v - \frac{(\varphi(n+1)-\varphi(n))(v(n+1)-v(n))}{\varphi}\Big)|u(n)|^2 \\
   & \geq \sum_{n=1}^\infty w(n) |u(n)|^2.
\end{align*}
\end{proof}
\begin{remark}
The result proved in Lemma \ref{lem2.1} holds true for general infinite graphs (see \cite{gupta1} for a proof using ground state transform and \cite{huang} for a proof using expansion of squares).
\end{remark}
Now we are ready to prove Theorem \ref{thm2.1}.
\begin{proof}[Proof of Theorem \ref{thm2.1}]
Let $v(n):= n^\alpha$ and $\varphi(n):= n^\beta $ on $\mathbb{N}$ and $\varphi(0):=0$ and $w_{\alpha, \beta}$ be as defined by \eqref{2.13}. It can be easily checked that the triplet $(v,\varphi, w)$ satisfies \eqref{2.17}. Now Theorem \ref{thm2.1} directly follows from the Lemma \ref{lem2.1}.
\end{proof}

In the next Subsection we would be concerned with finding the parameters $\alpha$ and $\beta$ for which the weight $w_{\alpha, \beta}$ can be estimated from below by $\frac{(\alpha-1)^2}{4} n^{\alpha-2}$.

\subsection{Proof of Corollaries \ref{cor2.1} and \ref{cor2.2}}\label{subsec:proof of cor(supersolution)}

The goal is to find parameters $\alpha$ and $\beta$ for which $w_{\alpha, \beta}(n) \geq \frac{(\alpha-1)^2}{4}n^{\alpha-2}$. With this in mind, we introduce the function $g_{\alpha, \beta}(x):= 1+(1+x)^\alpha - (1-x)^\beta - (1+x)^{\alpha+\beta}$. The goal now becomes to find parameters $\alpha$ and $\beta$ for which  
\begin{align*}
    g_{\alpha,\beta}(x) \geq \frac{(\alpha-1)^2}{4}x^2
\end{align*}
for $0 < x \leq 1/2$ and $w_{\alpha,\beta}(1)= 1+2^\alpha -2^{\alpha+\beta} \geq (\alpha-1)^2/4$.

Recall that, for $x \in (0,1)$, the Taylor series gives
\begin{equation}\label{2.21}
    (1\pm x)^r = \sum_{k=0}^\infty {r \choose k} (\pm1)^k x^k.
\end{equation}

Using \eqref{2.21}, we get the following expansion of $g_{\alpha,\beta}(x)$
\begin{equation}\label{2.22}
    g_{\alpha, \beta}(x) = \sum_{k=2}^\infty \Bigg[{\alpha \choose k} -(-1)^k {\beta \choose k} - {\alpha+\beta \choose k}\Bigg]x^k.
\end{equation}

Observe that the coefficient of $x^2$ is maximized when $ \beta = (1-\alpha)/2$. Taking $\beta = (1-\alpha)/2$, 
\begin{equation}\label{2.23}
    g(x) := g_{\alpha, \beta}(x) = \frac{(\alpha-1)^2}{4} x^2 + \sum_{k=3}^\infty \Bigg[{\alpha \choose k} -(-1)^k {(1-\alpha)/2 \choose k} - {(1+\alpha)/2 \choose k}\Bigg]x^k.
\end{equation}

In the next Lemma, we prove that the coefficients of $x^k$ in \eqref{2.23} are non-negative for $\alpha \in [1/3,1) \cup\{0\}$, which will be used as an ingredient in the proof of Corollary \ref{cor2.2}.
\begin{lemma}\label{lem2.2} 
Let $b_k(\alpha)$ be defined as 
\begin{align*}
    b_k(\alpha) := {\alpha \choose k} -(-1)^k {(1-\alpha)/2 \choose k} - {(1+\alpha)/2 \choose k}.
\end{align*}
Then $b_k(\alpha) \geq 0$ for $\alpha \in [1/3,1) \cup\{0\}$ and $k \geq 3$.
\end{lemma}

\begin{proof}
We first cover the case $\alpha =0$. For $ k\geq 3$ we have
\begin{align*}
    b_k(0) = -(-1)^k {1/2 \choose k} - {1/2 \choose k} = -{1/2 \choose k}(1+ (-1)^k).
\end{align*}
Clearly, for odd $k$, $b_k(0) = 0$ and for even $k$ we have $b_k(0) = -2{1/2 \choose k}$, which is non-negative. This proves the non-negativity of $b_k(0)$ for $k \geq 3$.

Next we assume that $\alpha \in [1/3,1)$. Let $\alpha_1 := (1-\alpha)/2$ and $\alpha_2 := (1+\alpha)/2$. Then
\begin{align*}
    b_k(\alpha) &= {\alpha \choose k} - (-1)^k{\alpha_1 \choose k} -{\alpha_2 \choose k}\\
    &= (-1)^{k-1}\frac{\alpha(1-\alpha)...(k-1-\alpha)}{k!} + \frac{\alpha_1(1-\alpha_1)...(k-1 -\alpha_1)}{k!} \\
    &+ (-1)^k\frac{\alpha_2(1-\alpha_2)....(k-1-\alpha_2)}{k!} .
\end{align*}
We treat the case of odd and even $k$ separately. First consider the case when $k$ is odd.
\begin{align*}
    b_k(\alpha) &= {\alpha \choose k} + \frac{\alpha_1(1-\alpha_2)..(k-1-\alpha_2)}{k!}\Bigg[\prod_{i=1}^{k-1} \frac{(i-\alpha_1)}{(i-\alpha_2)} - \frac{\alpha_2}{\alpha_1}\Bigg] \\
    &= {\alpha \choose k} + \frac{\alpha_1}{\alpha_2}{\alpha_2 \choose k}\Bigg[\prod_{i=1}^{k-1} \frac{(i-\alpha_1)}{(i-\alpha_2)} - \frac{\alpha_2}{\alpha_1}\Bigg].
\end{align*}
Note that for $i \geq 1$ we have $\frac{i-\alpha_1}{i-\alpha_2} = \frac{2i-1+\alpha}{2i-1-\alpha} \geq 1$. Therefore we have
\begin{align*}
    \prod_{i=1}^{k-1}\frac{(i-\alpha_1)}{(i-\alpha_2)} - \frac{\alpha_2}{\alpha_1} = \Big(\prod_{i=2}^{k-1}\frac{(i-\alpha_1)}{(i-\alpha_2)} - 1\Big) \frac{\alpha_2}{\alpha_1} \geq 0.
\end{align*}
The above inequality along with non-negativity of ${\alpha \choose k}, {\alpha_2 \choose k}$ for odd $k$ proves that, $b_k(\alpha) \geq 0$ for odd $k \geq 3$. \newpage
Next we consider the case when $k$ is even.  
\begin{equation}\label{2.24}
    \begin{split}
        b_k(\alpha) &= - \frac{\alpha(1-\alpha)...(k-1-\alpha)}{k!} + \frac{\alpha_1(1-\alpha_1)...(k-1 -\alpha_1)}{k!} - {\alpha_2 \choose k}\\
        &= \frac{\alpha_1(1-\alpha)...(k-1-\alpha)}{k!}\Big(\prod_{i=1}^{k-1} \frac{i-\alpha_1}{i-\alpha} - \frac{\alpha}{\alpha_1}   \Big) -  {\alpha_2 \choose k}.
    \end{split}
\end{equation}
Consider the following polynomial in $\alpha$:
\begin{align*}
    P(\alpha) &:= \prod_{i=1}^{7} \frac{i-\alpha_1}{i-\alpha} - \frac{\alpha}{\alpha_1}\\
    &= \prod_{i=1}^{7} \frac{2i-1+\alpha}{2(i-\alpha)} - \frac{2\alpha}{1-\alpha} = \frac{1}{\prod_{i=1}^{7}2(i-\alpha)}Q(\alpha).
\end{align*}
where 
\begin{equation}\label{2.25}
    Q(\alpha):= \prod_{i=1}^7 (2i-1+\alpha) - 2^8 \alpha \prod_{i=2}^7(i-\alpha).
\end{equation}
Next we show that $Q(\alpha)$ is non-negative for $\alpha \in [1/3,1)$. Note that showing $Q(\alpha) \geq 0$ is equivalent to showing 
\begin{equation}\label{2.26}
    \log(\prod_{i=1}^7 (2i-1+\alpha)) \geq \log(2^8 \alpha \prod_{i=2}^7(i-\alpha)).
\end{equation}
We introduce
\begin{align*}
    R(\alpha) &:= \log(\prod_{i=1}^7 (2i-1+\alpha)) - \log(2^8 \alpha \prod_{i=2}^7(i-\alpha))\\
    &= \sum_{i=1}^7 \log(2i-1+\alpha) - \log(2^8) - \log(\alpha) -\sum_{i=2}^7 \log(i-\alpha).
\end{align*}
It is straightforward to check that $R''(\alpha) \geq 0$ whenever $1/3 \leq \alpha \leq 1$. This, along with the fact that $R'(1/3)$ is non-negative, implies that $R'(\alpha) \geq 0$ in the specified domain. This means that the function $R(\alpha)$ is non-decreasing in the interval $(1/3, 1)$. Since $R(1/3) =0$, we can conclude that $R(\alpha) \geq 0$ in the interval $(1/3,1)$. Therefore we have $Q(\alpha) \geq 0$ which further implies that $P(\alpha)$ is non-negative in the interval $[1/3,1)$.

Also note that $\frac{i-\alpha_1}{i-\alpha} \geq 1$ for $1/3 \leq \alpha \leq 1$.  Using this fact along with the non-negativity of $P(\alpha)$ in \eqref{2.24} we get 
\begin{equation}\label{2.27}
    b_k(\alpha) \geq 0
\end{equation}
for even $k\geq 8$ and $1/3 \leq \alpha <1$. 

Now it remains to show that $b_4(\alpha)$ and $b_6(\alpha)$ are non-negative. Doing standard computations, we find that
\begin{equation}\label{2.28}
    b_4(\alpha) = \frac{1}{192}(5-\alpha)(1-\alpha)(7\alpha^2 - 6\alpha +3),
\end{equation}
and 
\begin{equation}\label{2.29}
    b_6(\alpha) = \frac{1}{23040}(1-\alpha)(9-\alpha)(31 \alpha^4 - 170 \alpha^3 + 536\alpha^2 - 310 \alpha + 105).
\end{equation}
It is very easy to see that $b_4(\alpha)$ is non-negative for $ 0 \leq \alpha < 1.$ Consider 
\begin{align*}
    T(\alpha) &:= 31 \alpha^4 - 170 \alpha^3 + 536\alpha^2 - 310 \alpha + 105.
\end{align*}
Let $\alpha^* := 7/20$. It can be easily verified that $T''(\alpha) \geq 0$ and both $T'(\alpha^*), T(\alpha^*)$ are non-negative. This implies the non-negativity of $T(\alpha)$ for $\alpha \in [\alpha^*,1)$. 

Now assume $\alpha \in [0,\alpha^*]$. Using arithmetic-geometric mean inequality we get
\begin{align*}
    31\alpha^4 + 536\alpha^2 \geq 2\sqrt{16616}\alpha^3.
\end{align*}
Now showing $T(\alpha)$ is non-negative boils down to showing $\Tilde{T}(\alpha):=2\sqrt{16616}\alpha^3 -170 \alpha^3 -310 \alpha +105 \geq 0$. Observing that $\Tilde{T}'(\alpha) \leq 0$ for $\alpha \in (0,1)$ and $\Tilde{T}(\alpha^*) \geq 0$ proves the non-negativity of $\Tilde{T}(\alpha)$ in the interval $[0,\alpha^*]$. This proves the non-negativity of $T$ and hence the non-negativity of $b_6(\alpha)$ in the interval $ \alpha \in [0,1)$.   
\end{proof}
Next we prove that $g(x) \geq \frac{(\alpha-1)^2}{4}x^2$ for $\alpha \in [0,1) \cup [5, \infty)$. We treat the cases $\alpha \in [0,1)$ and when $\alpha \in [5,\infty)$ separately.
\begin{lemma}\label{lem2.3}
Let $\alpha \in[0,1/3]$. Then for $0< x <1$ we have
\begin{equation}\label{2.30}
    g(x) \geq \frac{(\alpha-1)^2}{4} x^2.
\end{equation}
\end{lemma}

\begin{proof}
Let $E(x):= g(x) - \frac{(\alpha-1)^2}{4}x^2 =  1+(1+x)^\alpha - (1-x)^{(1-\alpha)/2} - (1+x)^{(1+\alpha)/2} -\frac{(\alpha-1)^2}{4}x^2$.
The first four derivatives of $E$ are given by 
\begin{align*}
    E'(x) &= \alpha(1+x)^{\alpha-1} + \frac{1-\alpha}{2}(1-x)^{\frac{-1-\alpha}{2}} - \frac{(1+\alpha)}{2}(1+x)^{\frac{\alpha-1}{2}} - \frac{(\alpha-1)^2}{2}x.\\
    E''(x) &= \alpha(\alpha-1)(1+x)^{\alpha-2} + \frac{(1+\alpha)(1-\alpha)}{4}(1-x)^{\frac{-3-\alpha}{2}} + \frac{(1+\alpha)(1-\alpha)}{4}(1+x)^{\frac{\alpha-3}{2}} \\
    &- \frac{(\alpha-1)^2}{2}.
\end{align*}
\begin{align*}
    E'''(x) &= \alpha(\alpha-1)(\alpha-2)(1+x)^{\alpha-3} + \frac{(1+\alpha)(1-\alpha)(3+\alpha)}{8}(1-x)^{\frac{-5-\alpha}{2}}\\
    &+ \frac{(1+\alpha)(1-\alpha)(\alpha-3)}{8}(1+x)^{\frac{\alpha-5}{2}}.\\
    E''''(x) &= \alpha(\alpha-1)(\alpha-2)(\alpha-3)(1+x)^{\alpha-4} + \frac{(1+\alpha)(1-\alpha)(3+\alpha)(5+\alpha)}{16}(1-x)^{\frac{-7-\alpha}{2}}\\
    &+ \frac{(1+\alpha)(1-\alpha)(\alpha-3)(\alpha-5)}{16}(1+x)^{\frac{\alpha-7}{2}}. 
\end{align*}
Note that $E(0)=E'(0)=E''(0)=0$ and $E'''(0) = \frac{3}{4}\alpha(1-\alpha)(3-\alpha)$ which is non-negative. Further assuming that $E''''(x)$ is non-negative completes the proof. In what follows we prove that $E''''(x)$ is non-negative. 
Using arithmetic-geometric mean inequality we get
\begin{align*}
    2\Bigg( \frac{(1+\alpha)^2(9-\alpha^2)(25-\alpha^2)}{16^2}(1+x)^{\frac{\alpha-7}{2}}&(1-x)^{\frac{-7-\alpha}{2}}\Bigg)^{\frac{1}{2}} \\
    &\leq \frac{(1+\alpha)(3+\alpha)(5+\alpha)}{16}(1-x)^{\frac{-7-\alpha}{2}}\\
    &+\frac{(1+\alpha)(3-\alpha)(5-\alpha)}{16}(1+x)^{\frac{\alpha-7}{2}}.
\end{align*}
Therefore proving $E''''(x)\geq 0$ reduces to showing 
\begin{equation}\label{2.31}
    2\Bigg(\frac{(1+\alpha)^2(9-\alpha^2)(25-\alpha^2)}{16^2}(1+x)^{\frac{\alpha-7}{2}}(1-x)^{\frac{-7-\alpha}{2}}\Bigg)^{\frac{1}{2}} \geq \alpha(2-\alpha)(3-\alpha)(1+x)^{\alpha-4},
\end{equation}
which is equivalent to proving
\begin{align*}
    &\log2 + 1/2 \log\Bigg(\frac{(1+\alpha)^2(9-\alpha^2)(25-\alpha^2)}{16^2}(1+x)^{\frac{\alpha-7}{2}}(1-x)^{\frac{-7-\alpha}{2}}\Bigg) \\
    &\geq \log\Big(\alpha(2-\alpha)(3-\alpha)(1+x)^{\alpha-4}\Big) .
\end{align*}
Consider the function 
\begin{align*}
    f(x) &:= \log2 + 1/2 \log\Bigg(\frac{(1+\alpha)^2(9-\alpha^2)(25-\alpha^2)}{16^2}(1+x)^{\frac{\alpha-7}{2}}(1-x)^{\frac{-7-\alpha}{2}}\Bigg)\\ &-\log\Big(\alpha(2-\alpha)(3-\alpha)(1+x)^{\alpha-4}\Big) \\
    &= \log2 + 1/2\log\Bigg(\frac{(1+\alpha)^2(9-\alpha^2)(25-\alpha^2)}{16^2}\Bigg) - \log\Big(\alpha(2-\alpha)(3-\alpha)\Big)\\
    &+ \frac{3}{4}(3-\alpha)\log(1+x) - \frac{7+\alpha}{4}\log(1-x).
\end{align*}
It can be easily checked that $f'(x)\geq 0$. Now we will show that $f(0)$ is non-negative for $\alpha \in (0,1/3]$. Consider \newpage
\begin{align*}
    2f(0)&= \log4  + \log\Bigg(\frac{(1+\alpha)^2(9-\alpha^2)(25-\alpha^2)}{16^2}\Bigg) - 2\log\Big(\alpha(2-\alpha)(3-\alpha)\Big)\\
    &= \log4 + \log\Bigg(\frac{(1+\alpha)^2(9-\alpha^2)(25-\alpha^2)}{16^2\alpha^2(2-\alpha)^2(3-\alpha)^2}\Bigg)\\
    &=\log4 + \log\Bigg(\frac{(1+\alpha)^2(3+\alpha)(25-\alpha^2)}{16^2\alpha^2(2-\alpha)^2(3-\alpha)}\Bigg).
\end{align*}
So $f(0)$ is non-negative iff 
\begin{equation}\label{2.32}
    \frac{(1+\alpha)^2(3+\alpha)(25-\alpha^2)}{16^2\alpha^2(2-\alpha)^2(3-\alpha)} \geq 1/4.
\end{equation}
Consider the function
\begin{align*}
    Q(\alpha) := (1+\alpha)^2(25-\alpha^2) - 64\alpha^2(2-\alpha)^2.
\end{align*}
It is straightforward to check that $Q''(\alpha)$ is negative in the interval $(0,1/3)$ and $Q'(0), Q(0)$ and  $Q(1/3)$ are non-negative. From this information one can easily conclude that $Q(\alpha) \geq 0$ in the interval $(0,1/3]$. Now consider
\begin{align*}
    \frac{(1+\alpha)^2(3+\alpha)(25-\alpha^2)}{16^2\alpha^2(2-\alpha)^2(3-\alpha)} \geq \frac{(1+\alpha)^2(25-\alpha^2)}{16^2\alpha^2(2-\alpha)^2} \geq 1/4.
\end{align*}
The last steps follows from the non-negativity of $Q(\alpha)$. This proves that $f(0)$ is non-negative whenever $\alpha \in (0,1/3]$. This fact, along with the non-negativity of $f'(x)$, implies $f(x) \geq 0$, which further implies $E''''(x) \geq 0$.
\end{proof}

\begin{remark}\label{rem2.7}
Using numerics, one can easily conclude that \eqref{2.32} is true for $\alpha \in (0,1)$. Therefore Lemma \ref{lem2.3} is true for $\alpha \in (0,1)$, i.e, $g(x) \geq \frac{(\alpha-1)^2}{4}x^2$ whenever $\alpha \in (0,1)$. But proving \eqref{2.32} in the interval $(0,1)$ mathematically becomes a bit tricky.
\end{remark}

\begin{remark}
Lemma \ref{lem2.2} along with Lemma \ref{lem2.3} proves that $g(x) \geq \frac{(\alpha-1)^2}{4}x^2$ for $x \in [0,1)$ and $\alpha \in  [0,1)$.
\end{remark}
Next we prove that $g(x) \geq \frac{(\alpha-1)^2}{4}x^2$ when $\alpha \geq 5$.
\begin{lemma}\label{lem2.4}
Let $\alpha \geq 5$. Then 
\begin{equation}\label{4.13}
    g(x) \geq \frac{(\alpha-1)^2}{4} x^2,
\end{equation}
for $0 < x \leq 1/2$.
\end{lemma}

\begin{proof}
Consider
\begin{align*}
    E(\alpha,x):=  1+(1+x)^{2\alpha+1}- (1-x)^{-\alpha}-(1+x)^{\alpha+1}-\alpha^2 x^2.
\end{align*}
Note that, under the transformation $\alpha \mapsto 2\alpha+1$, showing \eqref{4.13} reduces to proving $E(\alpha,x) \geq 0$ for $\alpha \geq 2$. The first three derivatives of $E$ w.r.t $\alpha$ are given by
\begin{align*}
    \partial_\alpha E(\alpha, x) &= 2(1+x)^{2\alpha+1}\log(1+x) +(1-x)^{-\alpha}\log(1-x) \\
    &- (1+x)^{\alpha+1}\log(1+x) -2\alpha x^2.\\
    \partial^2_{\alpha^2}E(\alpha, x) &= 4(1+x)^{2\alpha+1}\log^2(1+x) - (1-x)^{-\alpha}\log^2(1-x) \\
    &- (1+x)^{\alpha+1}\log^2(1+x) - 2x^2.\\
    \partial^3_{\alpha^3}E(\alpha,x) &= 8(1+x)^{2\alpha+1}\log^3(1+x) + (1-x)^{-\alpha}\log^3(1-x) - (1+x)^{\alpha+1}\log^3(1+x).
\end{align*}
The strategy of the proof is to show that $\partial^3_{\alpha^3}E(\alpha,x), \partial^2_{\alpha^2}E(2, x), \partial_\alpha E(2, x)$ and $E(2,x)$ are all non-negative, thereby completing the proof. Consider
\begin{align*}
    \partial&^3_{\alpha^3}E(\alpha,x) \\
    &= 8(1+x)^{2\alpha+1}\log^3(1+x) + (1-x)^{-\alpha}\log^3(1-x) - (1+x)^{\alpha+1}\log^3(1+x) \\
    &= [8(1+x)^{2\alpha+1}-(1+x)^{\alpha+1}]\log^3(1+x) + (1-x)^{-\alpha}\log^3(1-x)\\
    &= -(1-x)^{-\alpha}\log^3(1-x)\Big[(1+x)(1-x)^\alpha(8(1+x)^{2\alpha}-(1+x)^\alpha)\frac{\log^3(1+x)}{-\log^3(1-x)} - 1\Big]\\
    &= -(1-x)^{-\alpha}\log^3(1-x)\Big[(1+x)[8((1+x)^2(1-x))^\alpha-(1-x^2)^\alpha]\frac{\log^3(1+x)}{-\log^3(1-x)} - 1\Big]\\ 
    & \geq -(1-x)^{-\alpha}\log^3(1-x)\Big[(1+x)[8((1+x)^2(1-x))^2-(1-x^2)^2]\frac{\log^3(1+x)}{-\log^3(1-x)} - 1\Big]\\
    & = -(1-x)^{-\alpha}\log^3(1-x)\Big[(1+x)^3(1-x)^2[8(1+x)^2-1]\frac{\log^3(1+x)}{-\log^3(1-x)} - 1\Big]\\
    & \geq -(1-x)^{-\alpha}\log^3(1-x)\Big[7/4(1+x)^3\frac{\log^3(1+x)}{-\log^3(1-x)} - 1\Big].
\end{align*}
Therefore, for $\alpha \geq 2$, we have
\begin{equation}\label{2.34}
    \partial^3_{\alpha^3}E(\alpha,x)  \geq -(1-x)^{-\alpha}\log^3(1-x)\Big[7/4(1+x)^3\frac{\log^3(1+x)}{-\log^3(1-x)} - 1\Big].
\end{equation}
Next we prove the following inequalities for $0<x \leq 1/2$:
\begin{align*}
    & 7/4(1+x)^3 \log^3(1+x) + \log^3(1-x) \geq 0.\\
    &\partial^2_{\alpha^2}E(2, x) = 4(1+x)^5\log^2(1+x) - (1-x)^{-2} \log^2(1-x)\\
    & \hspace{63pt}- (1+x)^3\log^2(1+x) - 2x^2 \geq 0.
\end{align*}
\begin{align*}
    &\partial_\alpha E(2, x) = 2(1+x)^5\log(1+x) + (1+x)^{-2}\log(1-x) - (1+x)^3\log(1+x) - 4x^2 \geq 0. \\
    &E(2,x) = 1+(1+x)^5 - (1-x)^{-2}-(1+x)^3 - 4x^2 \geq 0.    
\end{align*}
Assuming the above inequalities are true, the result follows.
Standard computations yield
\begin{align*}
    E_1(x) &:= \partial_\alpha E(2, x) =2(1+x)^5\log(1+x) + (1-x)^{-2}\log(1-x) \\
    &- (1+x)^3\log(1+x) - 4x^2.\\
    E_1^{(5)}(x) &= 240\log(1+x) + \frac{6}{(1+x)^2} + 548 - \frac{1044}{(1-x)^7} + 720 \frac{\log(1-x)}{(1-x)^7}\\
    & \leq 240\log(3/2) + 6 + 548 - 1044 \leq 0.
\end{align*}
It can be easily checked that $E_1^{(i)}(0) \geq 0$ for $i\leq 4$ and $E_1(1/2) \geq 0$. This proves that $E_1^{(i)}(x)$ for $1 \leq i \leq 4$ is either non-negative or it has one zero say $y$, such that $E_1^{(i)}(x) \geq 0$ for $ x \leq y$ and $E_1^{(i)}(x) \leq 0$ for $ x \geq y$. Let us assume that $E_1^{(1)}(x)$ is non-negative, this implies that $E_1(x)$ is a non-decreasing function of $x$. This combined with the fact that $E_1(0) = 0$ proves that $E_1(x) \geq 0$. Another possibility is that $E_1^{(1)}(x)$ has one zero $y$. Then $E_1(x)$ is a non-deceasing function in $[0,y]$ and it is non-increasing in $[y, 1/2]$. This combined with non-negativity of $E_1(1/2)$ proves that $E_1(x)\geq 0$ in the interval $(0,1/2]$. Now consider the second derivative 
\begin{align*}
    E_2(x) &:= \partial^2_{\alpha^2}E(2, x) \\
    &= 4(1+x)^5\log^2(1+x) - (1-x)^{-2} \log^2(1-x) - (1+x)^3\log^2(1+x) - 2x^2.\\
    E_2^{(6)}(x) &= \frac{3(40x^2 + 80x + 39)}{(1+x)^3}\log(1+x) + \frac{274}{1+x} - \frac{1}{(1+x)^3} -\frac{1276}{(1-x)^8}\\ &+\frac{9(223-70\log(1-x))}{(1-x)^8}\log(1-x) \\
    & \leq 267\log(3/2) + 274 - 8/27 - 1276\leq 0.
\end{align*}
Simple calculations yield $E_2^{(i)}(0) \geq 0$ for $i \leq 5$ and $E_2(1/2) \geq 0$. This proves that $E_2(x) \geq 0$ for $x \in (0,1/2]$, via the same logic used in proving that $E_1(x)$ is non-negative. Next, we consider the third derivative
\begin{align*}
    E_3(x) &:= 7/4(1+x)^3 \log^3(1+x) + \log^3(1-x).\\
    E_3^{(5)}(x) & = - \frac{210}{(1-x)^5} + \frac{105}{2(1+x)^2} + 300 \frac{\log(1-x)}{(1-x)^5} - 72 \frac{\log^2(1-x)}{(1-x)^5} - \frac{105}{2} \frac{\log(1+x)}{(1+x)^2}\\
    &- \frac{63}{2} \frac{\log^2(1+x)}{(1+x)^2}
     \leq -210 + 105/2 \leq 0.
\end{align*}
Furthermore, $E_3^{(i)}(0) \geq 0$ for $i \leq 4$ and $E_3(1/2)\geq 0$. This proves the non-negativity of $E_3(x)$. 

Finally, we consider $E(2,x)$
\begin{align*}
    E_0(x) &:= E(2,x) = 1+(1+x)^5 - (1-x)^{-2}-(1+x)^3 - 4x^2.\\
    E_0^{(5)}(x) &= 120 -\frac{720}{(1-x)^7} \leq 0.
\end{align*}
It can be verified that $E_0^{(i)}(0) \geq 0$ for $i \leq 4$ and $E_0(1/2) \geq 0$. This implies that $E_0(x) \geq 0$ in the interval $(0,1/2]$. 
\end{proof}

\begin{remark}
Using Lemmas \ref{lem2.2}, \ref{lem2.3} and \ref{lem2.4} we can conclude that $g(x) \geq \frac{(\alpha-1)^2}{4}x^2$ for $0 < x \leq 1/2$ and $\alpha \in [0,1) \cup [5,\infty)$. This proves that, with the choice $\beta = (1-\alpha)/2$, we have $w_{\alpha,\beta}(n) \geq \frac{(\alpha-1)^2}{4}n^{\alpha-2}$ for $n \geq 2$ and $\alpha \in [0,1) \cup [5,\infty)$. Now it remains to show that $w_{\alpha, (1-\alpha)/2}(1) \geq (\alpha-1)^2/4$. We prove this in the next Lemma.
\end{remark}

\begin{lemma}\label{lem2.5}
Let $w_{\alpha, \beta}$ be the weight function as defined by \eqref{2.13}. Then for $\beta= (1-\alpha)/2$ and  $\alpha \in [0,1) \cup [5, \infty)$ we have
\begin{equation}\label{2.35}
    w_{\alpha,\beta}(1) = 1+2^\alpha - 2^{(\alpha+\beta)/2}=1+2^\alpha -2^{(1+\alpha)/2} \geq \frac{(\alpha-1)^2}{4}.
\end{equation}
\end{lemma}

\begin{proof}
We will consider the case, when $\alpha \in [0,1)$ and $\alpha \geq 5$ separately. First assume $\alpha \in [0,1)$. Using mean value theorem for the function $2^x$, we get, for $\xi \in [\alpha, (1+\alpha)/2]$, 
\begin{align*}
    2^{(1+\alpha)/2} - 2^\alpha = \frac{(1-\alpha)}{2} 2^\xi \log2 \leq \frac{(1-\alpha)}{2} 2^{(1+\alpha)/2}\log2.
\end{align*}
This implies that 
\begin{align*}
    w_{\alpha, (1-\alpha)/2}(1)  - \frac{(\alpha-1)^2}{4} \geq 1 - 2^{(1+\alpha)/2}\log2\frac{(1-\alpha)}{2} - \frac{(\alpha-1)^2}{4} =: g(\alpha).
\end{align*}
Derivatives of $g$ are given by
\begin{align*}
    g'(\alpha) &= 1/2[2^{(\alpha+1)/2}\log2 - 2^{(\alpha+1)/2}\frac{(1-\alpha)}{2}\log^22 - \alpha+1].\\
    g''(\alpha) &= 1/4[2^{(\alpha+3)/2}\log^22 - 2^{(\alpha+1)/2}\frac{1-\alpha}{2}\log^32 -2].\\
    g'''(\alpha) &= \frac{2^{(1+\alpha)/2}}{8}\log^32[3-(1-\alpha)/2 \log2] \geq 0.  
\end{align*}
Note that $g''(1) = \log^2(2)-1/2 \leq 0$, $g'(1) = \log2 \geq 0$ and $g(0) = (3-2\sqrt{2}\log(2))/4 \geq 0$. From this we can conclude that $w_{\alpha,(1-\alpha)/2}(1) \geq \frac{(\alpha-1)^2}{4}$ for $\alpha \in [0,1)$.

Now let $\alpha \geq 5$ case. Let $h(\alpha):= 1+ 2^{\alpha} - 2^{(1+\alpha)/2} - \frac{(\alpha-1)^2}{4}$. Derivatives of $h$ are given by
\begin{align*}
    h'(\alpha) &= 2^\alpha \log2 - \frac{2^{(1+\alpha)/2}}{2}\log2 - \frac{(\alpha-1)}{2}.\\
    h''(\alpha) &= 2^\alpha \log^22 - \frac{2^{(1+\alpha)/2}}{4}\log^22 - 1/2.\\
    h'''(\alpha) &= 2^\alpha \log^32 - \frac{2^{(1+\alpha)/2}}{8}\log^32 = \log^32(2^\alpha -2^{(\alpha-5)/2}) \geq 0.
\end{align*}
Noting that $h''(5) = 30\log^32 -1/2 \geq 0$, $h'(5) = (28 \log2-2)\geq 0$ and $h(5) = 21 \geq 0$. This proves that $h(\alpha) \geq 0$ for $\alpha \geq 5$.   
\end{proof}
Now we have all the pieces required to prove the Corollaries \ref{cor2.1} and \ref{cor2.2}. 
\begin{proof}[Proof of Corollary \ref{cor2.1}] Using Lemma \ref{lem2.2}, Lemma \ref{lem2.3} and Lemma \ref{lem2.4} we can conclude that
\begin{equation}\label{2.36}
    g(x) = 1+(1+x)^\alpha -(1-x)^{(1-\alpha)/2}-(1+x)^{(1+\alpha)/2} \geq \frac{(\alpha-1)^2}{4}x^2,
\end{equation}
for $0 < x \leq 1/2$ and $\alpha \in [0,1) \cup [5, \infty)$. Now taking $x=1/n$, we get, for $n \geq 2$,
\begin{equation}\label{2.37}
    1 + \Big(1+\frac{1}{n}\Big)^\alpha - \Big(1-\frac{1}{n}\Big)^{(1-\alpha)/2} - \Big(1+\frac{1}{n}\Big)^{(1+\alpha)/2} \geq \frac{(\alpha-1)^2}{4}\frac{1}{n^2}.
\end{equation}
Using \eqref{2.37} along with Lemma \ref{lem2.5}, we conclude that, for $\beta =(1-\alpha)/2$,
\begin{equation}\label{2.38}
    w_{\alpha, \beta}(n) \geq \frac{(\alpha-1)^2}{4} n^{\alpha-2}
\end{equation}
for all $n \geq 1$. Inequality \eqref{2.38} along with Theorem \ref{thm2.1}(with $\beta = (1-\alpha)/2$) proves Corollary \ref{cor2.1}. Next we prove the sharpness of the constant in Corollary \ref{cor2.1}.

Let $C$ be a constant such that
\begin{equation}\label{2.39}
    \sum_{n=1}^\infty |u(n)-u(n-1)|^2 n^\alpha \geq C \sum_{n=1}^\infty |u(n)|^2 n^{\alpha-2},
\end{equation}
for all $u \in C_c(\mathbb{N}_0)$ and $u(0)=0$. Let $N \in \mathbb{N}$, $\beta \in \mathbb{R}$  and $\alpha\geq 0$ such that $2\beta +\alpha-2 <-1$, in particular, $\beta < 1/2$. Consider the following family of finitely supported functions on $\mathbb{N}_0$. 
\begin{align*}
u_{\beta,N}(n):=
\begin{cases}
n^\beta \hspace{73pt} &\text{for} \hspace{5pt}  1 \leq  n \leq N.\\
-N^{\beta-1} n + 2N^{\beta} \hspace{13pt}  &\text{for} \hspace{5pt}  N \leq n \leq 2N.\\
0 \hspace{83pt} &\text{for} \hspace{5pt}  n \geq 2N \hspace{5pt} \text{and} \hspace{5pt} n=0.
\end{cases} 
\end{align*}
Clearly we have
\begin{equation}\label{2.40}
    \sum_{n=1}^\infty|u_{\beta,N}(n)|^2n^{\alpha-2} \geq \sum_{n=1}^{N}n^{2\beta+\alpha-2},
\end{equation}
and
\begin{equation}\label{2.41}
    \begin{split}
        \sum_{n=1}^\infty|u_{\beta, N}(n)-u_{\beta,N}(n-1)|^2n^{\alpha} &= \sum_{n=2}^{N}(n^\beta - (n-1)^\beta)^2n^{\alpha} + \sum_{n=N+1}^{2N}N^{2\beta-2}n^{\alpha} + 1.
    \end{split}
\end{equation}
Using the fact that $\beta <1/2$, we get the following  basic estimates:\\
$(n^\beta - (n-1)^\beta)^2 \leq \beta^2(n-1)^{2\beta-2}$.\\
$\sum_{n=N+1}^{2N} n^\alpha \leq \int_{N+1}^{2N+1} x^\alpha dx = \frac{(2N+1)^{\alpha+1} - (N+1)^{\alpha+1}}{\alpha+1}$.

Using the above, in \eqref{2.41}, we get
\begin{equation}\label{2.42}
    \begin{split}
        \sum_{n=1}^\infty|u_{\beta, N}(n)-u_{\beta, N}(n-1)|^2 n^\alpha \leq \beta^2 &\sum_{n=2}^{N}(n-1)^{2\beta-2}n^\alpha \\
        &+ \frac{N^{2\beta+\alpha-1}}{\alpha+1}\Bigg[\Big(2+\frac{1}{N}\Big)^{\alpha+1} - \Big(1+\frac{1}{N}\Big)^{\alpha+1}\Bigg]
        + 1.
    \end{split}
\end{equation}
Using estimates \eqref{2.40} and \eqref{2.42} in \eqref{2.39}, and taking limit $N \rightarrow \infty$, we get
\begin{equation}\label{2.43}
    C\sum_{n=1}^{\infty}n^{2\beta+\alpha-2} \leq \beta^2\sum_{n=2}^{\infty}(n-1)^{2\beta-2}n^\alpha + 1.
\end{equation}
Using Taylor's theorem for the function $x^\alpha$, we get, for $n \geq 2$, 
\begin{equation}\label{2.44}
    n^\alpha = (1+n-1)^\alpha \leq (n-1)^\alpha + {\alpha \choose 1}(n-1)^{\alpha-1} + .... + {\alpha \choose \lceil \alpha \rceil}(n-1)^{\alpha - \lceil \alpha \rceil}, 
\end{equation}
where $\lceil \alpha \rceil$ denotes the smallest integer greater than or equal to $\alpha$. 
Using \eqref{2.44} in \eqref{2.43}, we obtain
\begin{equation}\label{2.45}
    C\sum_{n=1}^{\infty}n^{2\beta+\alpha-2} 
    \leq \beta^2 \sum_{i=0}^{\lceil \alpha \rceil}{\alpha \choose i}\sum_{n=1}^\infty n^{2\beta + \alpha - i -2} + 1. 
\end{equation}
Finally, taking limit $\beta \rightarrow \frac{1-\alpha}{2}$, and observing that $\text{lim sup}_{\beta \rightarrow(1-\alpha)/2} \sum_{n=1}^ \infty n^{2\beta + \alpha -i-2}$ is finite for $i \geq 1$ and is infinite for $i=0$, we obtain
\begin{equation}\label{2.46}
    C \leq \frac{(\alpha-1)^2}{4}. 
\end{equation}
\end{proof}
\newpage
\begin{proof}[Proof of Corollary \ref{cor2.2}]
Let $g(x)$ be as defined by \eqref{2.23}, that is, $g(x) := 1 + (1+x)^\alpha - (1-x)^\beta - (1+x)^{\alpha+\beta}$ for $\beta = (1-\alpha/2)$. Using Taylor's expansion of $g(x)$ we get identity $\eqref{2.23}$ for $ x \in (0,1)$ 
\begin{equation}\label{2.47}
    g(x) = \frac{(\alpha-1)^2}{4}x^2 + \sum_{k=3}^\infty b_k(\alpha) x^k, 
\end{equation}
where 
\begin{align*}
    b_k(\alpha) := {\alpha \choose k} -(-1)^k {(1-\alpha)/2 \choose k} - {(1+\alpha)/2 \choose k}.
\end{align*}
Taking $x=1/n$ and multiplying both sides of \eqref{2.47} by a factor of $n^\alpha$, we obtain 
\begin{equation}\label{2.48}
    w_{\alpha, \beta}(n) = \frac{(\alpha-1)^2}{4}\frac{n^\alpha}{n^2} + \sum_{k=3}^\infty b_k(\alpha) \frac{n^\alpha}{n^k},
\end{equation}
for $\beta = (1-\alpha)/2$ and $n \geq 2$. Using \eqref{2.48} along with Lemma \ref{lem2.5} in Theorem \ref{thm2.1}(with $\beta = (1-\alpha)/2$) proves inequality \eqref{2.8} for $\alpha \in [0,1) \cup [5, \infty)$. Finally using Lemma \ref{lem2.2} to note the non-negativity of $b_k(\alpha)$ for $\alpha \in [1/3,1) \cup \{0\}$ we complete the proof of Corollary \ref{cor2.2}.
\end{proof}

\subsection{Limitations of the Method}\label{subsec:limitations(supersolution)}
In this section our first goal is to point out that the method described in this paper doesn't work for proving Corollary \ref{cor2.1} when $\alpha < 0$ or $\alpha \in (1,4)$. This will be proved in Lemmas \ref{lem2.6} and \ref{lem2.7}. Our second goal is to show that Corollary \ref{cor2.1} cannot be improved in the sense of Corollary \ref{cor2.2} when $\alpha$ doesn't lie in the interval $[1/3,1)$. This will be achieved partially via Lemma \ref{lem2.8}.

\begin{lemma}\label{lem2.6}
Let $\alpha < 0$. Then $\exists$ $\epsilon >0$(depending on $\alpha$) such that $g(x) < \frac{(\alpha-1)^2}{4}x^2$ for all $x \in (0,\epsilon)$.
\end{lemma}

\begin{proof}
Let $E(x):= g(x) - \frac{(\alpha-1)^2}{4}x^2$. Computations done in Lemma \ref{lem2.3} give $E(0) = E'(0) = E''(0) = 0$ and $E'''(0) = \frac{3}{4}\alpha(1-\alpha)(3-\alpha)$. Clearly $E'''(0) < 0$ for negative $\alpha$. The result now follows from the continuity of derivatives of $E(x)$.
\end{proof}

\begin{lemma}\label{lem2.7}
Let $\alpha \in (1,4)$ then $\exists$ $\epsilon >0$(depending on $\alpha$) such that $g(x) < \frac{(\alpha-1)^2}{4}$ for all $x \in (1/2-\epsilon, 1/2)$.
\end{lemma}

\begin{proof}
Let $E(x):= g(x) - \frac{(\alpha-1)^2}{4}x^2$. We show that $E(1/2)$ is negative whenever $\alpha \in (1,4)$. The result then follows from the continuity of the function $E(x)$.\\
Standard computations yield 
\begin{align*}
    f(\alpha) &:= E(1/2) = 1+(3/2)^\alpha - (1/2)^{(1-\alpha)/2} - (3/2)^{(1+\alpha)/2} - \frac{(\alpha-1)^2}{16}.\\
    f'(\alpha) &= (3/2)^\alpha \log(3/2) + \frac{1}{2}(1/2)^{(1-\alpha)/2}\log(1/2) - \frac{1}{2}(3/2)^{(1+\alpha)/2}\log(3/2) - \frac{\alpha-1}{8}.\\
    f''(\alpha) &= (3/2)^\alpha \log^2(3/2) - \frac{1}{4}(1/2)^{(1-\alpha)/2}\log^2(1/2)- \frac{1}{4}(3/2)^{(1+\alpha)/2}\log^2(3/2) -1/8.\\
    f'''(\alpha) &= (3/2)^\alpha \log^3(3/2) + \frac{1}{8} (1/2)^{(1-\alpha)/2}\log^3(1/2) - \frac{1}{8}(3/2)^{(1+\alpha)/2}\log^3(3/2)\\
    & = \frac{2^{(\alpha-1)/2}}{8}\log^3(2)\Bigg[\Big(8 \sqrt{2}\Big(\frac{3}{2\sqrt{2}}\Big)^\alpha - \sqrt{3}\Big(\frac{\sqrt{3}}{2}\Big)^\alpha\Big)\frac{\log^3(3/2)}{\log^3(2)}-1\Bigg]\\
    & \geq 2^{(\alpha-1)/2}\log^3(2)\Big(\frac{21}{2} \frac{\log^3(3/2)}{\log^3(2)}-1\Big)\geq 0.
\end{align*}
It can be easily seen that $f(4), f'(1), f''(1)$ are negative. Since $f''(1)$ is negative and $f'''(\alpha) \geq 0$, there are two possibilities. Firstly $f''(\alpha) \leq 0$. Then negativity of $f'(1)$ implies that $f'(\alpha)<0$, which further imply that $f(\alpha)$ is a strictly decreasing function. This along with $f(1) =0$ proves that $f(\alpha)<0$ for $\alpha \in (1,4)$.  Second possibility is that there exists $\beta \in (1,4)$ such that $f''(\alpha) < 0$ for $\alpha \in [1, \beta)$ and  $f''(\alpha) \geq 0$ for $\alpha \in [\beta, 4)$. Now we further have two possibilities, first $f'(\alpha) <0$, this along with $f(1)=0$ would imply that $f(\alpha)<0$ for $\alpha \in (1,4)$. Second possibility is that there exists $\gamma \in (1,4)$ such that $f'(\alpha)<0$ in $[1, \gamma)$ and $f'(\alpha) \geq 0$ in $[\gamma, 4)$. This along with $f(1) =0$ and $f(4) < 0$ implies that $f(\alpha) < 0$ for $\alpha \in (1,4)$.
\end{proof}

\begin{remark}\label{rem2.10}
Lemmas \ref{lem2.6} and \ref{lem2.7} say that the weights $w_{\alpha, \beta}$(with $\beta = (1-\alpha)/2$) obtained in Theorem $\ref{thm2.1}$ do not control the weight $\frac{(\alpha-1)^2}{4}n^{\alpha-2}$ whenever $\alpha < 0$ or $\alpha \in  (1,4)$. Therefore, one cannot obtain Corollary $\ref{cor2.1}$ from the Theorem \ref{thm2.1} when $\alpha < 0$ or $\alpha \in (1,4)$.
\end{remark}

\begin{remark}\label{rem5.4}
Using Theorem \ref{thm2.1}(with $\beta = (1-\alpha)/2$) and the Taylor expansion of $g$ \eqref{2.23}, and Lemma \eqref{lem2.5}, we conclude that \eqref{2.22} holds true for $\alpha \in [0,1) \cup [5,\infty)$. We conjecture that constants $b_k(\alpha)$ given by \eqref{2.9} are not non-negative for all $k$ when $\alpha$ does not lie in $[1/3,1)\cup \{0\}$, i.e, for every $\alpha \in (0,1/3) \cup [5, \infty)$, there exists $i \geq 1$ such that $b_i(\alpha) < 0$. Therefore, we do not have the improvement \eqref{2.22} of inequality \eqref{2.21} when $\alpha$ does not lie in $[1/3,1) \cup \{0\}$. In the next Lemma, we prove a result which supports the conjecture. 
\end{remark}

\begin{lemma}\label{lem2.8}
Let  $b_i(\alpha)$ be as defined by \eqref{2.16}. Let $\alpha = 2k+1$ then we have
\begin{align*}
    b_i(2k+1) = {2k+1 \choose i} - (-1)^i {-k \choose i} - {k+1 \choose i}.
\end{align*}
If $k \geq 2$, then  
\begin{align*}
    b_i(2k+1) &\geq 0 \hspace{19pt} \text{for} \hspace{5pt} 2 \leq i \leq k+1.\\
    b_i(2k+1) &< 0 \hspace{19pt} \text{for} \hspace{5pt} i > k+1.
\end{align*}
\end{lemma}

\begin{proof}
Clearly, ${2k+1 \choose i} = {k+1 \choose i} = 0$ for $i \geq 2k+2$. Therefore  $b_i(2k+1) < 0$ for $i \geq 2k+2$. 

Consider $ k+1 < i \leq 2k+1 $. In this case, we have
\begin{align*}
    b_i(2k+1) &= {2k+1 \choose i} - (-1)^i {-k \choose i}\\
    &= \frac{1}{i!}\Big((2k+1)2k(2k-1)..(2k+1-(i-1)) - k(k+1)(k+2)...(k+i-1) \Big)\\
    &= \frac{1}{i!}k(k+1)...(2k+1)\Big((k-1)..(2k+1-(i-1)) - (2k+2)...(k+i-1) \Big) \\
    &<  0.
\end{align*}
In the case when $2 \leq  i \leq k+1$, we have
\begin{align*}
    &b_i(2k+1)\\
    &= {2k+1 \choose i} - (-1)^i {-k \choose i} - {k+1 \choose i} \\
    & = \frac{1}{i!} \Big((2k+1)2k...(2k+1-(i-1)) - k(k+1)
    ..(k+i-1) - (k+1)k...(k+1-(i-1))\Big)\\
    &\geq \frac{1}{i!} \Big((2k+1)2k...(2k+1-(i-1)) - 2k(k+1)
    ..(k+i-1)\Big).
\end{align*}
Observing that $(2k-1)..(2k+1-(i-1)) \geq (k+2)...(k+i-1)$ for $k \geq 3$ and $i \leq k$, we get $b_i(k) \geq 0$. Now consider $i =k+1$. 
\begin{align*}
    b_i(2k+1) &= \frac{1}{i!} \Big((2k+1)2k..(k+1) - k(k+1)..(2k) - (k+1)!\Big)\\
    & = \frac{1}{i!}\Big((k+1)(k+1)(k+2)..2k - (k+1)k....1\Big)\\
    &\geq 0.
\end{align*}
The only case that remains is when $k=2$ and $i= 2$. It is straightforward that $b_2(5) = 4 \geq 0$.
\end{proof}
\newpage
\section{Fourier transform method}\label{sec: fourier method}
In this chapter, we use Fourier analytic tools to study discrete Hardy-type inequalities and its higher order versions on integers. As discussed in the introduction, one of the main hurdles in proving discrete inequalities is that we do not have `nice' calculus in the discrete setting. The \emph{Fourier transform method} bypasses this technical difficulty by converting the discrete inequality to some integral inequality by means of a \emph{Fourier transform}. Let us explain this with the help of an example. Consider the standard discrete Hardy inequality on integers: 
\begin{equation}\label{2.49}
    \sum_{n \in \Z} |u(n)-u(n-1)|^2 \geq 1/4 \sum_{n \in \Z\setminus\{0\}} |u(n)|^2 |n|^{-2}, 
\end{equation}
for $u \in C_c(\Z)$ with $u(0) =0 $.  

Let $u \in \ell^2(\Z)$, then its \emph{Fourier transform} $F(u) \in L^2(-\pi, \pi)$ is defined by
$$ F(u) := \frac{1}{\sqrt{2\pi}} \sum_{n \in \Z} u(n) e^{-inx}, \hspace{5pt} x \in (-\pi, \pi).$$
Since $\{e^{-inx}\}_{n \in \Z}$ forms an orthornormal basis of $L^2(-\pi, \pi)$ we have
\begin{align*}
    u(n) &= \frac{1}{\sqrt{2\pi}} \int_{-\pi}^\pi F(u) e^{inx} dx,\\
    \sum_{n \in \Z} |u(n)|^2 &= \int_{-\pi}^\pi |F(u)|^2 dx.
\end{align*}
The above two identities are referred to as \emph{Inversion formula} and \emph{Parseval's identity} respectively. Let $v(n) := u(n)/n$ for $n \neq 0$ and $v(0) := 0$. Then Parseval's identity gives us
\begin{equation}\label{2.50}
    \sum_{n \in \Z \setminus\{0\}} |u(n)|^2 |n|^{-2} = \sum_{n \in \Z} |v(n)|^2  = \int_{-\pi}^\pi |F(v)|^2 dx.
\end{equation}
On the other hand, using the inversion formula for $u(n)$ we get
$$ u(n) - u(n-1) = \frac{1}{\sqrt{2\pi}} \int_{-\pi}^\pi F(u)(1-e^{-ix})e^{inx} dx,$$
which implies that $F(u(n)-u(n-1)) = F(u)(1-e^{-ix})$. Another application of Parseval's identity and $F(v)'(x) = -i F(u)$ gives
\begin{equation}\label{2.51}
    \sum_{n \in \Z} |u(n)-u(n-1)|^2 = 4\int_{-\pi}^\pi |F(u)|^2 \sin^2(x/2) dx = 4\int_{-\pi}^\pi |F(v)'|^2 \sin^2(x/2) dx.
\end{equation}
Thus, proving \eqref{2.49} is equivalent to proving the following integral inequality (consequence of \eqref{2.50} and \eqref{2.51}):
\begin{equation}\label{2.52}
    \int_{-\pi}^\pi |g'(x)|^2 \sin^2(x/2) dx \geq 1/16 \int_{-\pi}^\pi |g(x)|^2 dx,
\end{equation}
for some suitable class of functions $g$. Inequality \eqref{2.52} is an improvement of a well-known Poincar\'e inequality on the interval. This will be proved later in Subsection \ref{subsec:auxiliary results(fourier transform)}.  

We exploit this method of connecting discrete Hardy-type inequalities with integral inequalities to prove several results in this Chapter. We begin with stating the main results in Subsection \ref{subsec: main results(fourier transform)}: weighted discrete Hardy-type inequalities. In Subsection \ref{subsec:auxiliary results(fourier transform)} we prove some auxiliary results (integral inequalities) using which we prove our main results in Subsection \ref{subsec:proofs of Hardy inequalities(foruier transform)} and \ref{Subsec: proof of higher order Hardy(foruier transform)}. Finally in Subsection \ref{subsec: combinatorial identity} we derive a new combinatorial identity using an integral identity proved in Subsection \ref{subsec:auxiliary results(fourier transform)}. The contents of this chapter are based on the work \cite{gupta2}.

\subsection{Main Results}\label{subsec: main results(fourier transform)}

\begin{theorem}[Improved weighted Hardy inequality]\label{thm2.2}
Let $u \in C_c(\mathbb{Z})$ and $u(0)=0$. Then for $k \geq 1$, we have
\begin{equation}\label{2.53}
    \sum_{n \in \mathbb{Z}}|u(n)-u(n-1)|^2 \Big(n-\frac{1}{2}\Big)^{2k} \geq \sum_{i=1}^k \gamma_i^k \sum_{n \in \mathbb{Z}} |u|^2 n^{2k-2i}+2^{-2k-2}\sum_{n \in \mathbb{Z}\setminus\{0\}} \frac{|u|^2}{n^2},
\end{equation}
where the non-negative constants $\gamma_i^k$ are given by
\begin{equation}\label{2.54}
    2^{2i}\gamma_i^k := 2{2k \choose 2i} - 2{k \choose i} + {k \choose i-1},
\end{equation}
and ${n \choose k}$ denotes the binomial coefficient.
\end{theorem}
Dropping the remainder terms in inequality \eqref{2.53} gives the following weighted Hardy inequalities: 
\begin{corollary}[Weighted Hardy inequality]\label{cor2.3}
Let $u \in C_c(\mathbb{Z})$ and $u(0)=0$. Then for $k \geq 1$, we have
\begin{equation}\label{2.55}
    \sum_{n \in \mathbb{Z}}|u(n)-u(n-1)|^2 \Big(n-\frac{1}{2}\Big)^{2k} \geq \frac{(2k-1)^2}{4} \sum_{n \in \mathbb{Z}} |u|^2 n^{2k-2}.
\end{equation}
Moreover, the constant $(2k-1)^2/4$ is sharp.
\end{corollary}
\begin{remark}
We would like to mention that inequality \eqref{2.55} was proved in paper \cite{liu} with the weight $(n-1/2)^\alpha$ and $\alpha \in (0,1)$. It is also worthwhile to notice that the above inequalities reduce to corresponding Hardy inequalities on $\mathbb{N}_0 :=\{0,1,2,..\}$ when we restrict ourselves to functions $u$ taking value zero on non-positive integers.   
\end{remark}
The method used in the proofs of above Hardy inequalities can also be extended to prove its higher-order versions. 

\begin{theorem}[Higher order Hardy inequalities]\label{thm2.3}
Let $m\in \mathbb{N}$. Then we have 
\begin{equation}\label{2.58}
      \sum_{n=0}^\infty |\Delta^m u(n)|^2 \geq \frac{1}{2^{4m}} \prod_{i=0}^{2m-1} (8m-3-4i) \sum_{n=1}^\infty \frac{|u(n)|^2}{n^{4m}},
\end{equation}
for all $u \in C_c(\mathbb{N}_0)$ with $u(i)=0$ for $0\leq i \leq 2m-1$,
and 
\begin{equation}\label{2.59}
    \sum_{n=1}^\infty |D(\Delta^m u)(n)|^2 \geq \frac{1}{2^{4m+2}} \prod_{i=0}^{2m} (8m+1-4i) \sum_{n=1}^\infty \frac{|u(n)|^2}{n^{4m+2}},
\end{equation}
for all $u \in C_c(\mathbb{N}_0)$ with $u(i)=0$ for $0\leq i \leq 2m$. Here $Du(n) := u(n)-u(n-1)$ denotes the first order difference operator and the second order difference operator ($\Delta$) on $\N_0$ is given by
\begin{align*}
    \Delta u(n):= 
\begin{cases}
    2u(n)-u(n-1)-u(n+1), \hspace{9pt} \text{if} \hspace{5pt} n \in \N \\
    u(0) - u(1), \hspace{95pt} \text{if} \hspace{5pt} n=0.
\end{cases}    
\end{align*}
\end{theorem}
Inequality \eqref{2.58} for $m=1$ gives a discrete analogue of well known \emph{Rellich inequality}.
\begin{corollary}[Rellich inequality]\label{cor2.6}
Let $u \in C_c(\mathbb{N}_0)$ and $u(0)=u(1)=0$. Then we have 
\begin{equation}\label{2.60}
    \sum_{n=1}^\infty |\Delta u(n)|^2 \geq \frac{5}{16} \sum_{n=1}^\infty \frac{|u(n)|^2}{n^4}.
\end{equation}
\end{corollary}

\begin{remark}
It is worthwhile to notice that in Theorem \ref{thm2.3} the number of zero conditions on the function $u$ equals the order of the operator. There seems to be a lot of room for the improvement in the constants, though it is not clear how to get better explicit bounds.   
\end{remark}

\begin{remark}
Recently in \cite{gerhat}, authors proved inequality \eqref{2.60} with the best possible constant of 9/16. In fact, they proved the following improvement ($\varphi := n^{3/2}$):
$$ \sum_{n=1}^\infty |\Delta u|^2 \geq \sum_{n=1}^\infty (\Delta^2\varphi/\varphi) |u(n)|^2 \geq 9/16 \sum_{n=1}^\infty |u(n)|^2 |n|^{-2}.$$
A similar improved result was conjectured for operators $\Delta^m$, $D\Delta^m$, which is still open. However, authors in \cite{huang} answered a weak version of the conjecture, namely, they computed the sharp constants in inequalities \eqref{2.58} and \eqref{2.59}. The methods used in both the papers rely on some suitable factorization of the operators at hand.
\end{remark}

Finally, we obtain explicit constants in weighted versions of higher order Hardy Inequalities.
\begin{theorem}[Power weight higher order Hardy Inequalities]\label{thm2.4}
Let $m \geq 1$ and $u \in C_c(\mathbb{Z})$ with $u(0)=0$. Then we have
\begin{equation}\label{2.61}
    \sum_{n \in \mathbb{Z}} |\Delta^m u(n)|^2 n^{2k} \geq \prod_{i=0}^{m-1} C(k-2i) \sum_{n \in \mathbb{Z}} |u(n)|^2 n^{2k-4m}
\end{equation}
for $k \geq 2m$ and 
\begin{equation}\label{2.62}
    \sum_{n \in \mathbb{Z}}|D(\Delta^m u)(n)|^2  \Big(n-\frac{1}{2}\Big)^{2k} \geq \frac{(2k-1)^2}{4} \prod_{i=0}^{m-1} C(k-1-2i) \sum_{n \in \mathbb{Z}}|u(n)|^2 n^{2k-4m-2}
\end{equation}
for $k \geq 2m+1$,

where $C(k)$ is given by
\begin{equation}\label{2.63}
    C(k):= k(k-1)(k-3/2)^2.
\end{equation}
Here $Du(n) := u(n)-u(n-1)$ and $\Delta u(n) := 2 u(n)-u(n-1)-u(n+1)$ denotes the first and second difference operators on $\Z$.
\end{theorem}

\begin{remark}
By taking $n^\beta$ as test functions in the inequalities \eqref{2.61} and \eqref{2.62} it can be easily seen that the sharp constants in these inequalities are of the order $O(k^{4m})$ and $O(k^{4m+2})$ respectively. Therefore constants obtained in Theorem \ref{thm2.4} are asymptotically sharp for large values of $k$.
\end{remark}
\vspace{-19pt}
\subsection{Auxiliary Results}\label{subsec:auxiliary results(fourier transform)}
\begin{lemma}\label{lem2.9}
Let $u \in C^\infty([-\pi, \pi])$. Furthermore, assume that all derivatives of $u$ are $2\pi$-periodic, i.e. $d^k u(-\pi) = d^ku(\pi)$ for all $k \in \mathbb{N}_0$. For every $k \in \mathbb{N}$ we have
\begin{equation}\label{2.64}
    \int_{-\pi}^{\pi} |d^k (u\sin(x/2))|^2 dx = \sum_{i=0}^k \alpha_i^k \int_{-\pi}^{\pi} |d^i u|^2 dx  + \sum_{i=0}^{k} \beta_i^k \int_{-\pi}^{\pi}|d^i u|^2\sin^2(x/2) dx,
\end{equation}
where 
\begin{equation}\label{2.65}
    2^{2(k-i)}\alpha_i^k := \frac{1}{2} {2k \choose 2i} - \frac{1}{2}(-1)^{k-i} {k \choose i}^2 - \frac{1}{2}(-1)^{k-i} \xi_i^k,
\end{equation}
\begin{equation}\label{2.66}
    2^{2(k-i)}\beta_i^k := (-1)^{k-i} \xi_i^k + (-1)^{k-i}{k \choose i}^2,
\end{equation}
and 
\begin{equation}\label{2.67}
    \xi_i^k := \sum_{\substack{0 \leq m \leq \text{min}\{i,k-i\} \\ 1 \leq n \leq k-i}} (-1)^n 2^{n-m} {k+1 \choose i-m} {k \choose i+n}{n-1 \choose m}.
\end{equation}
\end{lemma}

\begin{proof}
Using the Leibniz product rule for the derivative we get
\begin{align*}
    |d^k (u(x)\sin(x/2))|^2 &= |\sum_{i=0}^k {k \choose i}d^iu(x)d^{k-i}\sin(x/2)|^2\\
    &= \sum_{i=0}^k {k \choose i}^2|d^iu(x)|^2|d^{k-i}\sin(x/2)|^2 \\
    & + 2\text{Re} \sum_{0 \leq i<j \leq k} {k \choose i}{k \choose j}d^i u(x) \overline{d^ju(x)} d^{k-i}\sin(x/2) d^{k-j}\sin(x/2).
\end{align*}
Integrating both sides, we obtain
\begin{equation}\label{2.68}
    \begin{split}
        \int_{-\pi}^{\pi} |d^k (u(x)\sin(x/2))|^2 &= \sum_{i=0}^k {k \choose i}^2 \int_{-\pi}^{\pi}     |d^iu(x)|^2|d^{k-i}\sin(x/2)|^2\\
        &+2\text{Re} \sum_{0 \leq i<j \leq k} {k \choose i}{k \choose j} \int_{-\pi}^{\pi} d^i u(x) \overline{d^ju(x)} d^{k-i}\sin(x/2) d^{k-j}\sin(x/2). 
    \end{split}
\end{equation}
Let $i<j$ and $I(i,j):=$Re $\int_{-\pi}^{\pi} d^i u(x) \overline{d^ju(x)} d^{k-i}\sin(x/2) d^{k-j}\sin(x/2)$. Applying integration by parts iteratively, we get
\begin{equation}\label{2.69}
    I(i,j)= \text{Re} \int_{-\pi}^{\pi} d^i u(x) \overline{d^ju(x)} d^{k-i}\sin(x/2) d^{k-j}\sin(x/2) = \sum_{\sigma = i}^{\lfloor\frac{i+j}{2}\rfloor} \int_{-\pi}^{\pi} C_{\sigma}^{i,j}(x) |d^\sigma u|^2,
\end{equation}
where $C_{\sigma}^{i,j}$ is given by
\begin{align*}
    C_{\sigma}^{i,j} = {j-\sigma -1 \choose \sigma -i-1}(-1)^{j-\sigma}d^{i+j-2\sigma}w_{ij}(x)  + \frac{1}{2}{j-\sigma-1 \choose \sigma-i}(-1)^{j-\sigma}d^{i+j-2\sigma}w_{ij}(x),
\end{align*}
and $w_{ij}(x):= d^{k-i}\sin(x/2) d^{k-j}\sin(x/2)$ \footnote{See Appendix \ref{appendix:C} for a proof of identity \eqref{2.69}}.

Using \eqref{2.69} in \eqref{2.68}, we see that 
\begin{equation}\label{2.70}
    \int_{-\pi}^{\pi} |d^k(u(x)\sin(x/2))|^2 = \sum_{i=0}^{k} \int_{-\pi}^{\pi} D_i(x)|d^i u|^2,
\end{equation}
since the derivatives which appear in the expression of $I(i,j)$ are of order between $i $ and $\lfloor\frac{i+j}{2}\rfloor$. Observing that the terms which contributes to $D_i$ are of the form $I(i-m,i+n)$ with the condition $m\leq n$, we get the following expression for $D_i(x)$:
\begin{align*}
    D_i(x) = 2\sum_{\substack{0 \leq m \leq \text{min}\{i,k-i\}\\ m \leq n \leq k-i}} {k \choose i-m} {k \choose i+n}C_i^{i-m,i+n}(x),
\end{align*}
where $C_i^{i,i}(x):= \frac{1}{2}|d^{k-i}\sin(x/2)|^2$.

It can be checked that for non-negative integers $l$, $d^l w_{ij}(x) \in \{\sin^2(x/2), \cos^2(x/2), \cos x, \sin x \}$ (with some multiplicative constant). Thus $D_i(x)$ is a linear combination of $\sin^2(x/2), \cos^2(x/2), \newline \cos x$ and $\sin x$. Namely, we have
\begin{align*}
    D_{i}(x) = C_1^i \sin^2(x/2) + C_2^i \cos^2(x/2) + C_3^i \cos x + C_4^i \sin x.
\end{align*}

Note that $\sin^2(x/2)$ can appear in the expression of $D_i$ iff $w_{i-m,i+n}$ is a multiple of $\sin^2(x/2)$ and $m=n$. Further, observing that $w_{i-m,i+m}$ is a multiple of $\sin^2(x/2)$ iff $k-i+m$ is even, we get 
\begin{equation}\label{2.71}
    C_1^i = 2 \sum_{\substack{ 1 \leq m \leq \text{min}\{i,k-i\} \\ k-i+m \hspace{2pt}\text{is even}}} 2^{-2(k-i)} {k \choose i-m}{k \choose i+m} + 2^{-2(k-i)} {k \choose i}^2 \delta_i,
\end{equation}
where 
\begin{align*}
\delta_i := 
\begin{cases}
    1 \hspace{11pt} \text{if} \hspace{5pt} k-i \hspace{5pt} \text{is even}.\\
    0 \hspace{11pt} \text{if} \hspace{5pt} k-i \hspace{5pt} \text{is odd}.  
\end{cases} 
\end{align*}

Similarly, $\cos^2(x/2)$ can appear in the expression of $D_i$ iff $w_{i-m,i+n}$ is a multiple of $\cos^2(x/2)$ and $m=n$, and $w_{i-m,i+m}$ is a multiple of $\cos^2(x/2)$ iff $k-i+m$ is odd. Therefore we have  
\begin{equation}\label{2.72}
    C_2^i = 2\sum_{\substack{ 1 \leq m \leq \text{min}\{i,k-i\} \\ k-i+m \hspace{2pt}\text{is odd}}} 2^{-2(k-i)} {k \choose i-m}{k \choose i+m} + 2^{-2(k-i)} {k \choose i}^2(1-\delta_i).
\end{equation}

Let us compute the coefficient of $\sin x$ in $D_i$. Observe that $\sin x$ can appear in $D_i$ in two different ways; first, when either $w_{i-m,i+n}$ is a multiple of $\sin^2(x/2)$ or $\cos^2(x/2)$ and $n-m$ is odd; secondly, when $w_{i-m,i+n}$ is a multiple of $\sin x$ and $n-m$ is even. Further, observing that $w_{i-m,i+n}$ is a multiple of $\sin^2(x/2)$ or $\cos^2(x/2)$ iff $n-m$ is even and $w_{i-m,i+n}$ is a multiple of $\sin x$ iff $n-m$ is odd implies that $C_4^i=0$. \newpage 

After computing $C_1^i, C_2^i$ and $C_4^i$, it's not hard to see that 
\begin{equation}\label{2.73}
    C_3^i = (-1)^{k-i-1}  2^{-2(k-i)}\sum_{\substack{0 \leq m \leq \text{min}\{i,k-i\} \\ m < n \leq k-i}} (-1)^n 2^{n-m} {k \choose i-m} {k \choose i+n}\Bigg({n-1 \choose m-1} + \frac{1}{2}{n-1 \choose m} \Bigg).
\end{equation}
Simplifying further, we find that $D_i(x) = (C_2^i+C_3^i) + (C_1^i - C_2^i -2C_3^i)\sin^2(x/2)$. Next we simplify the constants $(C_2^i+C_3^i)$ and $(C_1^i - C_2^i -2C_3^i)$. Let $$\xi_i^k:= \sum\limits_{\substack{0 \leq m \leq \text{min}\{i,k-i\} \\ 1 \leq n \leq k-i}} (-1)^n 2^{n-m} {k+1 \choose i-m} {k \choose i+n}{n-1 \choose m}$$
and consider
\begin{align*}
    (-1)^{k-i-1}C_3^i &= 2^{-2(k-i)} \sum_{\substack{0 \leq m \leq \text{min}\{i,k-i\} \\ m < n \leq k-i}} (-1)^n 2^{n-m} {k \choose i-m} {k \choose i+n}\Big({n-1 \choose m-1} + \frac{1}{2}{n-1 \choose m} \Big)\\
    &= \frac{2^{-2(k-i)}}{2}\xi_i^k - \sum_{1 \leq m \leq \text{min}\{i,k-i\}} (-1)^m 2^{-2(k-i)} {k \choose i-m}{k \choose i+m}\\
    &=\frac{2^{-2(k-i)}}{2}\xi_i^k - \frac{(-1)^{k-i}}{2} \Big(C_1^i-C_2^i + 2^{-2(k-i)}{k \choose i}^2(1-2\delta_i)\Big).
\end{align*}
Simplifying further we obtain 
\begin{equation}\label{2.74}
    2^{2(k-i)}\Big(C_1^{i}-C_2^{i}-2C_{3}^i\Big) = (-1)^{k-i}\xi_i^k + (-1)^{k-i} {k \choose i}^2.     
\end{equation}
Using the expression of $C_3^i$ from \eqref{2.74}, we get
\begin{equation}\label{2.75}
    \begin{split}
        2^{2(k-i)}\Big(C_2^i + C_3^i\Big) &= \sum_{1 \leq m \leq \text{min}\{i,k-i\}} {k \choose i-m} {k \choose i+m} + \frac{1}{2}{k \choose i}^2 - \frac{1}{2}(-1)^{k-i} {k \choose i}^2 - \frac{1}{2}(-1)^{k-i} \xi_i^k\\
        &= \frac{1}{2} {2k \choose 2i} - \frac{1}{2}(-1)^{k-i} {k \choose i}^2 - \frac{1}{2}(-1)^{k-i} \xi_i^k.
    \end{split}
\end{equation}
In the last step we used Chu-Vandermonde Identity: ${m+n \choose r } = \sum\limits_{i=0}^r {m \choose i}{n \choose r-i}$ with some change of variable.
\end{proof}
\begin{lemma}\label{lem2.10}
Let $u$ be a function satisfying the hypothesis of Lemma \ref{lem2.9}. Furthermore, assume that $u$ has zero average, i.e $\int_{-\pi}^\pi u dx = 0$. Then we have
\begin{equation}\label{2.76}
    \int_{-\pi}^\pi |u'|^2 \sin^2(x/2) dx \geq \frac{1}{16} \int_{-\pi}^\pi |u|^2 dx.
\end{equation}
\end{lemma}
\begin{remark}
Inequality \eqref{2.76} is an improvement of well known \emph{Poincar\'e-Friedrichs inequality} in dimension one \cite[Theorem 258]{HLP1952}:
\begin{align*}
    \int_{-\pi}^\pi |u'(x)|^2 dx \geq \int_{-\pi}^\pi |u(x)|^2 dx,
\end{align*}
since $\sin^2(x/2) \leq 1$.
\end{remark}

\begin{proof}
Let $ w(x) := \frac{1}{4}\sec(x/2)$. Expanding the square we obtain
\begin{align*}
    |u' \sin(x/2) + w(u-u(\pi))|^2 &= |u'|^2 \sin^2(x/2) + w^2|u-u(\pi)|^2 + 2\text{Re}  [(w\sin(x/2))u'(\overline{u-u(\pi)})]\\
    &=|u'|^2 \sin^2(x/2) + w^2|u-u(\pi)|^2 + w\sin(x/2)(|u-u(\pi)|^2)'.
\end{align*}
Fix $\epsilon >0$. Doing integration by parts, we obtain
\begin{equation}\label{2.77}
\begin{split}
    \int_{-\pi + \epsilon}^{\pi - \epsilon}  |u' \sin(x/2) + w(u-u(\pi))|^2 dx &= \int_{-\pi + \epsilon}^{\pi -\epsilon} |u'|^2 \sin^2(x/2) \\
    &+ \int_{-\pi + \epsilon}^{\pi -\epsilon}(w^2 - (w\sin(x/2))')|u-u(\pi)|^2 dx + B.T. \geq 0,
\end{split}
\end{equation}
where the boundary term B.T. is given by
\begin{equation}\label{2.78}
    B.T.:= w(\pi - \epsilon)\sin((\pi-\epsilon)/2)|u(\pi-\epsilon)-u(\pi)|^2 - w(-\pi + \epsilon)\sin((-\pi+\epsilon)/2)|u(-\pi+\epsilon)-u(\pi)|^2.
\end{equation}
Therefore we have
\begin{equation}\label{2.79}
    \int_{-\pi + \epsilon}^{\pi -\epsilon} |u'|^2 \sin^2(x/2) dx \geq \int_{-\pi + \epsilon}^{\pi -\epsilon}(-w^2 + (w\sin(x/2))')|u-u(\pi)|^2 dx - B.T.
\end{equation}
Using $-w^2 + (w\sin(x/2))' = \frac{1}{16} \sec^2x \geq 1/16$ above, we obtain
\begin{equation}\label{2.80}
    \int_{-\pi + \epsilon}^{\pi -\epsilon} |u'|^2 \sin^2(x/2) dx \geq \frac{1}{16}\int_{-\pi + \epsilon}^{\pi -\epsilon}|u-u(\pi)|^2 dx - B.T.
\end{equation}
Using periodicity of $u$ along with the first order taylor expansion of $u$ around $\pi$ and $-\pi$, one can easily conclude that B.T. goes to 0 as $\epsilon$ goes to 0. Now taking limit $\epsilon \rightarrow 0$ on both sides of \eqref{2.80} and using dominated convergence theorem, we obtain
\begin{align*}
    \int_{-\pi}^\pi |u'|^2 \sin^2(x/2) dx &\geq \frac{1}{16}\int_{-\pi}^\pi |u-u(\pi)|^2 dx \\
    &= \frac{1}{16}\int_{-\pi}^\pi |u|^2 + \frac{1}{16}\int_{-\pi}^\pi |u(\pi)|^2 -\frac{2}{16}\text{Re}\overline{u(\pi)}\int_{-\pi}^\pi u  dx \geq \frac{1}{16} \int_{-\pi}^\pi |u|^2 dx.
\end{align*}
\end{proof}

\begin{lemma}\label{lem2.11}
Let $u$ be a function satisfying the hypotheses of Lemma \ref{lem2.10}. For $k \in \mathbb{N}$, the following holds
\begin{equation}\label{2.81}
    \int_{-\pi}^{\pi} |d^k(u(x)\sin(x/2))|^2 dx \geq \sum_{i= 0}^{k-1}\Big(\alpha_i^k + \frac{1}{16}\beta_{i+1}^k \Big) \int_{-\pi}^{\pi} |d^i u(x)|^2 dx + \beta_0^k \int_{-\pi}^\pi |u|^2 \sin^2(x/2) dx,
\end{equation}
where $\alpha_i^k$ and $\beta_i^k$ are as defined in \eqref{2.65} and \eqref{2.66} respectively. 
\end{lemma}

\begin{proof}    
Let $f = d^{i-1} u$. Applying Lemma \ref{lem2.10} to $f$ we get
\begin{equation}\label{2.82}
    \int_{-\pi}^\pi |d^i u|^2 \sin^2(x/2)dx \geq \frac{1}{16} \int_{-\pi}^\pi |d^{i-1}u|^2 dx.
\end{equation}
Using \eqref{2.82} in \eqref{2.64} and using $\alpha_k^k =0$ gives the desired estimate \eqref{2.81}.
\end{proof}
\begin{remark}\label{rem3.4}
Note that in proving \eqref{2.81} we have assumed the non-negativity of the constants $\beta_i^k$, which will be proved in Subsection \ref{subsec: combinatorial identity}.
\end{remark}
The next two lemmas are weighted versions of Lemmas \ref{lem2.10} and \ref{lem2.11} proved above, and will be used in proving the higher order Hardy inequalities. 

\begin{lemma}\label{lem2.12}
Let $u$ be a function satisfying the hypotheses of Lemma \ref{lem2.9}. Furthermore, assume that $\int_{-\pi}^\pi u \sin^{2k-2}(x/2) dx = 0$. For $k \geq 1$, we have
\begin{equation}\label{2.83}
    \int_{-\pi}^\pi |u'|^2 \sin^{2k}(x/2) dx \geq \frac{(4k-3)}{16} \int_{-\pi}^\pi |u|^2 \sin^{2k-2}(x/2) dx. 
\end{equation}
\end{lemma}

\begin{proof}
Let $w := \frac{1}{4} \sin^{k-1}(x/2)\sec(x/2)$. Expanding the square, we obtain
\begin{align*}
    |u' \sin^k(x/2) + w(u-u(\pi))|^2 &= |u'|^2 \sin^{2k}(x/2) + w^2|u-u(\pi)|^2 + 2\sin^k(x/2) w \text{Re}[\overline{u'}(u-u(\pi))]\\
    &= |u'|^2 \sin^{2k}(x/2) + w^2|u-u(\pi)|^2 + \sin^k(x/2)w\Big(|u-u(\pi)|^2\Big)'.
\end{align*}

Now integrating over $(-\pi + \epsilon, \pi -\epsilon)$ for a fixed $\epsilon >0$, we get
\begin{align*}
    \int_{-\pi + \epsilon}^{\pi - \epsilon}|u' \sin^2(x/2) + w(u - u(\pi))|^2 & = \int_{-\pi+\epsilon}^{\pi-\epsilon} |u'|^2 \sin^{2k}(x/2) + \int_{-\pi+\epsilon}^{\pi-\epsilon} w^2|u-u(\pi)|^2\\
    &+  \int_{-\pi+\epsilon}^{\pi-\epsilon} w\sin^k(x/2)\Big(|u-u(\pi)|^2\Big)'. 
\end{align*}
\newpage
Finally, using integrating by parts, we obtain 
\begin{equation}\label{2.84}
    \begin{split}
        \int_{-\pi + \epsilon}^{\pi - \epsilon}|u' \sin^{2k}(x/2) + w(u - u(\pi))|^2 & = \int_{-\pi+\epsilon}^{\pi-\epsilon} |u'|^2 \sin^{2k}(x/2) + \int_{-\pi+\epsilon}^{\pi-\epsilon} w^2|u-u(\pi)|^2\\
        &- \int_{-\pi+\epsilon}^{\pi-\epsilon} \Big(w\sin^k(x/2)\Big)'|u-u(\pi)|^2 + B.T. \geq 0,
    \end{split}
\end{equation}
where the boundary term B.T. is given by
\begin{equation}\label{2.85}
    B.T. := |u(\pi-\epsilon)- u(\pi)|^2 w(\pi - \epsilon) \sin^k((\pi -\epsilon)/2) - |u(-\pi+\epsilon)- u(\pi)|^2 w(-\pi + \epsilon) \sin^k((-\pi + \epsilon)/2).  
\end{equation}
Now using $(w\sin^{k}(x/2))'- w^2 = \frac{1}{16} \sin^{2k-2}(x/2) \big(\sec^2(x/2) + 4k-4\big) \geq \frac{4k-3}{16} \sin^{2k-2}(x/2)$, we arrive at
\begin{equation}\label{2.86}
    \int_{-\pi+\epsilon}^{\pi-\epsilon} |u'|^2 \sin^{2k}(x/2) \geq \frac{(4k-3)}{16} \int_{-\pi+\epsilon}^{\pi-\epsilon} |u-u(\pi)|^2 \sin^{2k-2}(x/2) - B.T.
\end{equation}
Now taking limit $\epsilon \rightarrow 0$ on both sides of \eqref{2.86} and using dominated convergence theorem, we obtain
\begin{align*}
     \int_{-\pi}^\pi |u'|^2 \sin^{2k}(x/2) &\geq \frac{(4k-3)}{16} \int_{-\pi}^\pi |u-u(\pi)|^2 \sin^{2k-2}(x/2)\\
    &= \frac{(4k-3)}{16} \int_{-\pi}^\pi |u|^2 \sin^{2k-2}(x/2) + (4k-3)/16 \int_{-\pi}^\pi |u(\pi)|^2 \sin^{2k-2}(x/2)\\ &-\frac{(4k-3)}{8}\text{Re}\overline{u(\pi)} \int_{-\pi}^\pi u \sin^{2k-2}(x/2) \geq \frac{(4k-3)}{16} \int_{-\pi}^\pi |u|^2 \sin^{2k-2}(x/2). 
\end{align*}
\end{proof}
\begin{lemma}\label{lem2.13}
Suppose $u$ satisfies the hypotheses of Lemma \ref{lem2.10}. Further, assume that $u$ has zero average. For $k \geq 2$, we have
\begin{equation}\label{2.87}
    \int_{-\pi}^\pi |d^k(u\sin^2(x/2))|^2 dx \geq \alpha_{k-1}^k \Big(\alpha_{k-2}^{k-1} + \frac{1}{16}\beta_{k-1}^{k-1}\Big) \int_{-\pi}^\pi |d^{k-2}u|^2 dx.
\end{equation}
\end{lemma}

\begin{proof}
We begin with the observation that although $f=u \sin(x/2)$ does not satisfy the hypothesis $d^k f(-\pi) = d^k f(\pi)$ of Lemma \ref{lem2.9}, identity \eqref{2.62} still holds for $f$. In the proof of Lemma \ref{lem2.9}, the periodicity of derivatives is only used in the derivation of \eqref{2.67}; to make sure that no boundary term appears while doing integration by parts. The key observation is that $d^i f(-\pi) = -d^i f (\pi)$, which imply that $d^i f(-\pi) \overline{d^j f (-\pi)} = d^i f(\pi) \overline{d^j f (\pi)}$. This makes sure no boundary terms appears while performing integration by parts in \eqref{2.67} for the function $f$. \\

First using identity \eqref{2.62} for $u\sin(x/2)$ and then for $u$, along with non-negativity of the constants $\alpha_i^k$ and $\beta_i^k$ (will be proved in Subsection \ref{subsec: combinatorial identity}) we obtain
\begin{align*}
    \int_{-\pi}^\pi |d^k(u\sin^2(x/2))|^2dx &\geq \alpha_{k-1}^k \int_{-\pi}^\pi |d^{k-1}(u\sin(x/2))|^2 dx\\
    & \geq \alpha_{k-1}^k \Big(\alpha_{k-2}^{k-1} + \frac{1}{16}\beta_{k-1}^{k-1}\Big) \int_{-\pi}^\pi |d^{k-2}u|^2 dx.
\end{align*}
Last inequality uses Lemma \ref{lem2.10}.
\end{proof}
\subsection{Proof of Hardy inequalities}\label{subsec:proofs of Hardy inequalities(foruier transform)}

\begin{proof}[Proof of Theorem \ref{thm2.2}]
Let $u \in l^2(\mathbb{Z})$, then its \emph{Fourier transform} $\mathcal{F}(u) \in L^2((-\pi, \pi))$ is defined by
\begin{equation}\label{2.88}
    \mathcal{F}(u)(x) := (2\pi)^{-\frac{1}{2}}\sum_{n \in \mathbb{Z}} u(n) e^{-inx}, \hspace{19pt} x \in (-\pi, \pi).
\end{equation}
Let $1 \leq j \leq k$. Using the inversion formula for Fourier transform and integration by parts, we get
\begin{align*}
    u(n)n^{k-j} &= (2\pi)^{-\frac{1}{2}} \int_{-\pi}^\pi \mathcal{F}(u)(x)n^{k-j} e^{inx}dx\\ 
    &=  \frac{(2\pi)^{-\frac{1}{2}}}{i^{k-j}} \int_{-\pi}^\pi \mathcal{F}(u)(x)d^{{k-j}} e^{inx}dx = \frac{(-1)^{k-j}(2\pi)^{-\frac{1}{2}}}{i^{k-j}}\int_{-\pi}^\pi d^{{k-j}} \mathcal{F}(u)(x) e^{inx}dx. 
\end{align*}
Applying Parseval's Identity gives us 
\begin{equation}\label{2.89}
    \sum_{n \in \mathbb{Z}}|u(n)|^2 n^{2(k-j)} = \int_{-\pi}^{\pi}|d^{k-j} \mathcal{F}(u)(x)|^2 dx.
\end{equation}
Similarly one gets the following identity 
\begin{equation}\label{2.90}
    \sum_{n \in \mathbb{Z}}|u(n)-u(n-1)|^2 \Big(n-\frac{1}{2}\Big)^{2k} = 4 \int_{-\pi}^\pi |d^k (\mathcal{F}(u)\sin(x/2))|^2 dx.
\end{equation}
Finally, applying  Lemma \ref{lem2.11} on $\mathcal{F}(u)$ and then using \eqref{2.89}, \eqref{2.90} we get
\begin{equation}\label{2.91}
    \begin{split}
        \sum_{n \in \mathbb{Z}} |u(n)-u(n-1)|^2 \Big(n-\frac{1}{2}\Big)^{2k} &\geq \sum_{i=1}^k \gamma_i^k \sum_{n \in \mathbb{Z}} |u(n)|^2 n^{2(k-i)} + \beta_0^k \sum_{n \in \mathbb{Z}}|u(n)-u(n-1)|^2\\
        & \geq \sum_{i=1}^k \gamma_i^k \sum_{n \in \mathbb{Z}} |u(n)|^2 n^{2(k-i)} + \frac{\beta_0^k}{4} \sum_{n \in \mathbb{Z}\setminus\{0\}} \frac{|u(n)|^2}{n^2},
    \end{split}
\end{equation}
where $\gamma_i^k := 4\alpha_{k-i}^k + \frac{1}{4}\beta_{k-i+1}^k$. 
In the last step we used the classical Hardy inequality.

In Subsection \ref{subsec: combinatorial identity} we simplify the expressions of $\alpha_i^k$ and $\beta_i^k$. This will complete the proof of Theorem \ref{thm2.2}.
\end{proof}

\begin{proof}[Proof of Corollary \ref{cor2.3}]
Assuming $\gamma_i^k \geq 0$ (which will be proved in Subsection \ref{subsec: combinatorial identity}), Theorem \ref{thm2.2} immediately implies  
\begin{align*}
    \sum_{n \in \mathbb{Z}}|u(n)-u(n-1)|^2 \Big(n-\frac{1}{2}\Big)^{2k} \geq \gamma_1^k \sum_{n \in \mathbb{Z}} |u(n)|^2 n^{2k-2}.
\end{align*}
It can be easily checked that $\xi_{k-1}^k = -k(k+1)$. Using this we find that $\gamma_1^k = \frac{(2k-1)^2}{4}$. Next, we prove the sharpness of the constant $\gamma_1^k$. Let $C$ be a constant such that
\begin{equation}\label{2.92}
    \sum_{n \in \mathbb{Z}}|u(n)-u(n-1)|^2 \Big(n-\frac{1}{2}\Big)^{2k} \geq C \sum_{n \in \mathbb{Z}} |u(n)|^2 n^{2k-2}
\end{equation}
for all $u \in C_c(\mathbb{Z})$. Let $N \in \mathbb{N}$, $\beta \in \mathbb{R}$ and $\alpha\geq 0$ be such that $2\beta + 2k-2 <-1$. Consider the following family of finitely supported functions on $\mathbb{Z}$. 
\begin{align*}
u_{\beta,N}(n):=
\begin{cases}
n^\beta \hspace{73pt} &\text{for} \hspace{5pt}  1 \leq  n \leq N.\\
-N^{\beta-1} n + 2N^{\beta} \hspace{13pt}  &\text{for} \hspace{5pt}  N \leq n \leq 2N. \\
0 \hspace{83pt} &\text{for} \hspace{5pt}  n \geq 2N \hspace{5pt} \text{and} \hspace{5pt} n\leq 0.
\end{cases} 
\end{align*}
Clearly we have
\begin{equation}\label{2.93}
    \sum_{n \in \mathbb{Z}}|u_{\beta,N}(n)|^2n^{2k-2} \geq \sum_{n=1}^{N}n^{2\beta+2k-2},
\end{equation}
and
\begin{equation}\label{2.94}
    \begin{split}
        \sum_{n \in \mathbb{Z}}|u_{\beta, N}(n)-u_{\beta,N}(n-1)|^2(n-1/2)^{2k}
        &=\sum_{n=1}^\infty|u_{\beta, N}(n)-u_{\beta,N}(n-1)|^2(n-1/2)^{2k} \\
        &\leq  \sum_{n=2}^{N}(n^\beta - (n-1)^\beta)^2n^{2k} + \sum_{n=N+1}^{2N}N^{2\beta-2}n^{2k} + 1.
    \end{split}
\end{equation}
Some basic estimates:
\begin{align*}
    (n^\beta - (n-1)^\beta)^2 &\leq \beta^2(n-1)^{2\beta-2},\\
\sum_{n=N+1}^{2N} n^{2k} &\leq \int_{N+1}^{2N+1} x^{2k} dx = \frac{(2N+1)^{2k+1} - (N+1)^{2k+1}}{2k+1}.
\end{align*}

Using the above in \eqref{2.94}, we get
\begin{equation}\label{2.95}
    \begin{split}
        \sum_{n \in \mathbb{Z}}|u_{\beta, N}(n)-u_{\beta,N}(n-1)|^2(n-1/2)^{2k} &\leq \beta^2 \sum_{n=2}^{N}(n-1)^{2\beta-2}n^{2k} \\
        &+ \frac{N^{2\beta+2k-1}}{2k+1}\Bigg[\Big(2+\frac{1}{N}\Big)^{2k+1} - \Big(1+\frac{1}{N}\Big)^{2k+1}\Bigg]+1.
    \end{split}
\end{equation}

Using estimates \eqref{2.93} and \eqref{2.95} in \eqref{2.92} and taking limit $N \rightarrow \infty$, we get
\begin{align*}
    C\sum_{n=1}^{\infty}n^{2\beta+2k-2} &\leq \beta^2\sum_{n=2}^{\infty}(n-1)^{2\beta-2}n^{2k} + 1\\
    &= \beta^2 \sum_{i=0}^{2k}{2k \choose i}\sum_{n=1}^\infty n^{2\beta + 2k - i -2} + 1.
\end{align*}
Finally, taking limit $\beta \rightarrow \frac{1-2k}{2}$ on the both sides, we obtain
\begin{equation}\label{2.96}
    C \leq \frac{(2k-1)^2}{4}. 
\end{equation}
This proves the sharpness of $\gamma_1^k$.
\end{proof}

\subsection{Proof of Higher Order Hardy inequalities}\label{Subsec: proof of higher order Hardy(foruier transform)}

\begin{proof}[Proof of Theorem \ref{thm2.3}]
First we prove inequality \eqref{2.58} and then inequality \eqref{2.59}. Let $m \in \mathbb{N}$, $v \in C_c(\mathbb{Z})$ with $v(0)=0$ and
\begin{align*}
    \Tilde{v}(n):=
    \begin{cases}
        \frac{v(n)}{n^{2m}} \hspace{19pt} \text{if} \hspace{5pt} n \neq 0\\
        0 \hspace{33pt} \text{if} \hspace{5pt} n=0
    \end{cases}
\end{align*}
Using the inversion formula for Fourier transform, we obtain
\begin{align*}
    \Delta v = 2v(n)-v(n-1)-v(n+1) &=
    (2\pi)^{-\frac{1}{2}}\int_{-\pi}^\pi \mathcal{F}(v)(2-e^{-ix}-e^{ix})e^{inx} dx\\
    &= (2\pi)^{-\frac{1}{2}}\int_{-\pi}^\pi 4 \mathcal{F}(v)\sin^2(x/2)e^{inx} dx.
\end{align*}
Therefore we have $\mathcal{F}(\Delta v) = 4 \sin^2(x/2) F(v)$. Applying this formula iteratively, we obtain $\mathcal{F}(\Delta^m v) = 4^m \sin^{2m}(x/2)\mathcal{F}(v)$. Using Parseval's, identity we get \newpage
\begin{align*}
    \sum_{n \in \mathbb{Z}\setminus\{0\}} \frac{|v|^2}{n^{4m}} &= \int_{-\pi}^\pi |\mathcal{F}(\Tilde{v})|^2 dx.\\
    \sum_{n \in \mathbb{Z}}|\Delta^m v|^2 &= 4^{2m}\int_{-\pi}^\pi |\mathcal{F}(v)|^2 \sin^{4m}(x/2) dx = 4^{2m}\int_{-\pi}^\pi |\mathcal{F}(\Tilde{v})^{(2m)}|^2 \sin^{4m}(x/2) dx.
\end{align*}
Using Lemma \ref{lem2.12} iteratively and then Lemma \ref{lem2.10}, we obtain
\begin{align*}
    4^{-2m}\sum_{n \in \mathbb{Z}}|\Delta^m v|^2 =  \int_{-\pi}^\pi |\mathcal{F}(\Tilde{v})^{(2m)}|^2 \sin^{4m}(x/2) dx &\geq  \frac{1}{2^{8m}} \prod_{i=0}^{2m-1} (8m-3-4i)\int_{-\pi}^\pi |\mathcal{F}(\Tilde{v})|^2 dx\\
    &=\frac{1}{2^{8m}} \prod_{i=0}^{2m-1} (8m-3-4i) \sum_{n \in \mathbb{Z}\setminus\{0\}} \frac{|v|^2}{n^{4m}},
\end{align*}
under the assumption that 
\begin{equation}\label{2.97}
    \int_{-\pi}^\pi \mathcal{F}(\Tilde{v})^{(2m-k)} \sin^{2(2m-k)}(x/2)dx = \sum_{n\in \mathbb{Z}} \mathcal{F}^{-1}(\mathcal{F}(\Tilde{v})^{(2m-k)})(n)\mathcal{F}^{-1}(\sin^{2(2m-k)}(x/2)) = 0
\end{equation}
for $1 \leq k \leq 2m$. 
Next we compute the inverse Fourier transform of $\sin^{2(2m-k)}(x/2)$ to simplify the condition \eqref{2.97}.
Consider
\begin{align*}
    \sin^{2(2m-k)}(x/2) &= 2^{-(2m-k)}(1-\cos x)^{2m-k}\\
    &= 2^{-(2m-k)} \sum_{j=0}^{2m-k}{2m-k \choose j} (-1/2)^j (e^{ix} + e^{-ix})^j \\
    &=2^{-(2m-k)} \sum_{j=0}^{2m-k}\sum_{j'=0}^j {2m-k \choose j} {j \choose j'}(-1/2)^j e^{-ix(2j'-j)}.
\end{align*}
Using the above expression, we obtain
\begin{align*}
    \sum_{n \in \mathbb{Z}}\mathcal{F}^{-1}(\mathcal{F}(\Tilde{v})^{(2m-k)})(n)&\mathcal{F}^{-1}(\sin^{2(2m-k)}(x/2)) \\
    &=  \sum_{j=0}^{2m-k}\sum_{\substack{0 \leq j' \leq j\\ j'\neq j/2 }} {2m-k \choose j} {j \choose j'}(-1/2)^j(2j'-j)^{-k} v(2j'-j) = 0.
\end{align*}
So finally we arrive at the following inequality
\begin{equation}\label{2.98}
    \sum_{n \in \mathbb{Z}}|\Delta^m v|^2 \geq 2^{2m-3}(2m-1)! \sum_{n \in \mathbb{Z}\setminus\{0\}}\frac{|v|^2}{n^{4m}},
\end{equation}
provided $v \in C_c(\mathbb{Z})$ with $v(0) =0 $ satisfies
\begin{equation}\label{2.99}
    \sum_{j=0}^{2m-k}\sum_{\substack{0 \leq j' \leq j\\ j'\neq j/2 }} {2m-k \choose j} {j \choose j'}(-1/2)^j(2j'-j)^{-k} v(2j'-j) = 0
\end{equation}
for $1 \leq k \leq 2m$.

Let $u \in C_c(\mathbb{N}_0)$ with $u(i)=0$ for all $0\leq i \leq 2m-1$. We define $v \in C_c(\mathbb{Z})$ as
\begin{align*}
    v(n):=
    \begin{cases}
        u(n) \hspace{19pt} \text{if} \hspace{5pt} n \geq 0\\
        0 \hspace{36pt} \text{if} \hspace{5pt} n <0.
    \end{cases}
\end{align*}
It is quite straightforward to check that the condition \eqref{2.99} is trivially satisfied. Now applying inequality \eqref{2.98} to the above defined function $v$, we obtain
\begin{equation}\label{2.100}
     \sum_{n=1}^\infty |\Delta^m u|^2 \geq \frac{1}{2^{4m}} \prod_{i=0}^{2m-1} (8m-3-4i) \sum_{n=1}^\infty \frac{|u|^2}{n^{4m}}.
\end{equation}
This proves inequality \eqref{2.58}. Inequality \eqref{2.59} can be proved in a similar way, by following the proof of \eqref{2.58} step by step. 
\end{proof}

\begin{proof}[Proof of Theorem \ref{thm2.4}]
First we will prove inequality \eqref{2.61}. We begin by proving the result for $m=1$ and then apply the result for $m=1$ iteratively to prove it for general $m$. Using inversion formula and integration by parts, we obtain
\begin{align*}
    u(n)n^{k-2} = (2\pi)^{-\frac{1}{2}} \int_{-\pi}^\pi \mathcal{F}(u)(x)n^{k-2} e^{inx}dx = \frac{(-1)^{k-2}(2\pi)^{-\frac{1}{2}}}{i^{k-2}}\int_{-\pi}^\pi d^{{k-2}} \mathcal{F}(u)(x) e^{inx}dx. 
\end{align*}
Applying Parseval's Identity gives us 
\begin{equation}\label{2.105}
    \sum_{n \in \mathbb{Z}}|u(n)|^2 n^{2(k-4)} = \int_{-\pi}^{\pi}|d^{k-2} \mathcal{F}(u)(x)|^2 dx. 
\end{equation}
Similarly, one gets the following identity 
\begin{equation}\label{2.106}
    \sum_{n \in \mathbb{Z}}|\Delta u|^2 n^{2k} = 16 \int_{-\pi}^\pi |d^k (\mathcal{F}(u)\sin^2(x/2))|^2 dx.
\end{equation}
Now applying Lemma \ref{lem2.13} and then using equations \eqref{2.105} and \eqref{2.106}, we get
\begin{equation}\label{2.107}
    \sum_{n \in \mathbb{Z}}|\Delta u|^2 n^{2k} \geq  k(k-1)(k-3/2)^2 \sum_{n \in \mathbb{Z}} |u|^2 n^{2k-4}.
\end{equation}
In the last line we used $\alpha_{k-1}^k = k(k-1)$ and $\beta_k^k = 1$(see \eqref{2.65}-\eqref{2.67}). Now applying the inequality \eqref{2.107} inductively completes the proof of inequality \eqref{2.61}. For the proof of inequality \eqref{2.62}, we first apply inequality \eqref{2.55} and then inequality \eqref{2.61}.
\end{proof}

\subsection{Combinatorial identity}\label{subsec: combinatorial identity}
We prove a combinatorial identity using the functional identity proved in Lemma \ref{lem2.9}. The analytic proof of the identity as well as its appearance in the context of discrete Hardy-type inequalities is quite surprising. 
\begin{theorem}\label{thm2.5}
Let $k \in \mathbb{N}$ and $ 0 \leq i \leq k$. Then
\begin{equation}\label{2.108}
    \sum_{\substack{0 \leq m \leq \text{min}\{i,k-i\} \\ 1 \leq n \leq k-i}} (-1)^n 2^{n-m} {k+1 \choose i-m} {k \choose i+n}{n-1 \choose m} = (-1)^{k-i} {k \choose i} - {k \choose i}^2.    
\end{equation}
\end{theorem}

\begin{proof}
Using $\sin^2(x/2) = (1-\cos x)/2$, identity \eqref{2.64} can be re-written as  
\begin{equation}\label{2.109}
\begin{split}
    \sum_{i=0}^{k}(-1)^{k-i}2^{-2(k-i)}\Big(\xi_i^k + {k \choose i}^2\Big) \int_{-\pi}^{\pi}|d^i u|^2 \cos x dx - &\sum_{i=0}^k 2^{-2(k-i)}{2k \choose 2i} \int_{-\pi}^{\pi} |d^i u|^2 dx\\
    &= -2 \int_{-\pi}^{\pi} |d^k(u \sin(x/2))|^2 dx. 
\end{split}
\end{equation}
Let $u = e^{in(x/2)}\sin(x/2)$. Then some straightforward calculations give us the following identities for $m \geq 0$
\begin{equation}\label{2.110}
    2^{2m}\int_{\pi}^{\pi}|d^m u|^2 dx = \frac{\pi}{2} \Big((n+1)^{2m} + (n-1)^{2m} \Big).  
\end{equation}
\begin{equation}\label{2.111}
    2^{2m} \int_{-\pi}^{\pi} |d^m u|^2 \cos x = \frac{-\pi}{2} (n^2 -1)^m.
\end{equation}
\begin{equation}\label{2.112}
    2^{2k}\int_{-\pi}^{\pi}|d^k(u \sin (x/2))|^2 = \frac{\pi}{8} \Big( (n+2)^{2k} + (n-2)^{2k} + 4n^{2k}\Big).
\end{equation}
Using equations \eqref{2.110} - \eqref{2.112} in \eqref{2.109}, we obtain
\begin{equation}\label{2.113}
    \begin{split}
        -\frac{1}{2} \sum_{i=0}^k (-1)^{k-i} \Big(\xi_i^k + {k \choose i}^2 \Big)(n^2-1)^i &= \frac{1}{2} \sum_{i=0}^k {2k \choose 2i} \Big((n+1)^{2i} + (n-1)^{2i}\Big)\\
        &- \frac{1}{4} \Big((n+2)^{2k} + (n-2)^{2k} + 4n^{2k}\Big) \\
        &=  -\frac{1}{2} n^{2k}.
    \end{split}
\end{equation}
The last step uses 
\begin{equation}\label{2.114}
    \sum_{i=0}^k {2k \choose 2i} (n+1)^{2i} = \frac{1}{2}\Big((n+2)^{2k} + n^{2k}\Big),
\end{equation}
and 
\begin{equation}\label{2.115}
    \sum_{i=0}^k {2k \choose 2i} (n-1)^{2i} = \frac{1}{2} \Big((n-2)^{2k} + n^{2k}\Big).
\end{equation}
Therefore, for $n \in \mathbb{N}$, we have
\begin{equation}\label{2.116}
    \sum_{i=0}^k (-1)^{k-i} \Big(\xi_i^k + {k \choose i}^2 \Big)(n^2-1)^i =  \sum_{i=0}^{k}{k \choose i} (n^2 -1)^i,
\end{equation}
which implies the identity \eqref{2.108}.
\end{proof}

\begin{remark}
Using identity \eqref{2.108}, expressions of $\alpha_i^k , \beta_i^k$ defined by \eqref{2.65}, \eqref{2.66} respectively become
\begin{align*}
    2^{2(k-i)}\alpha_i^k = \frac{1}{2}{2k \choose 2i} - \frac{1}{2}{k \choose i} \hspace{19pt} \text{and} \hspace{19pt} 2^{2(k-i)}\beta_i^k = {k \choose i},
\end{align*}
and $\gamma_i^k := 4\alpha_{k-i}^k + \frac{1}{4}\beta_{k-i+1}^k$ becomes
\begin{align*}
    2^{2i}\gamma_i^k &= 2 {2k \choose 2i} - 2 {k \choose i} + {k \choose i-1}.
\end{align*}

From the above expressions, it is quite straightforward that the above constants are non-negative, thus justifying the assumptions used in the proofs of Lemma \ref{lem2.11}, Lemma \ref{lem2.13} and Corollary \ref{cor2.3}. Finally, the expression of $\gamma_i^k$ along with \eqref{2.91} completes the proof of Theorem \ref{thm2.2}.
\end{remark}

%% file: higher_Hardy_inequality/higher_Hardy_inequality.tex
\chapter{Hardy inequality on $\Z^d, d \geq 3$ }\label{ch:higher-hardy}

\section{Introduction}
In this chapter we study discrete Hardy inequalities on the $d$-dimensional lattice $\Z^d$ for $d \geq 3$. In particular, we are interested in understanding the behaviour of the best possible constants in  
\begin{equation}\label{3.1}
    \sum_{n \in \Z^d} |Du(n)|^2 \geq C_H(d) \sum_{n \in \Z^d} \frac{|u(n)|^2}{|n|^2},
\end{equation}
and 
\begin{equation}\label{3.2}
    \sum_{n \in \Z^d} |\Delta u(n)|^2 \geq C_R(d) \sum_{n \in \Z^d} \frac{|u(n)|^2}{|n|^4},
\end{equation}
where $Du(n) := (D_1u(n), D_2u(n),.., D_d u(n))$, 
\begin{align*}
    D_ju(n) := u(n)-u(n-e_j), \hspace{9pt} \Delta u(n) := \sum_{j=1}^d 2u(n)-u(n-e_j)-u(n+e_j),
\end{align*}
and $e_j$ is the $j^{th}$ canonical basis of $\R^d$. As discussed before \eqref{3.1} and \eqref{3.2} are discrete counterparts of Hardy and Rellich inequalities (\eqref{1.1}, \eqref{1.2}) respectively. 

Let us briefly go through known results. Discrete Hardy inequality \eqref{3.1} appeared for the first time in papers by Rozenblum and Solomyak, within the context of spectral theory of discrete Schr\"odinger operator \cite{solomyak1, solomyak2}. In those papers, authors proved Hardy inequality without any explicit estimate on the constant. In 2016, an explicit constant was computed by Kapitanski and Laptev \cite{laptev}. However, their constant did not scale with the dimension of the underlying space. Recently, in 2018 Keller, Pinchover and Pogorzelski developed a general framework for studying Hardy inequalities on a infinite graphs \cite{keller1}. But, their theory does not give classical Hardy inequality \eqref{3.1} when applied to integer lattices. More precisely, they proved

\begin{theorem}[Keller, Pinchover, Pogorzelski \cite{keller1}]\label{thm3.1}
Let $d \geq 3$ and $u \in C_c(\Z^d)$. Let $G(x)$ be the Green function associated to the $\Delta$ on the lattice $\Z^d$. Then
\begin{equation}\label{3.3}
    \sum_{n \in\Z^d} |Du(n)|^2 \geq \sum_{n \in \Z^d} \frac{\Delta \sqrt{G(x)}}{\sqrt{G(x)}}|u(n)|^2.
\end{equation}
Moreover, as $|n| \rightarrow \infty$ we have the following asymptotic expansion
\begin{equation}\label{3.4}
    \frac{\Delta \sqrt{G(x)}}{\sqrt{G(x)}} = \frac{(d-2)^2}{4}\frac{1}{|n|^2} + O\Bigg(\frac{1}{|n|^3}\Bigg). 
\end{equation}
\end{theorem}
It is not clear whether the lower order terms in the expansion \eqref{3.4} are positive or negative. Hence, it is not possible to derive \eqref{3.1} from Theorem \ref{thm3.1}

To the best of our knowledge, the work of Keller, Pinchover and Pogorzelski \cite{keller2} is the only one where Rellich inequality \eqref{3.2} has been studied in higher dimensions. They proved Rellich inequality with a weight which grows asymptotically as $|n|^{-4}$ as $|n| \rightarrow \infty$ for $d \geq 5$. However, similar to the Hardy inequality case, the sign of the lower order terms in the expansion of their weight is not clear. Thus, it is not possible to derive Rellich inequality \eqref{3.2} from their result.  

In this chapter we find the asymptotic behaviour of sharp constants in \eqref{3.1} and \eqref{3.2}, as $d \rightarrow \infty$. In fact, our method also yields the asymptotic behaviour of sharp constants in higher order versions of \eqref{3.1} and \eqref{3.2}. Along the way, we also prove Hardy type inequalities on a torus for functions with zero average. Hardy inequalities for this class of functions have not been studied before, and are interesting problems by themselves. 

Similar to the previous chapter, we begin by stating the main results in Section \ref{sec: main results(higher hardy inequality)}. In Section \ref{sec: converting discrete to continuous} we convert discrete Hardy inequalities at hand to some continuous Hardy type inequalities on a torus. In Section \ref{sec: continuous Hardy inequality}, we prove various Hardy type inequalities on a torus for higher order operators, which are later used to prove the main results in Section \ref{sec: main proof(higher hardy inequality)}. 

\section{Main Results}\label{sec: main results(higher hardy inequality)}
In the following theorems $C_c(\Z^d)$ denotes the space of finitely supported functions on $\Z^d$. 
\begin{theorem}\label{thm3.2}
Let $k \geq 0$ and let $u \in C_c(\Z^d)$ with $u(0)=0$. Let $C_1(k, d)$ be the sharp constant in the following inequality:
\begin{equation}\label{3.5}
    \sum_{n \in \Z^d} |D(\Delta^k u)(n)|^2 \geq C_1(k, d) \sum_{n \in \Z^d} \frac{|u(n)|^2}{|n|^{4k+2}}.
\end{equation}
Then $C_1(k, d) \sim d^{2k+1}$ as $d \rightarrow \infty$, that is, there exists positive constants $c_1, c_2$ and $N$ independent of $d$ such that $c_1 d^{2k+1}\leq C_1(k,d) \leq c_2 d^{2k+1}$ for all $ d\geq N$. 
\end{theorem}

\begin{theorem}\label{thm3.3}
Let $k \geq 1$ and let $u \in C_c(\Z^d)$ with $u(0)=0$. Let $C_2(k, d)$ be the sharp constant in the following inequality:
\begin{equation}\label{3.6}
    \sum_{n \in \Z^d} |\Delta^k u(n)|^2 \geq C_2(k, d) \sum_{n \in \Z^d} \frac{|u(n)|^2}{|n|^{4k}}.
\end{equation}
Then $C_2(k, d) \sim d^{2k}$ as $d \rightarrow \infty$.
\end{theorem}

\begin{remark}
We would like to point out that the sharp constants in continuous analogues of \eqref{3.5} and \eqref{3.6} on $\R^d$ grows as $d^{4k+2}$ and $d^{4k}$ as $d \rightarrow \infty$ respectively, see \cite{davies}, for a computation of sharp constants. This shows that the discrete Hardy inequalities are not the continuous ones in disguise.  
\end{remark}

\begin{remark}
Note that putting $k=0$ in \eqref{3.5} and $k=1$ in \eqref{3.6} give the discrete Hardy \eqref{3.1} and Rellich inequality \eqref{3.2} respectively. 
\end{remark} 

\begin{corollary}\label{cor3.1}
Let $u \in C_c(\Z^d)$ with $u(0) =0$. Let $C_H(d)$ be the sharp constant in the discrete Hardy inequality
\begin{equation}\label{3.7}
    \sum_{n \in \Z^d} |Du(n)|^2 \geq C_H(d) \sum_{n \in \Z^d} \frac{|u(n)|^2}{|n|^2}.
\end{equation}
Then $C_H(d) \sim d$, as $d \rightarrow \infty$ \footnote{In Appendix \ref{appendix:A} we used similar ideas to prove Hardy inequality \eqref{3.7} restricted to class of antisymmetric functions. We proved that sharp constant for this class of functions have substantially better constants in higher dimensions.}.
\end{corollary}

\begin{corollary}\label{cor3.2}
Let $u \in C_c(\Z^d)$ with $u(0) =0$. Let $C_R(d)$ be the sharp constant in the discrete Rellich inequality
\begin{equation}\label{3.8}
    \sum_{n \in \Z^d} |\Delta u(n)|^2 \geq C_R(d) \sum_{n \in \Z^d} \frac{|u(n)|^2}{|n|^4}.
\end{equation}
Then $C_R(d) \sim d^2$, as $d \rightarrow \infty$.
\end{corollary}

\begin{remark}\label{rem3.3}
Consider a function $u$ defined on $\Z^d$ as follows: $u(n) = 1$ when $|n|=1$ and zero otherwise. Using this function in inequalities \eqref{3.5} and \eqref{3.6} gives us the bounds $C_1(k, d) \leq 4^{2k+1}d^{2k+1} \sim d^{2k+1}$ and $C_2(k, d) \leq 4^{2k}d^{2k} \sim d^{2k}$(this will be proved in Section \ref{sec: main proof(higher hardy inequality)}). We prove inequalities \eqref{3.5}, \eqref{3.6} with explicit constants which grow as $d^{2k+1}$ and $d^{2k}$ respectively, as $d \rightarrow \infty$, thereby completing the proofs of our main results. 
\end{remark}

\begin{remark}\label{rem3.4}
A natural question that arises from this work is the determination of exact constant in the leading term of the asymptotic expansion of $C_H(d)$ as $d \rightarrow \infty$, in other words what is the value of $\lim\limits_{d \rightarrow \infty} C_H(d)/d ?$ 
\end{remark}

\section{Equivalent integral inequalities}\label{sec: converting discrete to continuous}

In this section, we convert the inequalities \eqref{3.5} and \eqref{3.6} into some equivalent integral inequalities on a torus. In this Chapter, $Q_d := (-\pi, \pi)^d$ denotes an open square in $\R^d$.
\begin{definition}
Let $\psi : \overline{Q_d} \rightarrow \C$ be a map. Then we say it is $2\pi$-periodic in each variable if 
\begin{align*}
    \psi(x_1,..,x_{i-1}, -\pi, x_{i+1},.., x_d) = \psi(x_1,..,x_{i-1}, \pi, x_{i+1},.., x_d),
\end{align*}
for all $1 \leq i \leq d$.
\end{definition}
\begin{lemma}\label{lem3.1}
Let $k \geq 0$ be an integer. Let $ u\in C_c(\Z^d)$ with $u(0)=0$. There exists $\psi \in C^\infty(\overline{Q_d})$, all of whose derivatives are $2\pi$-periodic in each variable and which has zero average, such that
\begin{equation}\label{3.9}
    \sum_{n \in \Z^d} \frac{|u(n)|^2}{|n|^{4k+2}} = \int_{Q_d} |\nabla(\Delta^k \psi)(x)|^2 dx,
\end{equation}
and 
\begin{equation}\label{3.10}
    \sum_{n \in \Z^d} |D(\Delta^k u)(n)|^2 = 4^{2k+1}\int_{Q_d} |\Delta^{2k+1} \psi(x)|^2 \omega(x)^{2k+1}dx,
\end{equation}
where 
\begin{equation}\label{3.11}
    \omega(x) := \sum_{j=1}^d \sin^2(x_j/2).
\end{equation}
\end{lemma}

\begin{proof}
We first prove the result for $k=0$, and then extend the proof to general $k$. Let $u \in \ell^2(\Z^d)$. Its Fourier transform $\widehat{u}$ is defined as 
\begin{align*}
    \widehat{u}(x) := (2\pi)^{-\frac{d}{2}} \sum_{n \in \Z^d} u(n) e^{-i n \cdot x}, \hspace{9pt} x \in (-\pi, \pi)^d.
\end{align*}
Let $u_j(n) := \frac{n_j}{|n|^2} u(n) $, for $n \neq 0$ and $u_j(0) =0$. Since $\{e^{-in \cdot x}\}_{n \in \Z^d}$ forms an orthonormal basis of $L^2(Q_d)$, the Parseval's identity gives us
\begin{equation}
    \sum_{n \in \Z^d}\frac{|u(n)|^2}{|n|^2} = \sum_{j=1}^d \sum_{n \in \Z
    ^d}| u_j|^2 = \int_{Q_d} \sum_{j=1}^d |\widehat{u_j}|^2 dx.
\end{equation}
Using the inversion formula for Fourier transform we get
\begin{align*}
    u(n) - u(n-e_j) = (2\pi)^{-\frac{d}{2}}\int_{Q_d} \widehat{u}(x)(1-e^{-ix_j}) e^{in \cdot x} dx.
\end{align*}
Applying Parseval's identity and summing w.r.t. to $j$, we obtain
\begin{equation}
    \begin{split}
        \sum_{n \in \Z^d} \sum_{j=1}^d |u(n)-u(n-e_j)|^2 &= 4\int_{Q_d} |\widehat{u}(x)|^2 \sum_{j=1}^d  \sin^2(x_j/2)\\
        &= 4\int_{Q_d} \Big|\sum_{j=1}^d \partial_{x_j}\widehat{u_j}(x)\Big|^2 \sum_{j=1}^d  \sin^2(x_j/2) dx,    
    \end{split}
\end{equation}
where the last identity uses $\sum_j \partial_{x_j}\widehat{u_j}(x) = -i \widehat{u}$.

The inversion formula for Fourier transform along with integration by parts gives us
\begin{align*}
    (2\pi)^{\frac{d}{2}}n_k u_j(n) = \int_{Q_d} n_k \widehat{u_j}(x) e^{i n \cdot x} = i \int_{Q_d}  \partial_{x_k} \widehat{u_j}(x) e^{i n \cdot x}, 
\end{align*}
which implies that
\begin{equation}
    \partial_{x_k}\widehat{u_j} = \partial_{x_j}\widehat{u_k}.
\end{equation}
This further implies that there exists a smooth function $\psi$ such that $\widehat{u_j}(x) = \partial_{x_j} \psi(x)$, whose average is zero. It is easy to see that periodicity of $\widehat{u_j}$ along with its zero average imply that $\psi$ is also $2\pi$ periodic in each variable. This proves the result for $k=0$. 
\vspace{-2pt}
Next we prove the result for general $k > 0$ in a similar way. Let $j_1, j_2,.., j_{2k+1}$ be integers lying between $1$ and $d$. Consider the family of functions $u_{j_1j_2..j_{2k+1}}(n) := \frac{n_{j_1}..n_{j_{2k+1}}}{|n|^{4k+2}}$ for $n \neq 0$ and $u_{j_1j_2..j_{2k+1}}(0):=0$. Then we have
\begin{equation}\label{3.15}
    \sum_{n \in \Z^d} \frac{|u(n)|^2}{|n|^{4k+2}} = \sum_{1 \leq j_1,..,j_{2k+1} \leq d} \sum_{n \in \Z^d} |u_{j_1..j_{2k+1}}(n)|^2 =  \sum_{1 \leq j_1,..,j_{2k+1} \leq d} \int_{Q_d} |\widehat{u_{j_1..j_{2k+1}}}(x)|^2 dx.
\end{equation}
Using the inversion formula for the Fourier transform, we obtain 
\begin{equation}\label{3.16}
    \widehat{D_j u}(x) = (1-e^{-ix_j})\widehat{u}(x) \hspace{5pt} \text{and} \hspace{5pt} \widehat{\Delta u}(x) = 4 \sum_{j=1}^d \sin^2(x_j/2) \widehat{u}(x).
\end{equation}
Expression \eqref{3.16} along with Parseval's identity yields 
\begin{equation}\label{3.17}
    \begin{split}
        \sum_j \sum_{n \in \Z^d} |D_j \Delta^k u(n)|^2 &=\sum_j \int_{Q_d} |\widehat{D_j \Delta^k u}|^2 dx\\
        &= 4^{2k+1}\int_{Q_d} |\widehat{u}|^2 \Big(\sum_j \sin^2(x_j/2)\Big)^{2k+1}\\
        &= 4^{2k+1} \int_{Q_d} \Big|\sum_{1 \leq j_1,...,j_{2k+1}\leq d} \partial_{x_{j_{2k+1}}}..\partial_{x_{j_1}}\widehat{u_{j_1...j_{2k+1}}}\Big|^2  \omega^{2k+1}.   
    \end{split}
\end{equation}
We used $\sum_{1 \leq j_1,..,j_{2k+1}\leq d}\partial_{x_{j_{2k+1}}}..\partial_{x_{j_1}}\widehat{u_{j_1...j_{2k+1}}} = (-i)^{2k+1}\widehat{u}(x)$ in the last identity. Using the inversion formula and integration by parts, we further notice that for fixed $j_2,.., j_{2k+1}$
\begin{align*}
    \partial_{x_k} \widehat{u_{jj_2..j_{2k+1}}} = \partial_{x_j} \widehat{u_{kj_2..j_{2k+1}}}.
\end{align*}
This implies that $u_{j_1j_2..j_{2k+1}} = \partial_{x_{j_1}}\varphi_{j_2..j_{2k+1}}$ for some smooth function $\varphi_{j_2...j_{2k+1}}$ with zero average. Furthermore, along with its zero average, the periodicity of $\widehat{u_{j_1..j_{2k+1}}}$ implies that $\varphi_{j_2..j_{2k+1}}$ is $2\pi$-periodic in each variable. Since, $u_{j_1j_2..j_{2k+1}}$ is symmetric w.r.t. to $j_1$ and $j_2$ we get $\partial_{x_k}\varphi_{jj_3..j_{2k+1}} = \partial_{x_j} \varphi_{kj_3..j_{2k+1}}$. This implies that $\varphi_{j_2..j_{2k+1}}= \partial_{x_{j_2}}\xi_{j_3..j_{2k+1}}$ for some smooth $2\pi$-periodic function $\xi_{j_3...j_{2k+1}}$ whose average is zero. Using this argument iteratively, we get a smooth $2\pi$-periodic function $\psi(x)$ with zero average such that
\begin{equation}\label{3.18}
    \widehat{u_{j_1..j_{2k+1}}}(x) = \partial_{x_{j_{2k+1}}}...\partial_{x_{j_1}} \psi(x).
\end{equation}
Using \eqref{3.18} in \eqref{3.15} and \eqref{3.17}, we get expressions \eqref{3.9} and \eqref{3.10} respectively.
\end{proof}

\begin{lemma}\label{lem3.2}
Let $k \geq 0$ be an integer. Let $ u\in C_c(\Z^d)$ with $u(0)=0$. There exists $\psi \in C^\infty(\overline{Q_d})$, all of whose derivatives are $2\pi$-periodic in each variable and which has zero average, such that
\begin{equation}\label{3.19}
    \sum_{n \in \Z^d} \frac{|u(n)|^2}{|n|^{4k}} = \int_{Q_d} |\Delta^k \psi(x)|^2 dx,
\end{equation}
and 
\begin{equation}\label{3.20}
    \sum_{n \in \Z^d} |\Delta^k u(n)|^2 = 4^{2k}\int_{Q_d} |\Delta^{2k} \psi(x)|^2 \omega(x)^{2k}dx,
\end{equation}
where 
\begin{equation}\label{3.21}
    \omega(x) := \sum_{j=1}^d \sin^2(x_j/2).
\end{equation} 
\end{lemma}

\begin{proof}
The proof of this Lemma follows the proof of Lemma \ref{lem3.1} step by step but we include the proof here for the sake of completeness.

Let $1 \leq j_1, j_2,.., j_{2k} \leq d$. Consider the family of functions $u_{j_1j_2..j_{2k}}(n) := \frac{n_{j_1}..n_{j_{2k}}}{|n|^{4k}}$ for $n \neq 0$ and $u_{j_1j_2..j_{2k}}(0):=0$. Then we have
\begin{equation}\label{3.22}
    \sum_{n \in \Z^d} \frac{|u(n)|^2}{|n|^{4k}} = \sum_{1 \leq j_1,..,j_{2k} \leq d} \sum_{n \in \Z^d} |u_{j_1..j_{2k}}(n)|^2 =  \sum_{1 \leq j_1,..,j_{2k} \leq d} \int_{Q_d} |\widehat{u_{j_1..j_{2k}}}(x)|^2 dx.
\end{equation}
Using the inversion formula for the Fourier transform we obtain 
\begin{equation}\label{3.23}
    \widehat{\Delta u}(x) = 4 \sum_{j=1}^d \sin^2(x_j/2) \widehat{u}(x).
\end{equation}
Expressions \eqref{3.23} along with Parseval's identity yields
\begin{equation}\label{3.24}
    \begin{split}
        \sum_{n \in \Z^d} |\Delta^k u(n)|^2 = \int_{Q_d} |\widehat{\Delta^k u}|^2 dx &= 4^{2k}\int_{Q_d} |\widehat{u}|^2 \Big(\sum_j \sin^2(x_j/2)\Big)^{2k}\\
        &= 4^{2k} \int_{Q_d} \Big|\sum_{1 \leq j_1,...,j_{2k}\leq d} \partial_{x_{j_{2k}}}..\partial_{x_{j_1}}\widehat{u_{j_1...j_{2k}}}\Big|^2  \omega^{2k}.   
    \end{split}
\end{equation}
We used $\sum_{1 \leq j_1,..,j_{2k}\leq d}\partial_{x_{j_{2k}}}..\partial_{x_{j_1}}\widehat{u_{j_1...j_{2k}}} = (-i)^{2k}\widehat{u}(x)$ in the last identity. Using inversion formula and integration by parts we further notice that for fixed $j_2,.., j_{2k}$
\begin{align*}
    \partial_{x_k} \widehat{u_{jj_2..j_{2k}}} = \partial_{x_j} \widehat{u_{kj_2..j_{2k}}}.
\end{align*}
This implies that $u_{j_1j_2..j_{2k}} = \partial_{x_{j_1}}\varphi_{j_2..j_{2k}}$ for some smooth function $\varphi_{j_2...j_{2k}}$ with zero average. Furthermore, the periodicity of $\widehat{u_{j_1..j_{2k}}}$ along with its zero average implies that $\varphi_{j_2..j_{2k}}$ is $2\pi$-periodic in each variable. Since, $u_{j_1j_2..j_{2k}}$ is symmetric w.r.t. to $j_1$ and $j_2$ we get $\partial_{x_k}\varphi_{jj_3..j_{2k}} = \partial_{x_j} \varphi_{kj_3..j_{2k}}$. This implies that $\varphi_{j_2..j_{2k}}= \partial_{x_{j_2}}\xi_{j_3..j_{2k}}$ for some smooth $2\pi$-periodic function $\xi_{j_3...j_{2k}}$ whose average is zero. Using this argument iteratively we get a smooth $2\pi$-periodic function $\psi(x)$ with zero average such that
\begin{equation}\label{3.25}
    \widehat{u_{j_1..j_{2k}}}(x) = \partial_{x_{j_{2k}}}...\partial_{x_{j_1}} \psi(x).
\end{equation}
Using \eqref{3.25} in \eqref{3.22} and \eqref{3.24}, we get expressions \eqref{3.19} and \eqref{3.20} respectively.
\end{proof}

\section{Hardy-type inequalities on a torus}\label{sec: continuous Hardy inequality}
In this section we prove Hardy-type inequalities for the operators $\Delta^m$ and $\nabla (\Delta^m)$ for non-negative integers $m$ on the torus $Q_d$. These inequalities are proved for functions having zero average. We could not locate an occurrence of Hardy-type inequalities for these class of functions. We begin by proving a weighted Hardy inequality for the gradient.  \newpage

\begin{theorem}[Weighted Hardy inequality]\label{thm3.4}
Let $\psi \in C^\infty(\overline{Q_d})$ all of whose derivatives are $2\pi$-periodic in each variable. Further assume that $\psi$ has zero average. Let $k$ be a non-positive integer. Then for $d > -2k +2$ we have
\begin{equation}\label{3.26}
    H(k, d) \int_{Q_d} |\psi(x)|^2 \omega(x)^{k-1} dx \leq \int_{Q_d} |\nabla \psi(x)|^2 \omega(x)^{k} dx,
\end{equation}
where $\omega(x) := \sum_j \sin^2(x_j/2)$, 
\begin{equation}\label{3.27}
    H(k, d)^{-1} := \sum_{j=0}^{-k} d^j C_1(k + j, d) \prod_{i=0}^{j-1}C_2(k +i, d) + d^{-k}\prod_{i=0}^{-k}C_2(k+i, d),
\end{equation}
\begin{equation}
    C_1(k, d) := 16/(d+2k-2)^2 \hspace{9pt} \text{and} \hspace{9pt} C_2(k, d) := (3d+2k-2)/d(d+2k-2).
\end{equation}
\end{theorem}
\begin{remark}
Note that $C_1(k, d) \sim 1/d^2 , C_2(k,d) \sim 1/d$. This implies that $H(k, d)\sim d$ as $ d \rightarrow \infty$.
\end{remark}
\begin{proof}
The proof goes via expansion of squares. Let $F = (F_1,..,F_d)$ be a smooth real-valued vector field on $Q_d$ which is $2\pi-$ periodic in each variable. Let $\omega_{\epsilon} := \omega+\epsilon^2$ for $\epsilon \neq 0$. Consider the following square, for non-positive integer $k$ and a real parameter $\beta$
\begin{align*}
    |\omega_\epsilon^{k/2} \nabla \psi  + \beta \omega_\epsilon^{k/2} \psi F|^2 &= \omega_\epsilon^{k}|\nabla \psi|^2 + \beta^2 \omega_\epsilon^{k}|F|^2 |\psi|^2 + 2\omega_\epsilon^{k} \beta  \text{Re}(\overline{\psi}\nabla \psi \cdot F) \\
    &= \omega_\epsilon^{k}|\nabla \psi|^2 + \beta^2 \omega_\epsilon^{k}|F|^2 |\psi|^2 + \beta \omega_\epsilon^{k} F \cdot \nabla |\psi|^2.
\end{align*}
Integrating both sides and applying integration by parts, we obtain
\begin{align*}
    0 \leq \int_{Q_d}|\omega_\epsilon^{k/2} \nabla \psi  + \beta \omega_\epsilon^{k/2} \psi F|^2 dx &= \int_{Q_d} \omega_\epsilon^{k} |\nabla \psi|^2 dx \\
    & + \int_{Q_d} (\beta^2 \omega_\epsilon^{k}|F|^2 - \beta \text{div}(\omega_\epsilon^{k}F))|\psi|^2 dx, 
\end{align*}
which implies 
\begin{equation}\label{3.29}
    \int_{Q_d} |\nabla \psi|^2 \omega_\epsilon^{k} dx \geq \int_{Q_d} \Big(\beta\text{div}(\omega_\epsilon^{k}F)-\beta^2 \omega_\epsilon^{k}|F|^2\Big)|\psi|^2 dx.
\end{equation}
Let $F_j = \frac{\sin x_j}{\omega_\epsilon}$, then 
\begin{align*}
    |F|^2 &= \frac{4}{\omega_\epsilon^2}(\omega-\sum_j \sin^4(x_j/2)) \hspace{5pt}\text{and}\\
    \text{div}(\omega_\epsilon^{k}F)&= d\omega_\epsilon^{k-1} + 2(k-1)\omega\omega_\epsilon^{k-2} -2\omega\omega_\epsilon^{k-1} - 2(k-1)\omega_\epsilon^{k-2}\sum_i \sin^4(x_i/2).
\end{align*}
Using above expressions we obtain 
\begin{align*}
    \beta\text{div}(\omega_\epsilon^{k}F)-\beta^2 \omega_\epsilon^{k}|F|^2 &= d\beta\omega_\epsilon^{k-1} + \Big(2\beta(k-1)-4\beta^2\Big)\omega \omega_\epsilon^{k-2}\\
    &- \Big(2\beta(k-1)-4\beta^2\Big)\omega_\epsilon^{k-2}\sum_j \sin^4(x_j/2) -2\beta \omega\omega_\epsilon^{k-1}.
\end{align*}
Plugging the above identity in \eqref{3.29}, and taking limit $\epsilon \rightarrow 0$, we get for $ d>-2k + 2$
\begin{align*}
    \int_{Q_d} |\nabla \psi|^2 \omega^{k} dx &\geq \Big(-4\beta^2 + \beta(d+2k-2)\Big) \int_{Q_d} |\psi|^2 \omega^{k-1} dx\\
    &- \int_{Q_d}\Bigg(\Big(2\beta(k-1) -4\beta^2\Big)\sum_j \sin^4(x_j/2)/\omega^2 + 2\beta \Bigg) \omega^{k}|\psi|^2 dx.
\end{align*}
Choosing $\beta = (d+2k-2)/8$ with the aim of maximizing $-4\beta^2 + \beta(d+2k-2)$, and using the estimate $d\sum_j \sin^4(x_j/2) \geq \omega^2$, we obtain
\begin{align*}
   \int_{Q_d} |\psi|^2 \omega^{k-1} dx &\leq \frac{16}{(d+2k-2)^2} \int_{Q_d}|\nabla \psi|^2 \omega^{k} dx + \frac{3d+2k-2}{d(d+2k-2)}\int_{Q_d}|\psi|^2 \omega^{k} dx\\
   &=: C_1(k, d) \int_{Q_d}|\nabla \psi|^2 \omega^{k} dx + C_2(k, d) \int_{Q_d}|\psi|^2 \omega^{k} dx.
\end{align*}
Applying the above inequality inductively w.r.t. $k$ and using $\omega(x) \leq d$, we get 
\begin{align*}
    \int_{Q_d} |\nabla \psi|^2 \omega^{k} dx &\leq \sum_{j=0}^{-k} C_1(k + j, d) \prod_{i=0}^{j-1}C_2(k+i, d) \int_{Q_d} |\nabla \psi|^2 \omega^{k+j} dx\\
    &+ \prod_{i=0}^{-k}C_2(k+i, d)\int_{Q_d} |\psi|^2 dx\\
    & \leq \sum_{j=0}^{-k} d^j C_1(k + j, d) \prod_{i=0}^{j-1}C_2(k+i, d) \int_{Q_d} |\nabla \psi|^2 \omega^{k} dx \\
    &+ \prod_{i=0}^{-k}C_2(k+i, d)\int_{Q_d} |\psi|^2 dx.
\end{align*}
Finally, using the Poincare-Friedrichs inequality: $\int_{Q_d} |\psi|^2 \leq \int_{Q_d} |\nabla \psi|^2$ along with $\omega(x)^k \geq d^k$, we get the desired result.
\end{proof}

\begin{corollary}[Hardy inequality]\label{cor3.3}
Let $\psi \in C^\infty(\overline{Q_d})$ all of whose derivatives are $2\pi$-periodic in each variable. Further assume that $\psi$ has zero average. Then, for $d \geq 3$, we have
\begin{equation}\label{3.30}
    \frac{d(d-2)^2}{3d^2+8d+4} \int_{Q_d} \frac{|\psi(x)|^2}{\sum_j \sin^2(x_j/2)} dx \leq \int_{Q_d} |\nabla \psi(x)|^2 dx.
\end{equation}
\end{corollary}

\begin{proof}
Applying Theorem \ref{thm3.4} for $k=0$ and noting that $H(0, d)=\frac{d(d-2)^2}{3d^2+8d+4} $, the proof is complete.
\end{proof}

\begin{remark}
A natural question to ask is whether the constant in \eqref{3.30} is sharp. We certainly believe it's not the best possible one, since some crude estimates were used in the proof of Theorem \ref{thm3.4}. However, it would to good to understand how far it is from being sharp.
\end{remark}
In the next lemma, we prove a two-parameter family of inequalities from which we will derive a weighted Rellich and Hardy-Rellich type inequalities on the torus. 
\begin{lemma}\label{lem3.3}
Let $\psi \in C^\infty(\overline{Q_d})$ all of whose derivatives are $2\pi$-periodic in each variable. Let $\alpha \leq 0$ and $ d> -4\alpha+4$. Then for real parameters $\beta, \gamma$ satisfying $\beta^2 -\beta(2\alpha-1) \geq 0$ we have 
\begin{equation}\label{3.31}
    \begin{split}
        \int_{Q_d} \omega(x)^{2\alpha} |\Delta \psi(x)|^2 dx &\geq (2\gamma -\beta(d+4\beta-4\alpha)) \int_{Q_d} \omega(x)^{2\alpha-1}|\nabla \psi(x)|^2 dx\\
        &+ \frac{\gamma}{2}\Big((2\beta-2\alpha+1)(d+4\alpha-4) -2\gamma \Big) \int_{Q_d} \omega(x)^{2\alpha-2}|\psi(x)|^2 dx\\
        &+ E(x),
    \end{split}
\end{equation}
where $\omega (x) := \sum_i \sin^2(x_i/2)$ and 
\begin{equation}\label{3.32}
    \begin{split}
        E(x) &:= 2\beta \int_{Q_d} \Big(1 +(2\beta -2\alpha+1)\sum_i \sin^4(x_i/2)/\omega^2 \Big)\omega(x)^{2\alpha} |\nabla \psi(x)|^2 dx\\
        &-4\beta \int_{Q_d}\omega(x)^{2\alpha} \sum_i \frac{\sin^2(x_i/2)}{\omega}  |\partial_{x_i}\psi(x)|^2 dx\\
        &-\gamma(2\beta -2\alpha +1)\int_{Q_d} \Big(1 + 2(\alpha-1)\sum_i\sin^4(x_i/2)/\omega^2\Big)\omega(x)^{2\alpha-1}|\psi(x)|^2 dx.
    \end{split}
\end{equation}    
\end{lemma}

\begin{proof}
Let $F = (F_1,.., F_d)$ and $f$ be a smooth $2\pi$-periodic real-valued vector field and a scalar function respectively. Let $\omega_\epsilon := \omega+ \epsilon^2$ and $\beta, \gamma \in \R$. We expand the following square
\begin{equation}\label{3.33}
    \begin{split}
        Q(\psi) :=|\omega_\epsilon^\alpha \Delta \psi  + \beta \omega_\epsilon^\alpha F \cdot \nabla \psi + \gamma \omega_\epsilon^\alpha f \psi|^2 &= \omega_\epsilon^{2\alpha}|\Delta \psi|^2 + \beta^2 \omega_\epsilon^{2\alpha}|F \cdot \nabla \psi|^2 + \gamma^2 \omega_\epsilon^{2\alpha}f^2 |\psi|^2\\
        &+ 2\beta \omega_\epsilon^{2\alpha} \text{Re}(F \cdot \nabla \psi \Delta \overline{\psi})+ 2\gamma \omega_\epsilon^{2\alpha} \text{Re}(f \overline{\psi} \Delta \psi) \\
        &+ 2\beta \gamma \omega_\epsilon^{2\alpha} \text{Re}(f F\cdot \overline{\psi} \nabla \psi). 
    \end{split}
\end{equation}
Applying chain rule and integration by parts we obtain
\begin{align*}
    2 \beta \gamma \int_{Q_d}\omega_\epsilon^{2\alpha} \text{Re}(f F\cdot \overline{\psi} \nabla \psi) dx &= -\beta \gamma \int_{Q_d} \text{div}(\omega_\epsilon^{2\alpha}fF)|\psi|^2 dx,
\end{align*} 
\begin{align*}
    2\gamma \int_{Q_d} \omega_{\epsilon}^{2\alpha} \text{Re}(f \overline{\psi} \Delta \psi) dx &= -2\gamma \int_{Q_d} (\omega_\epsilon^{2\alpha}f)|\nabla \psi|^2 dx -2\gamma \int_{Q_d} \partial_{x_i} (\omega_\epsilon^{2\alpha} f) \text{Re}(\overline{\psi} \partial_{x_i}\psi) dx\\
    &=-2\gamma \int_{Q_d}(\omega_\epsilon^{2\alpha}f)|\nabla \psi|^2 dx + \gamma \int_{Q_d} \Delta(\omega_\epsilon^{2\alpha}f)|\psi|^2 dx,
\end{align*}
and 
\begin{equation}\label{3.34}
    \begin{split}
        2\beta \int_{Q_d} \omega_\epsilon^{2\alpha} \text{Re}(F \cdot \nabla \psi \Delta \overline{\psi}) dx 
        &=\beta \int_{Q_d} \text{div}(\omega_\epsilon^{2\alpha}F) |\nabla \psi|^2 dx \\
        &- 2\beta \sum_{i, j} \int_{Q_d} \partial_{x_j}(\omega_\epsilon^{2\alpha}F_i) \partial_{x_i} \psi \partial_{x_j} \overline{\psi} dx.     
    \end{split}
\end{equation}
Let $F_i(x) := \sin x_i/\omega_\epsilon$ and $f(x) := 1/\omega_\epsilon $. Then we have
\begin{align*}
    \partial_{x_j}(\omega_\epsilon^{2\alpha}F_i) &= \omega_\epsilon^{2\alpha-1}(1-2\sin^2(x_i/2)) \delta_{ij} + (2\alpha-1)/2\omega_\epsilon^{2\alpha-2}\sin x_i \sin x_j,\\
    \partial_{x_i}(\omega_\epsilon^{2\alpha} f F_i) &= \omega_\epsilon^{2\alpha-2}(1-2\sin^2(x_i/2)) + 2(2\alpha-2)\omega_\epsilon^{2\alpha-3}(\sin^2(x_i/2)-\sin^4(x_i/2)),\\
    \partial^2_{x_i^2} (\omega_\epsilon^{2\alpha}f)&= (2\alpha-1)/2\omega_\epsilon^{2\alpha-2}(1-2\sin^2(x_i/2)) \\
    &+ (2\alpha-1)(2\alpha-2)\omega_\epsilon^{2\alpha-3}(\sin^2(x_i/2)-\sin^4(x_i/2)).
\end{align*}
Plugging the above identities in \eqref{3.34}, we get,
\begin{equation}\label{3.35}
    \begin{split}
        2 \beta \gamma \int_{Q_d}\omega_\epsilon^{2\alpha} \text{Re}(f F\cdot \overline{\psi} \nabla \psi) dx 
        &= -\beta \gamma\int_{Q_d}  \Big(d+ (4\alpha -4)\omega/\omega_\epsilon \Big) \omega_\epsilon^{2\alpha-2}|\psi|^2 dx\\
        &+ 2\beta \gamma \int_{Q_d} \Big(\omega/\omega_\epsilon + 2(\alpha-1)\sum_i \sin^4(x_i/2)/\omega_\epsilon^2\Big)\omega_\epsilon^{2\alpha-1}|\psi|^2 dx,
    \end{split}
\end{equation}
\begin{equation}\label{3.36}
    \begin{split}
        2\gamma \int_{Q_d} \omega_{\epsilon}^{2\alpha} &\text{Re}(f \overline{\psi} \Delta \psi) dx = -\gamma(2\alpha-1)\int_{Q_d}\Big(\omega/\omega_\epsilon +  2(\alpha-1)\sum_i \sin^4(x_i/2)/\omega_\epsilon^2\Big)\omega_\epsilon^{2\alpha-1}|\psi|^2 dx\\
        & -2\gamma \int_{Q_d}\omega_\epsilon^{2\alpha-1}|\nabla \psi|^2 dx + \gamma(2\alpha-1)\int_{Q_d} \Big(d/2 + 2(\alpha-1)\omega/\omega_\epsilon \Big)\omega_\epsilon^{2\alpha-2}|\psi|^2 dx,
    \end{split}
\end{equation}
and 
\begin{equation}\label{3.37}
    \begin{split}
       &2\beta \int_{Q_d} \omega_\epsilon^{2\alpha} \text{Re}(F \cdot \nabla \psi \Delta \overline{\psi}) dx\\
       &= \int_{Q_d}\Big(d \beta + 2\beta(2\alpha-1)\omega/\omega_\epsilon-2\beta\Big) \omega_\epsilon^{2\alpha-1}|\nabla \psi|^2 dx\\
       &- \beta(2\alpha-1)\int_{Q_d}\omega_\epsilon^{2\alpha} |F \cdot \nabla \psi|^2 dx + 4\beta \int_{Q_d}\omega_\epsilon^{2\alpha} \sum_i \frac{\sin^2(x_i/2)}{\omega_\epsilon}  |\partial_{x_i}\psi|^2 dx \\
       &- 2\beta \int_{Q_d} \Big(\omega/\omega_\epsilon +(2\alpha-1)\sum_i \sin^4(x_i/2)/\omega_\epsilon^2 \Big)\omega_\epsilon^{2\alpha} |\nabla \psi|^2 dx.  
    \end{split}
\end{equation}
Integrating both sides in \eqref{3.33}, and using identities \eqref{3.35}-\eqref{3.37} we obtain
\begin{align*}
    \int_{Q_d}Q(\psi) dx &= \int_{Q_d} \omega_\epsilon^{2\alpha} |\Delta \psi|^2 dx + \int_{Q_d} \Big(-2\gamma + \beta \big(d-2 + (4\alpha-2)\omega/\omega_\epsilon\big)\Big) \omega_{\epsilon}^{2\alpha-1}|\nabla \psi|^2 dx\\
    &+ \int_{Q_d} \frac{\gamma}{2}\Big(2\gamma + (2\alpha-2\beta-1)(d+4(\alpha-1)\omega/\omega_\epsilon)\Big) \omega^{2\alpha-2} |\psi|^2 dx\\
    &+ \Big(\beta^2 -\beta(2\alpha-1)\Big)\int_{Q_d}\omega_\epsilon^{2\alpha} |F \cdot \nabla \psi|^2 dx + E(x), 
\end{align*}
where 
\begin{align*}
    E(x) &:= 4\beta \int_{Q_d}\omega_\epsilon^{2\alpha} \sum_i \frac{\sin^2(x_i/2)}{\omega_\epsilon}  |\partial_{x_i}\psi|^2 dx \\
    &- 2\beta \int_{Q_d} \Big(\omega/\omega_\epsilon +(2\alpha-1)\sum_i \sin^4(x_i/2)/\omega_\epsilon^2 \Big)\omega_\epsilon^{2\alpha} |\nabla \psi|^2 dx\\
    &+ \gamma(2\beta -2\alpha +1)\int_{Q_d} \Big(\omega/\omega_\epsilon + 2(\alpha-1)\sum_i\sin^4(x_i/2)/\omega_\epsilon^2\Big)\omega_\epsilon^{2\alpha-1}|\psi|^2 dx.
\end{align*}
Further using the non-negativity of $\int_{Q_d} Q(\psi) dx$ along with Cauchy's inequality($|F \cdot \nabla \psi| \leq |F||\nabla \psi|$), we get, for $\beta^2 - \beta(2\alpha-1) \geq 0$
\begin{equation}\label{3.38}
    \begin{split}
        \int_{Q_d} \omega_\epsilon^{2\alpha} |\Delta \psi|^2 dx & \geq \int_{Q_d} \Big(2\gamma - \beta (d-2+ 2(2\beta-2\alpha+1) \omega/\omega_\epsilon)\Big) \omega_{\epsilon}^{2\alpha-1}|\nabla \psi|^2 dx\\
        &+ \int_{Q_d} \frac{\gamma}{2}\Big((2\beta - 2\alpha+1)(d+4(\alpha-1)\omega/\omega_\epsilon) -2\gamma\Big) \omega^{2\alpha-2} |\psi|^2 dx\\
        &+ E_1(x),
    \end{split}
\end{equation}
where $E_1(x)$ is given by 
\begin{equation}\label{3.39}
    \begin{split}
        E_1(x) &:= -4\beta \int_{Q_d}\omega_\epsilon^{2\alpha} \sum_i \frac{\sin^2(x_i/2)}{\omega_\epsilon}  |\partial_{x_i}\psi|^2 dx \\
        &+ 2\beta \int_{Q_d} \Big(\omega/\omega_\epsilon +(2\beta -2\alpha+1)\sum_i \sin^4(x_i/2)/\omega_\epsilon^2 \Big)\omega_\epsilon^{2\alpha} |\nabla \psi|^2 dx\\
        &-\gamma(2\beta -2\alpha +1)\int_{Q_d} \Big(\omega/\omega_\epsilon + 2(\alpha-1)\sum_i\sin^4(x_i/2)/\omega_\epsilon^2\Big)\omega_\epsilon^{2\alpha-1}|\psi|^2 dx.    
    \end{split}
\end{equation}
Taking limit $\epsilon \rightarrow 0$ in \eqref{3.38}, we get for $d > -4\alpha+4$
\begin{equation}\label{3.40}
    \begin{split}
        \int_{Q_d} \omega^{2\alpha} |\Delta \psi|^2 dx &\geq (2\gamma -\beta(d+4\beta-4\alpha)) \int_{Q_d} \omega^{2\alpha-1}|\nabla \psi|^2 dx\\
        &+ \frac{\gamma}{2}\Big((2\beta-2\alpha+1)(d+4\alpha-4) -2\gamma \Big) \int_{Q_d} \omega^{2\alpha-2}|\psi|^2 dx + E_2(x),
    \end{split}
\end{equation}
where
\begin{equation}\label{3.41}
    \begin{split}
        E_2(x) &:= -4\beta \int_{Q_d}\omega^{2\alpha} \sum_i \frac{\sin^2(x_i/2)}{\omega}  |\partial_{x_i}\psi|^2 dx\\
        &+ 2\beta \int_{Q_d} \Big(1 +(2\beta -2\alpha+1)\sum_i \sin^4(x_i/2)/\omega^2 \Big)\omega^{2\alpha} |\nabla \psi|^2 dx\\
        &-\gamma(2\beta -2\alpha +1)\int_{Q_d} \Big(1 + 2(\alpha-1)\sum_i\sin^4(x_i/2)/\omega^2\Big)\omega^{2\alpha-1}|\psi|^2 dx.
    \end{split}
\end{equation}
\end{proof}
Next, we prove a weighted Hardy-Rellich type inequality on the torus $Q_d$.
\begin{theorem}[Weighted Hardy-Rellich inequality]\label{thm3.5}
Let $\psi \in C^\infty(\overline{Q_d})$ all of whose derivatives are $2\pi$-periodic in each variable. Further assume that $\psi$ has zero average. Let $k$ be a non-positive integer. Then for $d \geq -6k +8$ we have
\begin{equation}\label{3.42}
    HR(k, d) \int_{Q_d} |\nabla \psi(x)|^2 \omega(x)^{k-1} dx \leq \int_{Q_d} |\Delta \psi(x)|^2 \omega(x)^k dx,
\end{equation}
where $\omega(x) := \sum_j \sin^2(x_j/2)$, 
\begin{equation}\label{3.18}
    HR(k, d)^{-1} := \sum_{j=0}^{-k} d^j C_1(k + j, d) \prod_{i=0}^{j-1}C_2(k + i, d) + d^{-k}\prod_{i=0}^{-k}C_2(k+i, d),
\end{equation}
\begin{equation}
    C_1(k, d) := 16/(d-2k)^2 \hspace{9pt} \text{and} \hspace{9pt} C_2(k, d) := (3d-2k+4)/d(d-2k).
\end{equation}
\end{theorem}

\begin{remark}
Note that $C_1(k, d) \sim 1/d^2$ and $C_2(k, d) \sim 1/d$, which implies that $HR(k, d) \sim d$ as $d \rightarrow \infty$.
\end{remark}

\begin{proof}
We begin by writing down inequality \eqref{3.31}, with $\alpha \leq 0$, $\beta, \gamma$ being real numbers satisfying $\beta^2 -\beta(2\alpha-1) \geq 0$ and $d >-4\alpha+4$
\begin{equation}\label{3.45}
    \begin{split}
        \int_{Q_d} \omega^{2\alpha} |\Delta \psi|^2 dx &\geq (2\gamma -\beta(d+4\beta-4\alpha)) \int_{Q_d} \omega^{2\alpha-1}|\nabla \psi|^2 dx\\
        &+ \frac{\gamma}{2}\Big((2\beta-2\alpha+1)(d+4\alpha-4) -2\gamma \Big) \int_{Q_d} \omega^{2\alpha-2}|\psi|^2 dx + E(x).
    \end{split}
\end{equation}
Next, we choose $\gamma =0$ and $\beta = -(d-4\alpha)/8$. 
Note that for this choice of $\gamma$, the second term on the RHS of \eqref{3.45} vanishes and $\beta = -(d-4\alpha)/8 $ maximizes the coefficient of the first term on the RHS of \eqref{3.45} after taking $\gamma =0$. Also note that condition $\beta^2-\beta(2\alpha-1) \geq 0$ gives an additional constraint $d \geq -12\alpha+8$. For this choice of parameters \eqref{3.45} becomes
\begin{equation}\label{3.46}
    \int_{Q_d} \omega^{2\alpha}|\Delta \psi|^2 dx \geq \frac{(d-4\alpha)^2}{16} \int_{Q_d} \omega^{2\alpha-1} |\nabla \psi|^2 dx + E(x),
\end{equation}
and $E(x)$ becomes 
\begin{align*}
    E(x)
    &= -\frac{(d-4\alpha)}{4} \int_{Q_d} \omega^{2\alpha} |\nabla \psi|^2 dx + \frac{(d-4\alpha)(d+4\alpha-4)}{16} \int_{Q_d} \frac{\sum_i \sin^4(x_i/2)}{\omega^2} \omega^{2\alpha} |\nabla \psi|^2 dx\\
    &+(d-4\alpha)/2 \int_{Q_d}\omega^{2\alpha} \sum_i \frac{\sin^2(x_i/2)}{\omega}  |\partial_{x_i}\psi|^2 dx\\
    &\geq -\frac{(d-4\alpha)(3d-4\alpha +4)}{16 d}\int_{Q_d} \omega^{2\alpha}|\nabla \psi|^2 dx.\\
\end{align*}
In the last inequality we used $d \sum_i \sin^4(x_i/2) \geq \omega^2$ and $\alpha \leq 0$ to bound the last integral by zero from below. Using this lower bound on $E(x)$ in \eqref{3.46} and taking $\alpha := k/2$ we obtain 
\begin{equation}\label{3.47}
    \int_{Q_d} |\nabla \psi|^2 \omega^{k-1} dx \leq  \frac{16}{(d-2k)^2}\int_{Q_d} |\Delta \psi|^2 \omega^{k} dx + \frac{3d-2k+4}{d(d-2k)}\int_{Q_d}  |\nabla \psi|^2 \omega^{k} dx.
\end{equation}
Applying inequality \eqref{3.47} inductively w.r.t $k$ and using the bound $w(x) \leq d$ we get

\begin{align*}
    \int_{Q_d} |\nabla \psi|^2 \omega^{k-1} dx &\leq \sum_{j=0}^{-k} d^j C_1(k + j, d) \prod_{i=0}^{j-1}C_2(k+i, d) \int_{Q_d} |\Delta \psi|^2 \omega^k dx \\
    &+ \prod_{i=0}^{-k}C_2(k+i, d)\int_{Q_d} |\nabla \psi|^2 dx,
\end{align*}
where $C_1(k, d) := 16/(d-2k)^2$ and $C_2(k, d) := (3d-2k+4)/d(d-2k)$.\\
Using $\omega^{k} \geq d^{k} $ and $\int_{Q_d}|\nabla \psi|^2 dx \leq \int_{Q_d}|\Delta \psi|^2 dx$ in the above inequality completes the proof.
\end{proof}

\begin{corollary}[Hardy-Rellich inequality]\label{cor3.4}
Let $\psi \in C^\infty(\overline{Q_d})$ all of whose derivatives are $2\pi$-periodic in each variable. Further assume that $\psi$ has zero average. Then for $d \geq 8$, we have
\begin{equation}\label{3.48}
    \frac{d^2}{3d+20} \int_{Q_d} \frac{|\nabla \psi(x)|^2}{\sum_j \sin^2(x_j/2)} dx \leq \int_{Q_d} |\Delta \psi(x)|^2 dx.
\end{equation}
\end{corollary}

\begin{proof}
Inequality \eqref{3.48} follows directly from Theorem \ref{thm3.5} by taking $k=0$ and observing that $HR(0, d) = d^2/(3d+20)$.
\end{proof}

Next, we derive a Rellich inequality on the torus $Q_d$ from Lemma \ref{lem3.3}.
\begin{theorem}[Weighted Rellich inequality]\label{thm3.6}
Let $\psi \in C^\infty(\overline{Q_d})$ all of whose derivatives are $2\pi$-periodic in each variable. Further assume that $\psi$ has zero average. Let $k$ be a non-positive integer. Then, for $d > -2k +4$, we have
\begin{equation}\label{3.49}
    R(k, d) \int_{Q_d} |\psi(x)|^2 \omega(x)^{k-2} dx \leq \int_{Q_d} |\Delta \psi(x)|^2 \omega(x)^k dx,
\end{equation}
where $\omega(x) := \sum_j \sin^2(x_j/2)$, 
\begin{equation}\label{3.50}
    R(k, d) :=  \frac{(d-2k)^2(d+2k-4)^2}{256\Bigg(1+HR(k, d)^{-1}\Big(d C_1(k/2, d)+d C_2(k/2, d)H(k, d)^{-1}\Big)\Bigg)},
\end{equation}
non-negative constants $C_1(\alpha, d), C_2(\alpha, d)$ are given by
\begin{align*}
    C_1(\alpha, d) &:= \frac{2\beta(d-2\beta + 2\alpha-1)}{d},\\ C_2(\alpha, d) &:= \frac{\beta(d+4\beta-4\alpha)(d+2\alpha-2)(2\beta-2\alpha+1)}{2d},
\end{align*}
and $\beta := 1/8(-4 + 8\alpha + \sqrt{2}\sqrt{d^2 -4d + 16\alpha^2 -16 \alpha + 8}) \geq 0$. 
\end{theorem}

\begin{remark}
Note that $R(k, d) \sim d^2$ as $d \rightarrow \infty$, since $H(k, d)\sim d, HR(k, d) \sim d$ and $C_1(k/2, d) \sim d$, $C_2(k/2, d)\sim d^3$. 
\end{remark}
\begin{proof}
Let $\alpha \leq 0$ and $\beta, \gamma$ be real numbers satisfying $\beta^2-\beta(2\alpha-1) \geq 0$. Then we have inequality \eqref{3.31}, that is
\begin{equation}\label{3.51}
    \begin{split}
        \int_{Q_d} \omega^{2\alpha} |\Delta \psi|^2 dx &\geq (2\gamma -\beta(d+4\beta-4\alpha)) \int_{Q_d} \omega^{2\alpha-1}|\nabla \psi|^2 dx\\
        &+ \frac{\gamma}{2}\Big((2\beta-2\alpha+1)(d+4\alpha-4) -2\gamma \Big) \int_{Q_d} \omega^{2\alpha-2}|\psi|^2 dx + E(x),
    \end{split}
\end{equation}
where $E(x)$ is as defined in \eqref{3.32}.  We first choose $2\gamma = \beta(d+4\beta-4\alpha)$ for which the first term in RHS of \eqref{3.51} vanishes. We further choose $\beta = 1/8(-4 + 8\alpha + \sqrt{2}\sqrt{d^2 -4d + 16\alpha^2 -16 \alpha + 8})$ with the aim of maximizing the coefficient of the second term in the RHS of \eqref{3.51}. This choice of parameters implies that
\begin{equation}\label{3.52}
   \int_{Q_d} |\Delta \psi|^2 \omega^{2\alpha}  dx  \geq \frac{(d-4\alpha)^2(d+4\alpha-4)^2}{256}\int_{Q_d}  |\psi|^2 \omega^{2\alpha-2}dx + E(x).
\end{equation}
Note that $\beta, \gamma \geq 0$ and $\beta^2-\beta(2\alpha-1)\geq 0$ under the condition $d > -4\alpha+4$. Using the estimates $\omega \geq \sin^2(x_i/2)$ and $d \sum_i \sin^4(x_i/2) \geq \omega^2$, we get
\begin{align*}
    E(x) \geq &-\frac{2\beta(d -2\beta+2\alpha-1)}{d}\int_{Q_d}\omega^{2\alpha}|\nabla \psi|^2 dx\\
    &- \frac{\gamma(d+ 2\alpha-2)(2\beta -2\alpha +1)}{d} \int_{Q_d}\omega^{2\alpha-1}|\psi|^2 dx.
\end{align*}
Above inequality along with \eqref{3.52} gives 
\begin{equation}
    \begin{split}
        \frac{(d-4\alpha)^2(d+4\alpha-4)^2}{256} \int_{Q_d} |\psi|^2 \omega^{2\alpha-2} dx \leq \int_{Q_d} |\Delta \psi|^2 \omega^{2\alpha} dx &+ C_1(\alpha, d) \int_{Q_d} |\nabla \psi|^2 \omega^{2\alpha} dx \\
        &+ C_2(\alpha, d) \int_{Q_d} |\psi|^2 \omega^{2\alpha-1} dx,
    \end{split}
\end{equation}
where $C_1(\alpha, d) := 2\beta(d-2\beta+ 2\alpha-1)/d$ and $C_2(\alpha, d) := \gamma(d+2\alpha-2)(2\beta-2\alpha+1)/d$. Note that $d > -4\alpha+4$ implies $C_1(\alpha, d), C_2(\alpha, d) \geq 0$. From now on, we assume that $2\alpha \in \Z$. Using weighted Hardy inequality \eqref{3.26} and weighted Hardy-Rellich inequality \eqref{3.42} we obtain
\begin{equation}
    \begin{split}
    (d-4\alpha)^2(d+4\alpha-4)^2/256 \int_{Q_d} |\psi|^2 \omega^{2\alpha-2} dx \leq C(\alpha, d)\int_{Q_d} |\Delta \psi|^2 \omega^{2\alpha} dx,
    \end{split}
\end{equation}
where $C(\alpha, d) := 1+HR(2\alpha, d)^{-1}\Big(dC_1(\alpha, d)+dC_2(\alpha, d)H(2\alpha, d)^{-1}\Big)$. 

Taking $\alpha := k/2$ completes the proof.
\end{proof}

\begin{remark}
Note that Theorem \ref{thm3.6} with $k=0$ gives the Rellich inequality on torus $Q_d$, although unlike the Hardy and Hardy-Rellich inequalities \eqref{3.30}, \eqref{3.48}, the constant in the Rellich inequality has a messy expression.  
\end{remark}

In the next theorem we apply weighted Hardy and Rellich inequality \eqref{3.26}, \eqref{3.49} respectively to prove a Hardy type inequality for the operators $\Delta^m$ and $\nabla( \Delta^m)$.
\begin{theorem}\label{thm3.7}
Let $\psi \in C^\infty(\overline{Q_d})$ all of whose derivatives are $2\pi$-periodic in each variable. Further assume that $\psi$ has zero average. Let $k$ be a non-positive integer and $m$ be a non-negative integer.
\begin{enumerate}
    \item If $ d > -2k+4m$, then
    \begin{equation}\label{3.55}
        C(m, k, d)\int_{Q_d} |\psi(x)|^2 \omega(x)^{k-2m} dx \leq \int_{Q_d} |\Delta^m \psi(x)|^2 \omega(x)^{k} dx,
    \end{equation}
    where 
    \begin{equation}\label{3.56}
        C(m, k, d) := \prod_{i=0}^{m-1} R(k-2i, d),
    \end{equation}
    and $R(k, d)$ is the constant in the Rellich inequality \eqref{3.50}.
    \item If $d > -2k+4m+2$, then
    \begin{equation}\label{3.57}
        \widetilde{C}(m, k, d)\int_{Q_d} |\psi(x)|^2 \omega(x)^{k-(2m+1)} dx  \leq \int_{Q_d} |\nabla(\Delta^m \psi)(x)|^2 \omega(x)^{k} dx,
    \end{equation}
    where 
    \begin{equation}\label{3.58}
        \widetilde{C}(m, k, d) := H(k, d)\prod_{i=0}^{m-1}R(k-2i-1, d),
    \end{equation}
    and $H(k, d)$ is the constant in the Hardy inequality \eqref{3.27}.
\end{enumerate}
\end{theorem}

\begin{proof}
Applying the weighted Rellich inequality \eqref{3.49} with $\psi$ replaced by $\Delta^{m-1} \psi$, we get
\begin{equation}\label{3.59}
    \int_{Q_d} |\Delta^m \psi(x)|^2 \omega(x)^k dx \geq R(k, d) \int_{Q_d} |\Delta^{m-1}\psi(x)|^2 \omega(x)^{k-2} dx.
\end{equation}
Applying inequality \eqref{3.59} inductively w.r.t. $m$ we get \eqref{3.55}. Applying weighted Hardy inequality \eqref{3.26} with $\psi$ replaced by $\Delta^m \psi$ we obtain
\begin{equation}\label{3.60}
    \int_{Q_d} |\nabla (\Delta^m \psi)|^2 \omega^{k-1} \geq H(k, d)\int_{Q_d} |\Delta^m \psi|^2 \omega^{k-1} dx.
\end{equation}
Inequality \eqref{3.60} along with \eqref{3.55} gives inequality \eqref{3.57}.
\end{proof}
\begin{corollary}\label{cor3.5}
Let $\psi \in C^\infty(\overline{Q_d})$ all of whose derivatives are $2\pi$-periodic in each variable. Further assume that $\psi$ has zero average. Let $m$ be a non-negative integer.
\begin{enumerate}
    \item If $d > 4m$, then 
    \begin{equation}\label{3.61}
        C(m, d) \int_{Q_d} \frac{|\psi(x)|^2}{\big(\sum_i \sin^2(x_i/2)\big)^{2m}} dx \leq \int_{Q_d} |\Delta^m \psi(x)|^2 dx,
    \end{equation}
    where $C(m, d) := \prod\limits_{i=0}^{m-1}R(-2i, d)$ and $R(k, d)$ is the constant in the inequality \eqref{3.50}.
    \item If $d >4m +2$, then 
    \begin{equation}\label{3.62}
        \widetilde{C}(m, d) \int_{Q_d} \frac{|\psi(x)|^2}{\big(\sum_i \sin^2(x_i/2)\big)^{2m+1}} dx \leq \int_{Q_d} |\nabla(\Delta^m \psi)(x)|^2 dx,
    \end{equation}
    where $\widetilde{C}(m,d) := H(0, d)\prod\limits_{i=0}^{m-1} R(-2i-1, d)$ and $H(0,d)$ is the constant in \eqref{3.27}.
\end{enumerate}
\end{corollary}

\begin{remark}\label{rem3.10}
We note that $C(m, d) \sim d^{2m}$ and $\widetilde{C}(m, d) \sim d^{2m+1}$ as $d \rightarrow \infty$, since $H(k, d) \sim d$ and $R(k, d)\sim d^2$ as $d \rightarrow \infty$.
\end{remark}
\begin{proof}
Putting $k=0$ in Theorem \ref{thm3.7} proves Corollary \ref{cor3.5}.
\end{proof}
\section{Proof of the Main results}\label{sec: main proof(higher hardy inequality)}

\begin{proof}[Proof of Theorem \ref{thm3.2}]
From Lemma \ref{lem3.1} we know that there exists a smooth function $\psi$ all of whose derivatives are $2\pi$-periodic and has zero average, such that
\begin{equation}\label{3.63}
    \sum_{n \in \Z^d} \frac{|u(n)|^2}{|n|^{4k+2}} = \int_{Q_d} |\nabla(\Delta^k \psi)|^2 dx,
\end{equation}
and 
\begin{equation}\label{3.64}
    \sum_{n \in \Z^d} |D \Delta^k u(n)|^2 = 4^{2k+1}\int_{Q_d} |\Delta^{2k+1} \psi|^2 \omega(x)^{2k+1}dx,
\end{equation}
for $\omega(x) := \sum_i \sin^2(x_i/2)$. Integration by parts, along with H\"older's inequality, gives
\begin{equation}\label{3.65}
    \begin{split}
        \int_{Q_d} |\nabla(\Delta^k \psi)(x)|^2 dx &= -\int_{Q_d} \Delta^{2k+1} \psi(x)  \psi(x) dx \\
        &\leq \Bigg(\int_{Q_d}|\Delta^{2k+1} \psi(x)|^2 \omega^{2k+1} dx\Bigg)^{1/2} \Bigg(\int_{Q_d} \frac{|\psi|^2}{\omega^{2k+1}}dx\Bigg)^{1/2}.    
    \end{split}
\end{equation}
Inequality \eqref{3.65} along with Hardy-type inequality \eqref{3.62} gives
\begin{equation}\label{3.66}
    \widetilde{C}(k,d) \int_{Q_d} |\nabla(\Delta^k \psi)(x)|^2 dx \leq \int_{Q_d} |\Delta^{2k+1} \psi(x)|^2 \omega^{2k+1} dx,
\end{equation}
where $\widetilde{C}(k, d)$ is as defined by \eqref{3.62}. Inequalities \eqref{3.63} and \eqref{3.64} along with \eqref{3.66} yield
\begin{equation}\label{3.67}
  4^{2k+1} \widetilde{C}(k, d)\sum_{n \in \Z^d} \frac{|u(n)|^2}{|n|^{4k+2}} \leq \sum_{n \in \Z^d} |D\Delta^k u(n)|^2.
\end{equation}
This proves that $C_1(k, d) \geq 4^{2k+1}\widetilde{C}(k, d)$, where $C_1(k, d)$ is the sharp constant in \eqref{3.5}. Next, we bound $C_1(k, d)$ from above. Let $v(n) := \Delta^{k-1} u(n)$ for $u\in C_c(\Z^d)$ and consider
\begin{align*}
    \sum_{n \in \Z^d} |\Delta^k u(n)|^2 = \sum_{n \in \Z^d} |\Delta v(n)|^2 &= \sum_{n \in \Z^d} \Big(\sum_{j=1}^d(2v(n)-v(n-e_j)-v(n+e_j) \Big)^2\\
    &\leq \sum_{n \in \Z^d} \sum_{j, k} \varphi_j(n) \varphi_k(n)\\
    &\leq \sum_{j, k} \Big(\sum_{n \in \Z^d} |\varphi_j(n)|^2\Big)^{1/2}\Big(\sum_{n \in \Z^d} |\varphi_k(n)|^2\Big)^{1/2},
\end{align*}
where $\varphi_j(n) := |2v(n) - v(n-e_j) -v(n+e_j)|$. Applying H\"older's inequality and invariance of $\Z^d$ w.r.t. translations along co-ordinate directions we get, $\sum \limits_{n \in \Z^d}|\varphi_j(n)|^2  \leq 16 \sum\limits_{n \in \Z^d}|v(n)|^2$. Therefore, we have
\begin{align*}
    \sum_{n \in \Z^d} |\Delta^k u(n)|^2 \leq 16d^2 \sum_{n \in \Z^d}|v(n)|^2 = 16 d^2\sum_{n \in \Z^d} |\Delta^{k-1} u(n)|^2.
\end{align*}
Applying the above inequality iteratively, we obtain
\begin{equation}\label{3.68}
    \sum_{n \in \Z^d} |\Delta^k u(n)|^2 \leq 4^{2k} d^{2k} \sum_{n \in \Z^d}|u(n)|^2.  
\end{equation}
Consider a function $u$ defined on $\Z^d$ as follows: $u(n) := 1$ if $|n| = 1$ and $u(n) := 0$ everywhere else. Applying inequality \eqref{3.68}, we get
\begin{align*}
    \sum_{n \in \Z^d} |D (\Delta^k u)(n)|^2 \leq 4d \sum_{n \in \Z^d} |\Delta^k u(n)|^2 &\leq 4^{2k+1} d^{2k+1} \sum_{n \in \Z^d}|u(n)|^2\\
    &= 4^{2k+1} d^{2k+1} \sum_{n \in \Z^d} \frac{|u(n)|^2}{|n|^{4k+2}}. 
\end{align*}
This proves that $C_1(k, d) \leq 4^{2k+1} d^{2k+1}$. Therefore we have $4^{2k+1} \widetilde{C}(k,d) \leq C_1(k, d) \leq 4^{2k+1} d^{2k+1}$. This proves that $C_1(k, d) \sim d^{2k+1}$ as $ d \rightarrow \infty$, since $\widetilde{C}(k, d) \sim d^{2k+1}$ as $d \rightarrow \infty$(see Remark \ref{rem3.10}).
\end{proof}

\begin{proof}[Proof of Theorem \ref{thm3.3}]
From Lemma \ref{lem3.2} we get the following identities
\begin{equation}\label{3.69}
    \sum_{n \in \Z^d} \frac{|u(n)|^2}{|n|^{4k}} = \int_{Q_d} |\Delta^k \psi(x)|^2 dx,
\end{equation}
\begin{equation}\label{3.70}
    \sum_{n \in \Z^d} |\Delta^k u(n)|^2 = 4^{2k}\int_{Q_d} |\Delta^k \psi(x)|^2 \omega(x)^{2k} dx. 
\end{equation}
Integration by parts and H\"older's inequality gives
\begin{equation}\label{3.71}
    \begin{split}
        \int_{Q_d} |\Delta^k \psi(x)|^2 dx &= \int_{Q_d} \Delta^{2k}\psi(x) \psi(x) dx \\
        &\leq \Bigg(\int_{Q_d} |\Delta^k \psi(x)|^2 \omega(x)^{2k} dx\Bigg)^{1/2}\Bigg(\int_{Q_d}\frac{|\psi(x)|^2}{\omega(x)^{2k}} dx\Bigg)^{1/2}.    
    \end{split}
\end{equation}
Inequality \eqref{3.71} along with Hardy inequality \eqref{3.61} gives
\begin{equation}\label{3.72}
    C(k, d)\int_{Q_d} |\Delta^k \psi(x)|^2 dx \leq \int_{Q_d} |\Delta^k \psi(x)|^2 \omega(x)^{2k} dx,
\end{equation}
where $C(k, d)$  is defined in \eqref{3.61}. Inequalities \eqref{3.69}, \eqref{3.70} and \eqref{3.72} give
\begin{equation}
    4^{2k} C(k, d)\sum_{n \in \Z^d} \frac{|u(n)|^2}{|n|^{4k}} \leq \sum_{n \in \Z^d} |\Delta^k u(n)|^2.
\end{equation}
Therefore we have $C_2(k, d) \geq 4^{2k} C(k, d)$, where $C_2(k, d)$ is the sharp constant in the inequality \eqref{3.6}. Consider a function $u(n) :=1 $ if $|n|=1$ and $u(n) := 0$ otherwise. Using inequality \eqref{3.68}, we get
\begin{align*}
    \sum_{n \in \Z^d} |\Delta^k u(n)|^2 \leq 4^{2k}d^{2k} \sum_{n \in \Z^d}\frac{|u(n)|^2}{|n|^{4k}}. 
\end{align*}
This shows that $C_2(k, d) \leq 4^{2k} d^{2k}$. Therefore we have $4^{2k} C(k, d) \leq C_2(k, d) \leq 4^{2k} d^{2k}$, which implies that $C_2(k, d) \sim d^{2k}$ as $d \rightarrow \infty$, since $C(k, d) \sim d^{2k}$ as $ d \rightarrow \infty$(see Remark \ref{rem3.10}).
\end{proof}

%% file: 1D_Rearrangement_inequality/1D_Rearrangement_inequality.tex
\chapter{Rearrangement inequality on 1D lattice graph }\label{ch:1d-rearrangement}
\section{Introduction}
As discussed in Chapter \ref{ch:intro}, Rearrangement inequalities in the continuum are an extremely useful analytic tool which leads to several important results in analysis. For example, they can be used to prove that optimizers of various functional inequalities such as Sobolev inequality, Hardy-littlewood-Sobolev inequality are spherically symmetric functions. They also lead to various isoperimetric results, for example, they tell us that balls minimize the ground state (smallest eigenvalue) of Dirichlet Laplacian amongst sets of a fixed volume. The implications of the theory has a long list, we refer to survey articles \cite{frank1} and \cite{sabin} for some interesting applications.

In this and the next chapter, we study to what extent Rearrangement inequalities can be extended to the setting of graphs. Before moving to the actual theme, let us briefly recall the notion of rearrangement in the continuum, highlighting the results which we extend to the discrete setting. 

Given a suitable function $f: \R^d \rightarrow \R$, one define its \emph{symmetric decreasing rearrangement} to be 
\begin{equation}\label{4.1}
    f^*(x) := \int_{0}^\infty \chi_{\{x: |f(x)|>t\}^*} (x) dt,
\end{equation}
where give $\Omega \subseteq \R^d$, $\Omega^*$ is the open ball of same measure as $\Omega$ centered at the origin. It can be easily deduced that $f^*$ is  spherically symmetric, decreasing, and its $L^p$ norm is same as that of $f$ (see Chapter \ref{ch:intro}). One of the useful rearrangement inequality is a functional version of isoperimetric inequality:
\begin{equation}\label{4.2}
    \int_{\R^d} |\nabla f^*|^p dx \leq \int_{\R^d} |\nabla f|^p dx,
\end{equation}
for $p \geq 1$. Inequality \eqref{4.2} is famously known as Polya-Szeg\H{o} inequality, and is a consequence of the fact that perimeter of level sets decreases under rearrangement. This inequality has successfully been used in analysis: for proving sharp functional inequalities, isoperimetric results in spectral theory, etc. From now on, the focus would be to understand whether such an inequality can hold when the underlying space is a graph. 

\subsection{Defining discrete rearrangement}
From now on, $G = (V, E)$ represents a connected graph with a countably infinite vertex set $V$. We also assume that it is \emph{locally finite}, that is, degree of each vertex is finite. For $x, y \in V$, $x \sim y$ means that $(x, y) \in E$. There are natural analogues of $L^p-$spaces of functions and gradients on a graph and we define, for $1 \leq p < \infty$ and $f: V \rightarrow \mathbb{R}$, the norms
$$ \|f\|^p_{L^p} = \sum_{x \in V} |f(x)|^p \qquad \mbox{and} \qquad \| \nabla f\|_{L^p}^p = \sum_{x \sim y}|f(x) - f(y)|^p$$
with the usual modification at $p = \infty$. 

It is natural to ask: Given a function $f: V \rightarrow \R$, how should one define its `rearrangement' $f^*: V \rightarrow \R$? One obvious answer is to define it as its defined in the continuous setting \eqref{4.1}. \\
(Possible `vague' definition) 
\begin{equation}\label{4.3}
    f^*(x) := \int_{0}^\infty \chi_{\{x: |f(x)|>t\}^*} (x) dt. 
\end{equation}
This definition is mathematically sound modulo `what is $\Omega^*$, for a subset $\Omega$ of the vertex set'? With the aim of making an applicable theory, it is natural to put the following constraints:
\begin{enumerate}
    \item In order to preserve $\ell^p$ norm we want $|\Omega| = |\Omega^*|$, that is, they have the same size.\\
    \item It is also important to keep $\Omega^*$ fixed as the size of $\Omega$ is fixed. This means that $\Omega^*$ only depends on the size of $\Omega$ and not on its geometry. This invariant property is useful for applications. 
\end{enumerate}
An obvious way to meet the above requirements is by labelling the graph with natural numbers and defining $\Omega^*$ as the first $k$ vertices w.r.t. the labelling, for a set $\Omega$ of size $k$. Then, it is not hard to see that definition \eqref{4.3} is equivalent to 
\begin{definition}\label{def4.1}
Let $G=(V, E)$ be a graph and let $\eta: \N \rightarrow V$ be a bijective map (called \emph{labelling} of the graph $G$). Let $f: V \rightarrow \R$ be a function \emph{vanishing at infinity}, that is, $\{x: |f(x)|>t\}$ is finite for all $t >0$. Then its $\emph{rearrangement}$ is a function $f^*: V \rightarrow \R$ defined by
$$ f^*(\eta(k)) := k^{th} \hspace{5pt} \text{largest value assumed by} \hspace{5pt} |f|.$$
\end{definition}
With this definition, one obviously has $\|f\|_p = \|f^*\|_p$ for $p>0$. It remains to understand whether a discrete analogue of Polya-Szeg\H{o} inequality \eqref{4.2} holds. In other words, do we have the Polya-Szeg\H{o} inequality $\|\nabla f^*\|_p \leq \|\nabla f\|_p$, for $1 \leq p \leq \infty?$ 

\subsection{Previous Results}
The study of discrete rearrangement inequalities was perhaps started by Pruss \cite{pruss}, who proved that the Polya-Szeg\H{o} inequality holds true on infinite regular trees for $p=2$, when a function is rearranged with respect to the \emph{spiral-like labelling} (see Fig. \ref{fig:tree}). In fact, Pruss studied more general Riesz rearrangement inequality on arbitrary graphs. 
\begin{theorem}[Pruss \cite{pruss}]\label{thm4.1}
Let $T_q $ be the infinite $q-$regular tree and let $f : V\rightarrow \R_{\geq 0}$ be a function vanishing at infinity and $f^*$ be the rearrangement of $f$ with respect to the spiral-like labelling. Then
\begin{equation}\label{4.4}
    \sum_{x \sim y}|f^*(x)-f^*(y)|^2 \leq \sum_{x \sim y}|f(x)-f(y)|^2. 
\end{equation}
\end{theorem}
When $q=2$, the $2-$regular tree is simply the integer lattice graph on $\mathbb{Z}$, also known as the doubly-infinite path. In that case, Hajaiej \cite{hajaiej1} used the idea of \emph{polarization} to extended the inequality from $L^2$ to $L^p$ for $p \geq 1$. Recently, Steinerberger \cite{steinerberger} proved \eqref{4.4} for all $p \geq 1$ and all $q-$regular trees. 

The Polya-Szeg\H{o} inequality in $\mathbb{R}^n$ is a generalization of the isoperimetric inequality and relies on all the symmetries of the Euclidean space $\mathbb{R}^n$ and the special role played by the sphere. As such, it is well understood to be a rather special object and is deeply tied to the underlying geometry. It is, for example, not at all clear how one would define rearrangement on a generic compact manifold. Likewise, the infinite regular tree is a highly symmetric object and a relevant question is whether something interesting can be said about rearrangements on more `generic' graphs like the lattice graph $\mathbb{Z}^d$. \newpage

\begin{figure}[h!]
    \centering
    \includegraphics[width=0.4\textwidth]{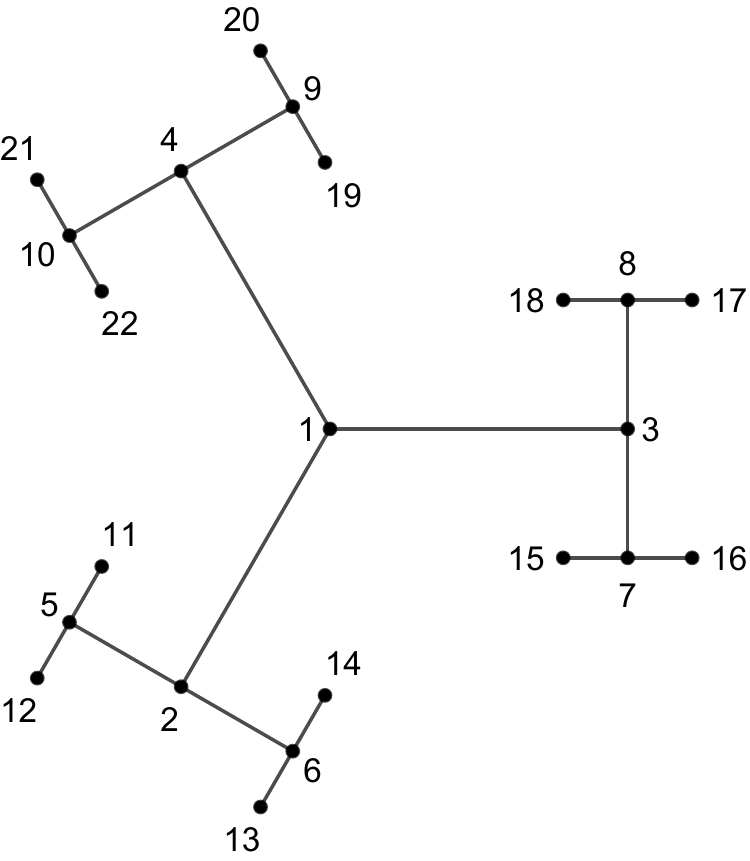}
    \caption{The spiral-like labelling on the $3$-regular infinite tree (only the first three layers of the infinite tree are shown).}
    \label{fig:tree}
\end{figure}

\begin{definition}
The \emph{lattice graph} is defined as a graph on $\Z^d$, such that two vertices $x, y \in \Z^d$ are connected by an edge if and only if $\|x-y\|_{\ell^1} = \sum_{i=1}^d |x_i-y_i| = 1$.
\end{definition}

\begin{figure}[h!]
\begin{center}
\begin{tikzpicture}[scale=0.9]
    \filldraw (0.5, 0) circle (0.06cm);
    \filldraw (1.5, 0) circle (0.06cm);
    \filldraw (2.5, 0) circle (0.06cm);
    \filldraw (0.5, 1) circle (0.06cm);
    \filldraw (1.5, 1) circle (0.06cm);
    \filldraw (2.5, 1) circle (0.06cm);
    \filldraw (0.5, 2) circle (0.06cm);
    \filldraw (1.5, 2) circle (0.06cm);
    \filldraw (2.5, 2) circle (0.06cm);
    \draw  (0,0) -- (3,0);
    \draw  (0,1) -- (3,1);
    \draw  (0,2) -- (3,2);
    \draw  (0.5,-0.5) -- (0.5,2.5);
    \draw  (1.5,-0.5) -- (1.5,2.5);
    \draw  (2.5,-0.5) -- (2.5,2.5);
\end{tikzpicture}
\end{center}
\caption{Part of the two dimensional lattice graph $(\mathbb{Z}^2, \ell^1)$}
\end{figure}
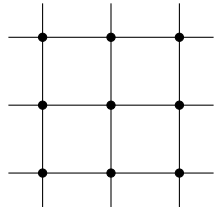
It is natural to study rearrangements on the lattice graph which comes from solutions of isoperimetric problems. To that end, we define two natural notions of `boundary' of a set:
\begin{definition}
Let $G=(V, E)$ be a graph. Let $X$ be a finite subset of $V$. Then we define its \emph{vertex} and \emph{edge boundary} as
\begin{equation}
    \partial_V(X) := \{y \in V \setminus X: y \sim x, \hspace{5pt} \text{for some} \hspace{5pt} x \in X\},
\end{equation}
and
\begin{equation}
    \partial_E(X) := \{e=(x, y) \in E: x \in X \hspace{5pt} \text{and} \hspace{5pt} y \in V \setminus X\}
\end{equation}
respectively. The size of these sets are called \emph{vertex} and \emph{edge perimeter} of $X$ respectively. 
\end{definition}
These notions give rise to two isoperimetric problems on graphs: Amongst all possible subsets of vertices of fixed size, which has the smallest possible vertex perimeter, and which one has the least possible edge perimeter? These problems have been solved for the lattice graph in all dimensions, we refer to \cite{bol1, bol2} for edge isoperimetric problem, and \cite{wang} for the vertex isoperimetric problem. The solutions of both the problems are \emph{nested}: this means that there exists a nested sequence of sets of vertices $A_1 \subseteq A_2 \subseteq A_3 \subseteq \dots$ such $|A_i| = i$ and $A_i$ has the least possible perimeter (vertex or edge), amongst all subsets of size $i$. On the two dimensional lattice graph, these nested solutions can be expressed explicitly. 
\begin{figure}[h!]
\centering
\begin{minipage}[r]{.35\textwidth}
  \begin{tikzpicture}[scale=1.4]
  \node at (0,0) {1};
    \node at (0.5,0) {2};
    \node at (0.5,0.5) {3};
  \node at (0,0.5) {4};
  \node at (-0.5,0.5) {5};
  \node at (-0.5,0) {6};
  \node at (-0.5,-0.5) {7};
  \node at (0,-0.5) {8};
  \node at (0.5,-0.5) {9};
    \node at (1,-0.5) {10};
    \node at (1,0) {11};
    \node at (1,0.5) {12};
    \node at (1,1) {13};
        \node at (0.5,1) {14};
    \node at (0,1) {15};
    \node at (-0.5,1) {16};
  \draw [->] (1 ,-0.35) -- (1,-0.15);
    \draw [->] (1 ,0.15) -- (1,0.35);
    \draw [->] (1 ,1.3/2) -- (1,1.7/2);
    \draw [->] (1.7/2,2/2) -- (1.4/2,2/2);
        \draw [->] (0.7/2,2/2) -- (0.4/2,2/2);
    \draw [->] (-0.3/2,2/2) -- (-0.6/2,2/2);
        \draw [->] (-1.3/2,2/2) -- (-1.8/2,2/2);
  \draw [->] (0.15, 0) -- (0.35, 0);
    \draw [->] (1/2, 0.3/2) -- (1/2, 0.7/2);
  \draw [->] ( 0.7/2,1/2) -- (0.3/2,1/2);
 \draw [->] ( -0.3/2,1/2) -- (-0.7/2,1/2);
 \draw [->] ( -1/2,0.7/2) -- (-1/2,0.3/2);
 \draw [->] ( -1/2,-0.3/2) -- (-1/2,-0.7/2);
 \draw [->] ( -0.7/2,-1/2) -- (-0.3/2,-1/2);
  \draw [->] ( 0.3/2,-1/2) -- (0.7/2,-1/2);
  \draw [->] ( 1.3/2 ,-1/2) -- (1.7/2,-1/2);
    \end{tikzpicture}
    \vspace{0pt}
\end{minipage}
\begin{minipage}[l]{.35\textwidth}
\begin{tikzpicture}
\node at (-3,0) {};
\node at (0,0) {1};
\node at (0,0.5) {2};
\node at (0.5,0) {3};
\node at (-0.5,0) {4};
\node at (0,-0.5) {5};
\node at (0.5,0.5) {6};
\node at (-0.5,0.5) {7};
\node at (0,1) {8};
\node at (1,0) {9};
\node at (0.5,-0.5) {10};
\node at (-1,0) {11};
\node at (-0.5,-0.5) {12};
\node at (0,-1) {13};
\end{tikzpicture}
\end{minipage}
\caption{Left: Solution of edge isoperimetric problem. Right: Solution of vertex isoperimetric problem.} 
\label{fig:spiral and wang-wang}
\end{figure}
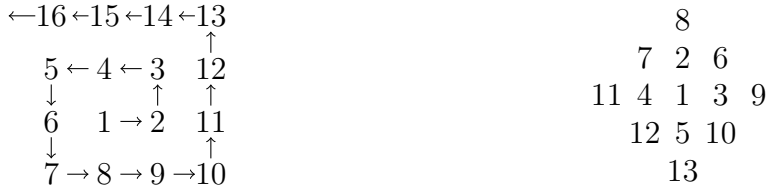
\newline The left labelling in Fig. \ref{fig:spiral and wang-wang} solves edge isoperimetric problem on $\Z^2$: the first $k$ vertices have the smallest possible edge perimeter, amongst all possible subsets of size $k$. Similarly, the right labelling solves vertex isoperimetric problem. Let us say a few words about these labellings. The one on the left is quite explicit, and goes in a spiral way around the origin. In particular, the vertices are labelled with respect to the $\ell^\infty$ norm: if $\|x\|_{\ell^\infty} < \|y\|_{\ell^\infty}$, then the label $x$ is smaller than that of $y$. Thus, asymptotically, the solutions of the edge isoperimetric problem are $\ell^\infty$ balls in $\Z^2$. We refer to it as \emph{spiral} labelling and the corresponding rearrangement as \emph{spiral rearrangement}.
On the other hand, the vertices on the right of Fig. \ref{fig:spiral and wang-wang} are labelled w.r.t. $\ell^1$ balls, meaning if $\|x\|_{\ell^1} < \|y\|_{\ell^1}$, then the label of $x$ is smaller than that of $y$. Hence, the asymptotic solutions of edge isoperimetric problem on $\Z^2$ are $\ell^1$ balls. This labelling is not as explicit as the spiral one, to see the exact definition of this labelling, we refer to a paper by Wang and Wang \cite{wang}. We call this labelling \emph{Wang-Wang labelling} and the corresponding rearrangement as \emph{Wang-Wang rearrangement}.

It was shown in \cite{steinerberger}, that the spiral rearrangement satisfies Polya-Szeg\H{o} inequality for $p=1$ and Wang-Wang rearrangement satisfies Polya-Szeg\H{o} inequality for $p=\infty$. We include the proof of $p=1$, to indicate how isoperimetric inequality appears in this context.

\begin{theorem}[\cite{steinerberger}]
Let $f: \Z^2 \rightarrow \R$ be a function vanishing at infinity and let $f^*$ be its spiral rearrangement. Then Polya-Szeg\H{o} inequality holds for $p=1$:
\begin{equation}
    \sum_{x \sim y} |f^*(x)-f^*(y)| \leq \sum_{x \sim y}|f(x)-f(y)|.
\end{equation}
\end{theorem}

\begin{proof}
The proof goes via a discrete \emph{coarea formula}: Let $f: V \rightarrow \R_{\geq 0}$ be a function. Then for $p \geq 1$ we have 
\begin{equation}\label{4.8}
    \sum_{x \sim y}|f(x)-f(y)|^p = \int_0^\infty \sum_{(x,y) \in E(\{f>t\}, \{f>t\}^c)}|f(x)-f(y)|^{p-1} dt,
\end{equation}
where $E(X, Y)$ denotes the edges between set $X$ and $Y$. We give its proof in Section \ref{sec:one dimension rearrangement}. If $p=1$, then 
\begin{equation}\label{4.9}
    \sum_{x \sim y}|f(x)-f(y)| = \int_0^\infty |\partial_E(\{f>t\})| dt.
\end{equation}
Since $f$ is rearranged with respect to the spiral labelling, we have the isoperimetric inequality 
$$ |\partial_E(\{f>t\})|  \geq |\partial_E(\{f^*>t\})|,$$
which proves the result.
\end{proof}
At this point it is natural to ask what happens when $1<p<\infty$? Does spiral or Wang-Wang or some other rearrangement satisfy Polya-Szeg\H{o} inequality for $1<p<\infty$? Recently, a surprising impossibility result was proved for $p=2$ \cite{hajaiej2}. 

\begin{theorem}[\cite{hajaiej2}]
Let $\eta$ be a labelling of the two dimensional lattice graph. There always exists a compactly supported function $f: \Z^2 \rightarrow \R$ such that 
\begin{equation}
    \sum_{x \sim y} |f^*(x)-f^*(y)|^2 > \sum_{x \sim y}|f(x)-f(y)|^2,
\end{equation}
where $f^*$ is the rearrangement of $f$ w.r.t. labelling $\eta$. In fact, one can find an $f$ supported on only five points. 
\end{theorem}
This result shows that even on some `nice' symmetric graphs, it is not possible to rearrange functions in a way that decreases norms of the `gradient'. This impossibility result raises a lots of interesting questions: 

Question 1: Is it possible to quantify how badly Polya-Szeg\H{o} inequality fails? In other words, given a labelling $\eta$, what is the constant $c_p$ in 
\begin{equation}\label{4.11}
    \|\nabla f^*\|_p \leq c_p \|\nabla f\|_p?
\end{equation}
The constant $c_p(\eta)$ in a way measures the quality of a given rearrangement. The smaller $c_p(\eta)$ is, the better the rearrangement is. This viewpoint leads to an interesting geometric question:

Question 2: Fix $1\leq p \leq \infty$. Amongst all possible labellings on $\Z^2$, which one minimizes  $c_p(\eta$)? In other words, what is the best way to rearrange functions on $\Z^2$ for a given value of $p$?  

In this language, we know that the spiral rearrangement is optimal for $p=1$ and the Wang-Wang rearrangement is optimal for $p=\infty$. A particularly interesting question is whether the optimality of these rearrangements extends beyond these end points. In other words, does there exists $1<p_0, p_1 <\infty$, such that spiral rearrangement is optimal when $p \in (1, p_0)$ and Wang-Wang rearrangement is optimal when $p \in (p_1, \infty)$? And whether $p_0=p_1$? or are there other rearrangements which are optimal in the intermediate range? We discuss these questions as well as their partial solutions in Chapter \ref{ch:higher-rearrangement}. 

\subsection{Our Contributions}
All the results discussed so far use Definition \ref{def4.1} as their definition of rearrangement. While this definition is quite natural and gives rise to a rich theory, which potentially can have several applications, the rearranged function only has monotonicity property and lacks the the spherical symmetricity property, which could be useful in certain circumstances. In Section \ref{sec:one dimension rearrangement} we define a new notion of rearrangement on the one dimensional lattice graph. Under this notion, the rearranged function has both properties. In the process, we also obtain weighted Polya-Szeg\H{o} inequality on the integers, extending the result of Hajaiej \cite{hajaiej1} to the weighted case. We further apply the weighted inequality to prove weighted Hardy-type inequality on integers. The results of Section \ref{sec:one dimension rearrangement} are based on \cite{gupta3}. 

In Chapter \ref{ch:higher-rearrangement} we develop an approach to study rearrangement inequalities on general graphs, obtaining concrete results on the two dimensional lattice graph. In particular, we compute an explicit constant in \eqref{4.11} corresponding to spiral and Wang-Wang rearrangement. These results are then abstracted out to obtain inequalities of the type \eqref{4.11} on an arbitrary graph. The results of this Chapter \ref{ch:higher-rearrangement} are obtained in collaboration with Stefan Steinerberger \cite{gupta5}.

\section{Notions of rearrangement: 1D case}\label{sec:one dimension rearrangement}
We study the Polya-Szeg\H{o} inequality on one-dimensional lattice graph. To recall, it is a graph on $\Z$, such that $x, y \in \Z$ have an edge between them if and only if $|x-y| =1$. 

\begin{figure}[h!]
\centering
\begin{tikzpicture}[scale=1]
\draw [thick] (0,0) -- (7,0);
\filldraw (3.5, 0) circle (0.06cm);
\filldraw (4.5, 0) circle (0.06cm);
\filldraw (2.5, 0) circle (0.06cm);
\filldraw (5.5, 0) circle (0.06cm);
\filldraw (1.5, 0) circle (0.06cm);
\filldraw (6.5, 0) circle (0.06cm);
\filldraw (0.5, 0) circle (0.06cm);
\node at (0.5, 0.3) {-3};
\node at (1.5, 0.3) {-2};
\node at (2.5, 0.3) {-1};
\node at (3.5, 0.3) {0};
\node at (4.5, 0.3) {1};
\node at (5.5, 0.3) {2};
\node at (6.5, 0.3) {3};
\end{tikzpicture}
\caption{Part of one-dimensional lattice graph}
\label{fig:two}
\end{figure}
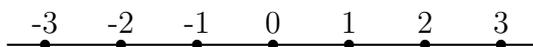
Polya-Szeg\H{o} inqeuality in this simple setting becomes:
\begin{equation}\label{4.12}
    \sum_{n \in \Z}|u^*(n)-u^*(n-1)|^p \leq \sum_{n \in \Z}|u(n)-u(n-1)|^p, 
\end{equation}
for $p \geq 1$. Inequality \eqref{4.12} is known to be true, when $u^*$ is defined as the decreasing rearrangement of $u$ w.r.t. to spiral-like labelling (see Def. \ref{def4.1} and Theorem \ref{thm4.1}). 
\begin{figure}[h!]
\centering
\begin{tikzpicture}[scale=1]
\draw [thick] (0,0) -- (7,0);
\filldraw (3.5, 0) circle (0.06cm);
\filldraw (4.5, 0) circle (0.06cm);
\filldraw (2.5, 0) circle (0.06cm);
\filldraw (5.5, 0) circle (0.06cm);
\filldraw (1.5, 0) circle (0.06cm);
\filldraw (6.5, 0) circle (0.06cm);
\filldraw (0.5, 0) circle (0.06cm);
\node at (0.5, 0.3) {7};
\node at (1.5, 0.3) {5};
\node at (2.5, 0.3) {3};
\node at (3.5, 0.3) {1};
\node at (4.5, 0.3) {2};
\node at (5.5, 0.3) {4};
\node at (6.5, 0.3) {6};
\end{tikzpicture}
\caption{$u^*(1) \, $ \text{is the largest value of} $\,u, \, u^*(2) \, $ \text{is the second largest, and so on}.}
\label{fig:integers}
\end{figure}
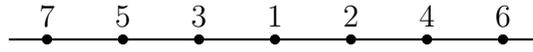

To recall, this result was first proved by Pruss \cite{pruss} for $p=2$, and then later extended to all $p \geq 1$ by Hajaiej \cite{hajaiej1}. Although this rearrangement is decreasing w.r.t. to natural spiral like labelling, but it lacks the spherical symmetric property: meaning $u^*(n) = u^*(-n)$. In this section, we define another notion of rearrangement which is decreasing (w.r.t. spiral like labelling), spherically symmetric and satisfy Polya-Szgeg\H{o} inqeuality. This is done in three steps: in Subection \ref{subsec:decreasing rearrangement} we define the standard decreasing rearrangement (Def. \ref{def4.1}, Fig. \ref{fig:integers}) on the \emph{half-line} (i.e., on non-negative integers) and prove the weighted version of the Polya-Szeg\H{o} inequality \eqref{4.12}. We would like to mention that although the Polya-Szeg\H{o} inequality  \eqref{4.12} for decreasing rearrangement is already known (see \cite{pruss, hajaiej1}, for instance), the weighted version seems to be missing from current literature. Moreover, our proof of the unweighted case via the discrete co-area formula is also new. Secondly, we define \emph{Fourier rearrangement} in Subsection \ref{subsec:fourier rearrangement}, which associates to every function $u \in \ell^2(\Z)$, a radial function $u^\#$. We prove Polya-Szeg\H{o} principle for the Fourier rearrangement, furthermore, we prove that Polya-Szeg\H{o} holds true, even when one replaces first order `derivatives' with higher order `derivatives' in \eqref{4.12}. A similar type of Fourier rearrangement procedure in $\R^d$ was carried out recently by Lenzmann and Sok \cite{lenzmann}, which was used by the authors to prove radial symmetry of various higher order variational problems. Our Fourier rearrangement method can be thought of as a discrete analogue of their work. Finally, in Subsection \ref{subsec:symmetric-decreasing rearrangement}, we combine both decreasing and Fourier rearrangements to define \emph{symmetric-decreasing rearrangement} $u^*$, of the function $u$, which is a spherically symmetric and decreasing function. We also prove the corresponding Polya-Szeg\H{o} inequality for this rearrangement.

In Subsection \ref{subsec:applications}, we apply the weighted Polya-Szeg\H{o} inequality proved in Subsection \ref{subsec:decreasing rearrangement} to prove discrete Hardy-type inequality on non-negative integers with power weights $n^\alpha$, for $1 < \alpha \leq 2$.

\subsection{Decreasing rearrangement on the half-line}\label{subsec:decreasing rearrangement}
We deal with functions which vanish at infinity in the following weak sense: Let $\Z^+$ be the set of all non-negative integers and let $u: \mathbb{Z}^+ \rightarrow \C$ be a function. We say that function $u$ $\emph{vanishes at infinity}$ if all level sets of $|u|$ are finite, i.e., $\{x: |u(x)| > t\}$ is finite for all $t > 0$. 

Let $A$ be a finite subset of $\mathbb{Z}^+$ such that $|A| = k$. Then we define its \emph{symmetric rearrangement} $A^*$ as the set of all non-negative integers which are less than or equal to $k-1$. Thus, $A^*$ is a ball in $\mathbb{Z}^+$ centered at origin whose size is same as the size of $A$. 

Let $u: \mathbb{Z}^+ \rightarrow \C$ be a function vanishing at infinity. We define its \emph{decreasing rearrangement} $\widetilde{u}$ as:
\begin{equation}\
    \widetilde{u}(n) := \int_0^\infty \chi_{\{|u|>t\}^*}(n) dt.
\end{equation}
Some useful properties of $\widetilde{u}$:
\begin{enumerate}
    \item $\widetilde{u}$ is always non-negative. 
    \item The level sets of $\widetilde{u}$ are rearrangements of level sets of $|u|$, i.e., 
    \begin{equation}
        \{n: \widetilde{u} > t\} = \{n: |u|>t\}^*,
    \end{equation}
    for all $t >0$. An important consequence of this is that level sets of $\widetilde{u}$ and $|u|$ have the same size. 
    \item It is easy to conclude that $\widetilde{u}(0)$ is the largest value of $|u|$, $\widetilde{u}(1)$ is the second largest value of $|u|$, $\widetilde{u}(2)$ is the third largest value of $|u|$ and so on. In particular, $\widetilde{u}$ is a decreasing function of $n$.
    \item Equimeasurability of levels sets of $\widetilde{u}$ and $|u|$ along with layer cake representation gives
    \begin{equation}\label{4.15}
        \sum_{n \in \mathbb{Z}^+}|u|^p = \sum_{n \in \mathbb{Z}^+}|\widetilde{u}|^p,
    \end{equation}
    for $p \geq 1$.
    \item Let $\Phi: \mathbb{R^+} \rightarrow\mathbb{R^+}$ be a bijective map. Then 
    \begin{equation}
        \widetilde{\Phi(|u|)} = \Phi(\widetilde{u}).
    \end{equation}
    \item The rearrangement is \emph{order preserving}, i.e, if $u$ and $v$ are two non-negative functions such that $u(n) \leq v(n)$ then $\widetilde{u}(n) \leq \widetilde{v}(n)$. This follows from the fact that $u(n) \leq v(n)$ is equivalent to $\{u>t\} \subseteq \{v>t\}$.
    \item (\emph{Hardy-Littlewood inequality}) Let $u$ and $v$ be non-negative functions vanishing at infinity. Then we have 
    \begin{equation}\label{4.17}
        \sum_{n \in \mathbb{Z}^+}u(n)v(n) \leq \sum_{n \in \mathbb{Z}^+}\widetilde{u}(n)\widetilde{v}(n). 
    \end{equation}
    Using layer-cake representation of functions $u, v$ and monotone convergence theorem we get
    \begin{align*}
        \sum_{n \in \mathbb{Z}^+}u(n)v(n) = \int_0^\infty \int_0^\infty \sum_{n \in \Z^+} \chi_{\{u>t\}}(n) \chi_{\{v>s\}}(n) dt ds. 
    \end{align*}
    Therefore it is enough to prove \eqref{4.17} when $u$ and $v$ are characteristic functions, i.e., $u = \chi_A$ and $v = \chi_B$ for finite subsets $A$ and $B$ of $\mathbb{Z}^+$. For this choice of $u$ and $v$ inequality \eqref{4.17} is equivalent to proving $|A \cap B| \leq |A^* \cap B^*|$. Without loss of generality we can assume that $|A|\leq |B|$, which implies $A^* \subseteq B^*$. Then $|A^* \cap B^*| = |A^*|= |A|\geq |A \cap B|$. This completes the proof of the Hardy-Littlewood inequality \eqref{4.17}.
    \item $l^p(\Z^+)$ distance decreases under decreasing rearrangement, i.e., 
    \begin{equation}\label{4.18}
        \sum_{n \in \Z^+}|\widetilde{u}(n) -\widetilde{v}(n)|^p \leq \sum_{n \in \Z^+} |u(n)-v(n)|^p,
    \end{equation}
    for all $u, v$ non-negative functions vanishing at infinity and $p \geq 1$. For $p=2$, inequality \eqref{4.18} follows directly from \eqref{4.15} and Hardy-Littlewood inequality \eqref{4.17}. For the proof of the general case $p\geq 1$, we refer \cite[theorem 3.5, chapter 3]{liebloss}, where inequality \eqref{4.18} was proved for rearrangements on the real line.
\end{enumerate}
Next we will prove the main theorem of this section, a \emph{weighted Polya-Szeg\H{o} inequality}. We would like to mention that the Polya-Szeg\H{o} inequality on the integers has been considered in the past \cite{pruss, hajaiej1}, but to the author's best knowledge, the Polya-Szeg\H{o} principle with weights has not been considered before. The major tool in our proof is a discrete analogue of the co-area formula, which is the content of our next lemma. We state the co-area formula in the general setting of graphs.
\begin{lemma}[Discrete co-area formula]
Let $G =(V,E)$ be a graph, where $V$ is a countable set of vertices, and $E$ is a symmetric relation on set $V$. We also assume that $G$ is locally finite, i.e., each vertex of $G$ has finite degree. Let $u: V(G) \rightarrow \mathbb{R}$ be a non-negative function. For $p \geq 1$ we have,
\begin{equation}\label{4.19}
    \sum_{x \sim y}|u(x)-u(y)|^p = \int_0^\infty \sum_{(x,y) \in E(\{u>t\}, \{u>t\}^c)}|u(x)-u(y)|^{p-1} dt,
\end{equation}
where $E(X,Y)$ denotes the set of edges between subsets $X$ and $Y$ of graph $G$, and $x \sim y$ means vertices $x$ and $y$ are connected by an edge.
\end{lemma}
\begin{proof}
Let $A: V \times V \rightarrow \R$ be a map defined as:
\begin{align*}
    A(x,y) := 
    \begin{cases}
    1 \hspace{19pt} \text{if} \hspace{5pt} x \sim y.\\
    0 \hspace{19pt} \text{if} \hspace{5pt} x \nsim y.
    \end{cases}
\end{align*}
Consider, 
\begin{align*}
    \sum_{x \sim y}|u(x)-u(y)|^p &= \sum_{x \in V} \sum_{y \in V} A(x,y)|u(x)-u(y)|^p\\
    &= \sum_{u(x) \leq u(y)}A(x,y)|u(x)-u(y)|^{p-1}(u(y)-u(x)) \\
    &+ \sum_{u(x) > u(y)}A(x,y)|u(x)-u(y)|^{p-1}(u(x)-u(y))\\
    &= \int_0^\infty  \sum_{u(x) \leq u(y)}A(x,y)|u(x)-u(y)|^{p-1} \chi_{[u(x),u(y))}(t) dt\\
    &+ \int_0^\infty \sum_{u(x) > u(y)}A(x,y)|u(x)-u(y)|^{p-1}  \chi_{[u(y),u(x))}(t) dt.
\end{align*}
Finally, using the symmetry of the integrand with respect to the variables $x$ and $y$ we obtain,
\begin{align*}
    \sum_{x \sim y}|u(x)-u(y)|^p = \int_0^\infty \sum_{(x,y) \in E(\{u>t\}, \{u> t\}^c)}|u(x)-u(y)|^{p-1} dt. 
\end{align*}
\end{proof}
\begin{theorem}[Weighted Polya-Szeg\H{o} inequality]\label{thm4.4}
Let $w: \Z^+ \rightarrow \R$ be a non-negative increasing function. Let $u: \Z^+ \rightarrow \C$ be a function which is vanishing at infinity. We have,  
\begin{equation}\label{4.20}
    \sum_{n \in \Z^+} |u(n)-u(n+1)|^p w(n) \geq \sum_{n \in \Z^+} |\widetilde{u}(n)-\widetilde{u}(n+1)|^p w(n), 
\end{equation}
for  all $p \geq 1$. Moreover, if $w(n) >0$, if $u$ produces equality in \eqref{4.20}, then $|u| = \widetilde{u}$. In particular, $|u|$ is a decreasing function.
\end{theorem}

\begin{remark}
Choosing $w(n)=|n|^\alpha$ for non-negative $\alpha$ in Theorem \ref{thm4.4} gives 
\begin{equation}\label{4.21}
    \sum_{n \in \Z^+}|u(n)-u(n+1)|^p |n|^\alpha \geq \sum_{n \in \Z^+}|\widetilde{u}(n)-\widetilde{u}(n+1)|^p |n|^\alpha,
\end{equation}
which is the discrete analogue of Corollary 8.1 of the paper \cite{alvino1}. Inequalities of the type \eqref{4.21} might be useful in studying a discrete first order interpolation inequalities with weights. 
\end{remark}

\begin{proof}
The proof goes via the discrete co-area formula. We interpret $\Z^+$ as a graph with $V = \Z^+$. Two points $x$ and $y$ are connected by an edge iff $|x-y| = 1$. Using the co-area formula \eqref{4.19}, we get ($\partial(\{u>t\}) := E(\{u>t\},\{u>t\}^c)$) 
\begin{equation}\label{4.22}
    \sum_{x \sim y}|u(x)-u(y)|^p w((x+y-1)/2) = \int_0^\infty \sum_{(x, y) \in \partial(\{u>t\})} |u(x)-u(y)|^{p-1}w((x+y-1)/2) dt.
\end{equation}
Using the reverse triangle inequality, we get $\sum \limits_{n \in \Z^+}|u(n)-u(n+1)|^p w(n) \geq \sum \limits_{n \in \Z^+}||u|(n)-|u|(n+1)|^p w(n)$. Therefore, without loss of generality, we can assume that $u$ is non-negative. Since $u$ vanishes at infinity, we can arrange the values $u$ takes in decreasing order: $t_1 > t_2 > t_3 ....$ and so on. Fix $ t_{i+1} \leq t < t_i$. Let $k$ be the size of the level set of $u$, i.e., $|\{u>t\}| = k$. Then it is easy to see that there will exist $x_1 \geq k-1$ such that $x_1 \in \{u>t\}$ and $x_1+1 \in \{u>t\}^c$. This is a consequence of $u$ vanishing at infinity. This implies that,
\begin{equation}\label{4.23}
    \begin{split}
        \frac{1}{2}\sum_{(x, y) \in \partial(\{u>t\})} |u(x)-u(y)|^{p-1}w((x+y-1)/2) & \geq |u(x_1)-u(x_1+1)|^{p-1}w(x_1) \\
        & \geq |t_i -t_{i+1}|^{p-1}w(k-1).   
    \end{split}
\end{equation}
Further, notice that $\{\widetilde{u}>t\} = [0, k-1]\cap \Z^+$ and
\begin{equation}\label{4.24}
    \frac{1}{2}\sum_{(x, y) \in \partial(\{\widetilde{u}>t\})} |\widetilde{u}(x)-\widetilde{u}(y)|^{p-1}w((x+y-1)/2)
    = |t_i - t_{i+1}|^{p-1}w(k-1). 
\end{equation}
Equation \eqref{4.23} and \eqref{4.24} imply 
\begin{equation}\label{4.25}
    \sum_{(x, y) \in \partial(\{u>t\})} |u(x)-u(y)|^{p-1}w((x+y-1)/2) \geq \sum_{(x, y) \in \partial(\{\widetilde{u}>t\})} |\widetilde{u}(x)-\widetilde{u}(y)|^{p-1}w((x+y-1)/2), 
\end{equation}
for all $t>0$. Equation \eqref{4.25} along with the coarea formula \eqref{4.22} completes the proof.

Next we characterize those functions which attain equality in \eqref{4.20}. Let $u$ be a non-negative function vanishing at infinity which produces equality in \eqref{4.20}. We have
\begin{equation}\label{4.26}
\begin{split}
    \int_0^\infty \sum_{(x, y) \in \partial(\{u>t\})} |u(x)-u(y)|^{p-1}&w((x+y-1)/2) dt \\
    &= \int_0^\infty \sum_{(x, y) \in \partial(\{\widetilde{u}>t\})} |\widetilde{u}(x)-\widetilde{u}(y)|^{p-1}w((x+y-1)/2) dt.
\end{split}
\end{equation}
Identity \eqref{4.26} along with \eqref{4.25} implies that
\begin{equation}\label{4.27}
    \sum_{(x, y) \in \partial(\{u>t\})} |u(x)-u(y)|^{p-1}w((x+y-1)/2) = \sum_{(x, y) \in \partial(\{\widetilde{u}>t\})} |\widetilde{u}(x)-\widetilde{u}(y)|^{p-1}w((x+y-1)/2),
\end{equation}
for a.e. $t > 0$. Let $t_{i+1} \leq  t < t_i $, where $t_1>t_2>t_3>...$ are the values of function $u$ arranged in decreasing order. Let $|\{u>t\}| = k$. Then we claim that there cannot exist $x \geq k$ such that $u(x) > t$. We will prove this via contradiction. Let us assume that such an $x$ exists. It is easy to see that there will exist $x_1 \geq k$ and $x_2 \leq x$ such that $x_1, x_2 \in \{u>t\}$ and $x_1+ 1, x_2 -1 \in \{u>t\}^c$. This is a consequence of the fact that the level set of $u$ has size $k$, which is a finite number. This immediately implies that
\begin{align*}
   \frac{1}{2}\sum_{(x, y) \in \partial(\{u>t\})} |u(x)-u(y)|^{p-1}w((x+y-1)/2) \geq &|u(x_1)-u(x_1 +1)|^{p-1}w(x_1) \\
   & + |u(x_2)-u(x_2-1)|^{p-1}w(x_2-1)\\
   & > |t_i - t_{i+1}|^{p-1}w(k-1). 
\end{align*}
The last step uses the fact that $w$ is a positive function. This contradicts identity \eqref{4.27}. Therefore $\{u>t\} = \{\widetilde{u}>t\}$ for a.e. $t>0$. Finally, using layer-cake representation for $u$ and $\widetilde{u}$, we get $u(x) = \widetilde{u}(x)$. 

Let us assume that function $u$ produces equality in \eqref{4.20}, then $|u|$ will also produce equality in \eqref{4.20} which would imply that $|u| = \widetilde{u}$.  
\end{proof}

\begin{remark}
We would like to point out that \eqref{4.20} is not true in general for higher order difference operators. Consider the following counter example: let $u(0) = u(2) = \alpha$ and $u(1) = \alpha + \delta$ and define $u = 0$ rest everywhere. Then 
\begin{align*}
    \sum_{n \in \Z^+}|\Delta u|^2  = \alpha^2 + 5\delta^2 + (\alpha-\delta)^2 \hspace{5pt} \text{and} \hspace{5pt} \sum_{n \in \Z^+}|\Delta \widetilde{u}|^2 = 2(\alpha^2 + \delta^2),   
\end{align*}
where we define $\Delta u(n) := 2u(n)-u(n-1)- u(n+1)$ for $n \geq 1$ and $\Delta u(0):=  u(0)-u(1)$. It is easy to check that $\sum \limits_{n \in \Z^+}|\Delta \widetilde{u}|^2 > \sum \limits_{n \in \Z^+} |\Delta u|^2$
for $0 < \delta \leq \alpha/2$.
\end{remark}

\subsection{Fourier rearrangement}\label{subsec:fourier rearrangement}
In this subsection, we introduce another rearrangement called the \emph{Fourier rearrangement} on $\Z$ and prove the corresponding Polya-Szeg\H{o} inequality. As we will see, under this rearrangement, norms of higher order operators decrease as well, something which was missing in the decreasing rearrangement introduced in the last subsection. But on the downside, the description of Fourier rearrangement of a function is not as straightforward as that of the decreasing rearrangement. Many basic yet important properties like equimeasurability of level sets fail to hold in this setting. 

Before mentioning the details, we would like to mention \cite{lenzmann}, where an analogous theory of Fourier rearrangement on $\R^d$ has been developed and applied. In this chapter we restrict ourselves to one-dimensional setting. Fourier rearrangement in higher dimensions is discussed in Appendix \ref{appendix:D}.

Let $u \in \ell^2(\mathbb{Z})$. Then we define its Fourier transform $F(u) \in L^2((-\pi, \pi))$ by 
\begin{align*}
    F(u)(x) := (2\pi)^{-\frac{1}{2}}\sum_{n \in \mathbb{Z}} u(n) e^{-inx} \hspace{19pt} x \in(-\pi, \pi),
\end{align*}
and the inverse of $F$ is given by
\begin{equation}
    F^{-1}(u)(n) = (2\pi)^{-\frac{1}{2}} \int_{-\pi}^\pi u(x) e^{inx} dx.
\end{equation}
We define the \emph{Fourier rearrangement} $u^\#$ \emph{of the function} $u \in \ell^2(\mathbb{Z})$ as 
\begin{align*}
    u^\# := F^{-1}((F(u))^*),
\end{align*}
where $f^*$ denotes the symmetric-decreasing rearrangement of the function $f$ on the real line (see chapter 3 of \cite{liebloss}).

Some basic properties of $u^\#$:
\begin{enumerate}
    \item $u^\#$ is radial and real valued. To prove this fact consider,
    \begin{align*}
        u^\#(-n) &= (2\pi)^{\frac{1}{2}} \int_{-\pi}^\pi F(u)^*(x) e^{-inx}dx = (2\pi)^{\frac{1}{2}}  \int_{-\pi}^\pi F(u)^*(-x) e^{inx}dx = u^\#(n).\\
        \overline{u^\#(n)} &= (2\pi)^{-\frac{1}{2}}\int_{-\pi}^\pi \overline{F(u)^*(x)} e^{-inx}dx = (2\pi)^{\frac{1}{2}} \int_{-\pi}^\pi F(u)^*(-x) e^{inx}dx = u^\#(n).
    \end{align*}
    \item $|u^\#(n)| \leq |u^\#(0)|$, i.e., $|u^\#|$ takes its maximum value at origin. This follows directly from the formula for the inverse of fourier transform:
    \begin{align*}
        |u^\#(n)| \leq (2\pi)^{-\frac{1}{2}} \int_{-\pi}^\pi |F(u)^*| dx = u^\#(0). 
    \end{align*}
    \item Using Parseval's identity and equimeasurability of symmetric decreasing rearrangement we get,   
    \begin{equation}\label{4.29}
        \sum_{n \in \mathbb{Z}} |u|^2 = \sum_{n \in \mathbb{Z}}|u^\#|^2. 
    \end{equation}
    \item  We have the following \emph{Hardy-Littlewood inequality}:
    \begin{equation}\label{4.30}
        \Big|\sum_{n \in \mathbb{Z}} u(n) \overline{v}(n)\Big| \leq \sum_{n \in \mathbb{Z}} u^\#(n) v^\#(n). 
    \end{equation}
    Inequality \eqref{4.30} follows from Parseval's identity and the Hardy-Littlewood inequality for symmetric decreasing rearrangement.
    \item $\ell^2(\Z)$ distance decreases under Fourier rearrangement, i.e.,
    \begin{align*}
        \sum_{n \in \mathbb{Z}}|u^\#-v^\#|^2 \leq \sum_{n \in \mathbb{Z}} |u - v|^2,
    \end{align*}
    for all $u, v \in \ell^2(\Z)$. It again follows from Parseval's identity and the $L^2(\R)$ contraction property of symmetric decreasing rearrangement.
\end{enumerate}
Before stating the main result, we define the discrete analogue of the first and second order `derivatives' on integers.
\begin{definition}
Let $u: \Z \rightarrow \C$ be a function defined on integers, then its \emph{first and second order derivative} are defined by
\begin{equation}
    Du(n) := u(n)-u(n-1),
\end{equation}
and
\begin{equation}
    \Delta u:= 2u(n)-u(n-1)-u(n+1)
\end{equation}
respectively.
\end{definition}

\begin{theorem}\label{thm4.5}
Let $u \in \ell^2(\mathbb{Z})$ and $k \geq 0$ then we have
\begin{equation}\label{4.33}
    \sum_{n \in \mathbb{Z}} |\Delta^k u(n)|^2 \geq \sum_{n \in \mathbb{Z}} |\Delta^k u^\#(n)|^2, 
\end{equation}
and 
\begin{equation}\label{4.34}
    \sum_{n \in \mathbb{Z}} |D(\Delta^k u)(n)|^2 \geq \sum_{n \in \mathbb{Z}} |D(\Delta^k u^\#)(n)|^2.
\end{equation}
Moreover equality holds in \eqref{4.33} or \eqref{4.34} if and only if $|F(u)(x)| = F(u)^*(x)$ for a.e. $x \in (-\pi, \pi)$. 
\end{theorem}

\begin{remark}
Putting $k=0$ in inequality \eqref{4.34} gives the Polya-Szeg\H{o} inequality \eqref{4.12} for the Fourier rearrangement $u^\#$.
\end{remark}

\begin{remark}
We would like to mention that Theorem \ref{thm4.5} is true for general class of operators. Let $T: \ell^2(\Z) \rightarrow \ell^2(\Z)$ be an operator. Assume that that there exists a measurable function $\omega : \R \rightarrow \C$ such that $|\omega(x)|$ is radial, strictly increasing with respect to $|x|$ and 
\begin{align*}
    F(T(u))(x) = \omega(x)F(u)(x),
\end{align*}
for all $x \in (-\pi, \pi)$. Then we have
\begin{align*}
    \sum_{n \in \Z} |Tu|^2 \geq \sum_{n \in \Z}|Tu^\#|^2.
\end{align*}
It is easy to see that operators $\Delta^k$ and $D \Delta^k$ satisfy the above conditions. One can also put assumptions on the parameters $a,b,c$ for which the above conditions are satisfied by \emph{Jacobi operators}: $Ju(n) := au(n) + bu(n-1) + cu(n+1)$. 
\end{remark}

\begin{proof}
We begin by computing the Fourier transforms of $Du$ and $\Delta u$. 
\begin{align*}
    Du(n) = u(n)-u(n-1) &= (2\pi)^{-\frac{1}{2}}\int_{-\pi}^\pi F(u)(1-e^{-ix}) e^{inx} dx,\\
    \Delta u(n) = 2u(n) - u(n-1) - u(n+1) &= (2\pi)^{-\frac{1}{2}} \int_{-\pi}^\pi F(u)(2-e^{-ix}-e^{ix}) e^{inx} dx.
\end{align*}
Therefore we have 
\begin{align*}
    F(Du) = F(u)(1-e^{-ix}) \hspace{19pt} \text{and} \hspace{19pt} 
    F(\Delta u) = F(u)(2-e^{ix}-e^{-ix}).
\end{align*}
Using Parseval's identity and applying the above formulas iteratively we obtain
\begin{equation}\label{4.35}
    \sum_{n \in \mathbb{Z}} |\Delta^k u|^2 = \int_{-\pi}^\pi |F(\Delta^k u)|^2 dx = 4^{2k} \int_{-\pi}^\pi |F(u)|^2 \sin^{4k}(x/2) dx,
\end{equation}
and
\begin{equation}\label{4.36}
    \sum_{n \in \mathbb{Z}} |D\Delta^k u|^2 = \int_{-\pi}^\pi |F(D\Delta^k u)|^2 dx = 4^{2k+1} \int_{-\pi}^\pi |F(u)|^2 \sin^{4k+2}(x/2) dx.
\end{equation}
Therefore proving \eqref{4.33} and \eqref{4.34} reduces to showing
\begin{equation}\label{4.37}
    \int_{-\pi}^\pi |F(u)|^2 \sin^{4k}(x/2) dx \geq \int_{-\pi}^\pi |F(u)^*|^2 \sin^{4k}(x/2) dx,
\end{equation}
and 
\begin{equation}\label{4.38}
    \int_{-\pi}^\pi |F(u)|^2 \sin^{4k+2}(x/2) dx \geq \int_{-\pi}^\pi |F(u)^*|^2 \sin^{4k+2}(x/2) dx
\end{equation}
respectively. Next we prove a general result which implies \eqref{4.37} and \eqref{4.38}.

\textbf{Claim:} Let $f$ be a non-negative measurable function vanishing at infinity and let $g$ be a non-negative measurable function which is  radially increasing. Then we have
\begin{equation}\label{4.39}
    \int_\mathbb{R} f(x)g(x)  dx \geq \int_\mathbb{R} f^*(x) g(x) dx.
\end{equation}
\begin{proof}
Using the Layer-cake representation formula proving \eqref{4.34} reduces to proving \eqref{4.34} for $f = \chi_\Omega$ for measurable set $\Omega $ of finite measure. Let $\Omega^*$ be the ball centered at origin such that $|\Omega^*| = |\Omega|$. Using the facts that $g$ is radially increasing and $|\Omega \setminus \Omega^*| = |\Omega^*\setminus \Omega|$ we obtain
\begin{align*}
    \int_{\Omega \setminus \Omega^*} g(x) dx \geq g(R)|\Omega \setminus \Omega^*| = g(R)|\Omega^* \setminus \Omega| \geq \int_{\Omega^*\setminus \Omega} g(x) dx,
\end{align*}
where $R$ is the radius of the ball $\Omega^*$. Now we have
\begin{align*}
    \int_\Omega g(x) dx = \int_{\Omega \setminus \Omega^*} g(x) dx + \int_{\Omega \cap \Omega^*} g(x) dx \geq \int_{\Omega^* \setminus \Omega} g(x) dx + \int_{\Omega^* \cap \Omega} g(x) dx = \int_{\Omega^*} g(x) dx.
\end{align*}
This prove inequality \eqref{4.39}. 
\end{proof}

Now coming back to inequality \eqref{4.37}, we define $f := |F(u)|^2 $ on the interval $(-\pi,\pi)$ and zero in the complement, and $g : = 4\sin^{4k}(x/2)$ in the interval $(-\pi, \pi)$ and extend $g(x)$ in a radial and strictly increasing way in the complement of $(-\pi, \pi)$. Now applying \eqref{4.39} for this choice of $f$ and $g$ yields \eqref{4.37}. One can similarly prove \eqref{4.38}, thereby completing the proof of inequalities \eqref{4.33} and \eqref{4.34}.

Next we study the equality cases in \eqref{4.33} and \eqref{4.34}. For that, we will study the equality case in inequality \eqref{4.39}. Let $f$ be a non-negative function which produces equality in \eqref{4.39}. We further assume that $g(x)$ is a strictly increasing function, then we have for a.e. $t>0$,
\begin{equation}
    \int_{\{f>t\}} g(x) dx = \int_{\{f^*>t\}} g(x) dx.
\end{equation}
This implies that $\int_{\Omega \setminus \Omega^*} g(x) = g(R)|\Omega\setminus \Omega^*| $, with $\Omega := \{f>t\}$ and $R$ being the radius of the ball $\Omega^*$. Now we claim that $|\Omega \setminus \Omega^*| = 0$, if this is not true then the fact that $g$ is strictly increasing would imply that $\int_{\Omega \setminus \Omega^*} g(x) > g(R)|\Omega\setminus \Omega^*|$, which leads to a contradiction. Therefore, we have $|\Omega \setminus \Omega^*| =0$, which implies that 
\begin{align*}
    \chi_{\{f>t\}}(x) = \chi_{\{f^*>t\}}(x),
\end{align*}
for a.e. $(x,t) \in \R \times (0, \infty)$. Finally using the layer-cake representation for $f$ and $f^*$, we find that equality  in \eqref{4.39} implies that $f(x) = f^*(x)$ for a.e. $x \in \R$.

Let's get back to the equality case in \eqref{4.33}. Let $u$ be a function which produces equality in \eqref{4.33} then we have 
\begin{equation}
    \int_{-\pi}^\pi |F(u)|^2 \sin^{4k}(x/2) dx = \int_{-\pi}^\pi (|F(u)|^2)^* \sin^{4k}(x/2) dx,
\end{equation}
which implies that $|F(u)|^2 = (|F(u)|^2)^* = |F(u)^*|^2$. This gives us $|F(u)| = F(u)^*$ for a.e. $ x \in (-\pi, \pi)$. 
\end{proof}

\begin{remark}
We would like to point out that unlike the decreasing rearrangement, weighted Polya-Szeg\H{o} inequalities of the type \eqref{4.21} with power weights do not hold true in general for Fourier rearrangements. We give an example to support this statement. Let $u(0)=u(1)=\beta$ and $u$ vanishes everywhere else. Then $F(u) = 2\beta(2\pi)^{-\frac{1}{2}} e^{-ix/2}\cos(x/2)$ and its rearrangement is given by $F(u)^* = 2|\beta|(2\pi)^{-\frac{1}{2}}\cos(x/2)$. Finally, using the inversion formula we obtain
\begin{equation}
    u^\# = \frac{4|\beta|}{\pi}\frac{(-1)^n}{1-4n^2}.
\end{equation}
Consider,
\begin{align*}
    \sum_{n \in \mathbb{Z}}|u(n)-u(n-1)|^2|n|^\alpha = \beta^2\Big((1/2)^\alpha + (3/2)^\alpha \Big),
\end{align*}
and 
\begin{align*}
    \sum_{n \in \mathbb{Z}}|u^\#(n)-u^\#(n-1)|^2|n|^\alpha = \Big(\frac{4|\beta|}{\pi}\Big)^2\sum_{n \in \mathbb{Z}} |v(n)-v(n-1)|^2 |n|^\alpha,
\end{align*}
where $v(n):= \frac{(-1)^n}{1-4n^2}$.

It is easy to check that $|v(n)-v(n-1)|^2 |n|^\alpha$ grows as $|n|^{\alpha-4}$ as $|n| \rightarrow \infty$. So clearly, for $\alpha > 4$, we will have
\begin{align*}
    \sum_{n \in \mathbb{Z}}|u^\#(n)-u^\#(n-1)|^2|n|^\alpha >  \sum_{n \in \mathbb{Z}}|u(n)-u(n-1)|^2|n|^\alpha.   
\end{align*}
In fact we will have $\sum \limits_{n \in \mathbb{Z}}|u^\#(n)-u^\#(n-1)|^2|n-1/2|^\alpha = \infty$. Therefore we cannot have a real constant $c$ for which 
\begin{align*}
    c\sum_{n \in \mathbb{Z}}|u(n)-u(n-1)|^2|n|^\alpha \geq \sum_{n \in \mathbb{Z}}|u^\#(n)-u^\#(n-1)|^2|n|^\alpha
\end{align*}
holds.
\end{remark}

\begin{remark}
In this remark, we would illustrate how one can deduce formulae for some infinite sums using Fourier rearrangement. Let $u$ be the function as defined in the above remark, that is, $u(0)=u(1) = \beta$ and $u$ is zero everywhere else. Then we have
\begin{align*}
    u^\#(n) = \frac{4|\beta|}{\pi} \frac{(-1)^n}{1-4n^2}.
\end{align*}
Then using the fact that $\ell^2(\Z)$ norm is preserved under Fourier rearrangement \eqref{4.29} we get,
\begin{align*}
    2\beta^2 = \Big(\frac{4\beta}{\pi}\Big)^2 \sum_{n \in \Z} \frac{1}{(4n^2-1)^2},
\end{align*}
which implies,
\begin{equation}
    \sum_{n \in \Z} \frac{1}{(4n^2-1)^2} = \pi^2/8.
\end{equation}
\end{remark}

Next, we analyze how the $\ell^p(\Z)$ norm of $u$ changes under the Fourier rearrangement. For that, we will have to introduce the notion of symmetric rearrangement of a subset of $\Z$. We start by defining a labeling of $\Z$: Let $\epsilon: \Z \rightarrow \N$ be defined as 
\begin{align*}
    \epsilon(n) := 
    \begin{cases}
    &2n \hspace{37pt} \text{if} \hspace{9pt} n > 0\\
    &1-2n \hspace{19pt} \text{if} \hspace{9pt} n \leq 0.
    \end{cases}
\end{align*}
Let $E$ be a finite subset of $\Z$ of size $k$. Then the \emph{symmetric rearrangement} $E^*$ of the set $E$ is defined as the first $k$ elements of $\Z$ with respect to labeling $\epsilon$.

\begin{lemma}\label{lem4.2}
Let $E$ be a finite subset of $\Z$ and $E^*$ be its symmetric rearrangement, then there exists a constant $c$(independent of $u$ and $E$) such that 
\begin{equation}\label{4.44}
    \sum_{n \in E} |u|^2 \leq c \sum_{n \in E^*} |u^\#|^2,
\end{equation}
for all $u \in \ell^2(\Z)$.
\end{lemma}

\begin{remark}
The proof of Lemma \ref{lem4.2} is an adaptation of the proof of Theorem 1 in \cite{montgomery}. Similar ideas were used in \cite{frank2} to prove inequalities of the type \eqref{4.44} in the context of Fourier rearrangement on $\R^d$.  
\end{remark}

\begin{proof}
Let $\tau$ be a fixed non-negative number, which we will choose later. Define
\begin{equation}
    F(u_1) := \chi_{\{|F(u)| \geq \tau\}} F(u) \hspace{5pt}  \text{and} \hspace{5pt} F(u_2) := \chi_{\{|F(u)| < \tau\}} F(u).
\end{equation}
Clearly we have $u = u_1 + u_2$ and 
\begin{align*}
    \sum_{n \in E}|u|^2 \leq 2\sum_{n \in E}|u_1|^2 + 2 \sum_{n \in E}|u_2|^2.
\end{align*}
Consider
\begin{equation}\label{4.46}
    \sum_{n \in E}|u_1|^2 \leq ||u_1||_\infty^2 |E| \leq |E|(2\pi)^{-\frac{1}{2}} \Big(\int_{\{|F(u)| \geq \tau \}} |F(u)| dx\Big)^2, 
\end{equation}
and 
\begin{equation}\label{4.47}
    \sum_{n \in E}|u_2|^2 \leq \sum_{n \in \Z} |u_2|^2 = \int_{\{|F(u)| < \tau\}} |F(u)|^2 dx.
\end{equation}
It is easy to see that the terms on the RHS of \eqref{4.46} and \eqref{4.47} are invariant under symmetric decreasing rearrangement, i.e., 
\begin{align*}
    \int_{\{|F(u)| \geq \tau \}} |F(u)| dx = \int_{\{|F(u^\#)| \geq \tau \}} |F(u^\#)| dx,   
\end{align*}
and 
$$ \int_{\{|F(u)| < \tau\}} |F(u)|^2 dx =\int_{\{|F(u^\#)| < \tau\}} |F(u^\#)|^2 dx.$$
Let $v := u^\#$, then it only remains to prove the following inequalities:
\begin{equation}\label{4.48}
    \sum_{n \in E^*}|v|^2 \gtrsim \int_{\{|F(v)| < \tau\}} |F(v)|^2 dx,
\end{equation}
and
\begin{equation}\label{4.49}
    \sum_{n \in E^*}|v|^2 \gtrsim |E| \Big(\int_{\{|F(v)| \geq \tau \}} |F(v)| dx\Big)^2. 
\end{equation}
First we prove inequality \eqref{4.48}. Let $R := (|E|-1)/2$ if $|E|$ is an odd number and $R := (|E|-2)/2$ if $|E|$ is an even number, and let $K(n)$ := max$(0, 1- |n|/R)$. Then 
\begin{align*}
    F(k)(x) = \frac{1}{R} \frac{\sin^2(Rx/2)}{\sin^2(x/2)}.
\end{align*}
It is easy to see that $F(k)(x) \geq 4R/\pi^2$ for $ |x| \leq \pi/R$. Now choose $\tau$ = inf $\{F(v)(x):  |x| \leq \pi/R\}$ and consider,
\begin{align*}
    \sum_{n \in E^*} |v|^2 \geq \sum_{n \in \Z} K(n)|v|^2 &= \int_\R \int_\R F(v)(x) F(k)(x-y) F(v)(y) dy dx \\
    & \geq \int_{x > \pi/R} \int_{x-\pi/R \leq y \leq x} F(v)(x) F(k)(x-y) F(v)(y) dy dx.\\
    & \geq 2/\pi \int_{|x| > \pi/R} |F(v)(x)|^2 dx \geq 2/\pi \int_{\{|F(v)| < \tau\}} |F(v)(x)|^2 dx. 
\end{align*}
Next we bound $\sum_{n \in E^*} |v|^2$ from below by 
\begin{align*}
   \int_\R \int_\R F(v)(x) F(k)(x-y) F(v)(y) dy dx  \geq 4R/\pi^2 \int_{0 \leq x \leq \pi/R } \int_{0 \leq y \leq \pi/R} F(v)(x)F(v)(y) dy dx.
\end{align*}
Finally using 
\begin{align*}
   4R/\pi^2 \int_{0 \leq x \leq \pi/R } \int_{0 \leq y \leq \pi/R} F(v)(x)F(v)(y) dy dx &\geq R/\pi^2 \Big( \int_{|x| \leq \pi/R} F(v) dx \Big)^2  \\
   &\geq \frac{1}{4\pi^2} |E| \Big( \int_{\{|F(u)| \geq \tau\}} F(v) dx \Big)^2, 
\end{align*}
proves \eqref{4.49}.
\end{proof}

\begin{theorem}
Let $\varphi: [0, \infty) \rightarrow \R$ be a non-negative convex function with $\varphi(0) = \varphi'(0)= 0$. Then, for all $u \in \ell^2(\Z)$ we have
\begin{equation}\label{4.50}
    \sum_{n \in \Z}\varphi(|u|^2) \leq  \sum_{n \in \Z} \varphi(c|u^\#|^2), 
\end{equation}
where $c$ is the constant in \eqref{4.44}.
\end{theorem}

\begin{proof}
Consider,
\begin{align*}
    \phi(s) = \int_{0}^s \phi'(t) dt &= - \int_0^s \phi'(t)(s-t)'dt\\
    &= \int_0^s \phi''(t)(s-t)dt = \int_0^\infty     \phi''(t)(s-t)_{+} dt.
\end{align*}
Choosing $s= |u|^2$ and summing both sides we obtain,
\begin{equation}\label{4.51}
    \sum_{n \in \mathbb{Z}}\phi(|u|^2) = \int_{\mathbb{R}} \sum_{n \in \mathbb{Z}}(|u|^2-t)_{+} \phi''(t) dt.
\end{equation}
Applying \eqref{4.44} with $E = \{|u|^2 >t\}$ we get, 
\begin{align*}
    \sum_{n \in \Z}(|u|^2-t)_+ = \sum_{n \in \{|u|^2 >t\}} (|u|^2 -t) \leq \sum_{ n \in \{|u|^2 > t\}^*} (c|u^\#|^2-t) \leq \sum_{n \in \Z}(c|u^\#|^2-t)_+. 
\end{align*}
Plugging the above estimate in \eqref{4.51} completes the proof.
\end{proof}

\begin{corollary}\label{cor4.1}
For $p > 2$, there exists a constant $c(p)$ such that 
\begin{equation}\label{4.52}
    \sum_{n \in \Z}|u|^p \leq c(p) \sum_{n \in \Z} |u^\#|^p.
\end{equation}
\end{corollary}

\begin{proof}
Apply \eqref{4.50} with $\varphi(x):= x^{p/2}$.
\end{proof}

\begin{remark}
We end this section with an \emph{open problem}. The estimates done above are very crude, and the constant $c(p)$ in \eqref{4.52} is an exponential function of $p$. It would be a worthwhile effort to find the sharp constant in \eqref{4.52}.
\end{remark}

\subsection{Symmetric-Decreasing rearrangement}\label{subsec:symmetric-decreasing rearrangement}
In this section, we combine two rearrangements defined in the above two subsections to construct a rearrangement which is both spherically symmetric and decreasing. More precisely, 
Let $u \in \ell^2(\mathbb{Z})$. We define 
\begin{equation}
    v(n) := \widetilde{\Big(u^\#|_{\Z^+}\Big)}. 
\end{equation}
Then the $\emph{symmetric-decreasing rearrangement}$ $u^*$ of the function $u$ is defined as 
\begin{equation}
    u^*(n) := 
    \begin{cases}
    v(n) \hspace{28pt} \text{if} \hspace{5pt} n \geq 0\\
    v(-n) \hspace{19pt} \text{if} \hspace{5pt} n < 0.\\
    \end{cases}
\end{equation}
Properties of function $u^*$:
\begin{enumerate}
    \item $u^*$ is always non-negative.
    \item It is easy to see that $u^*$ is radially symmetric and decreasing, that is, 
    \begin{align*}
        u^*(x) = u^*(y) \hspace{29pt} \text{if} \hspace{5pt} |x|=|y|,
    \end{align*}
    and 
    \begin{align*}
        u^*(x) \geq u^*(y) \hspace{29pt} \text{if} \hspace{5pt} |x| \leq |y|. 
    \end{align*}
    \item $\ell^2(\Z)$ norm is preserved, that is,
    \begin{equation}\label{4.55}
        \sum_{n \in \Z}|u|^2 = \sum_{n \in \Z}|u^*|^2.
    \end{equation}
    This is a consequence of the fact that both the decreasing and Fourier rearrangements preserve the $\ell^2$ norm plus the fact that $|u^\#(n)| \leq |u^\#(0)|$.
    \item (\emph{Hardy-Littlewood inequality}) Let $u,v \in l^2(\Z)$, then we have
    \begin{equation}\label{4.56}
        \Big|\sum_{n \in \Z} u(n) \overline{v(n)}\Big| \leq \sum_{n \in \Z} u^*(n) v^*(n).
    \end{equation}
    This is again a consequence of Hardy-Littlewood inequality for decreasing and Fourier rearrangements along with $|u^\#(n)| \leq |u^\#(0)|$.\\
    \item $\ell^2(\Z)$ distance decreases under symmetric-decreasing rearrangement, 
    \begin{equation}
        \sum_{n \in \Z} |u^* - v^*|^2 \leq \sum_{n \in \Z}|u-v|^2,
    \end{equation}
    for $u, v \in \ell^2(\Z)$, this follows from inequality \eqref{4.55} and \eqref{4.56}.
    
\end{enumerate}
\begin{theorem}
For $p >2$, there exists a constant $c(p)$ such that,
\begin{equation}\label{4.58}
    \sum_{n \in \Z}|u|^p \leq c(p) \sum_{n \in \Z}|u^*|^p,
\end{equation}
and 
\begin{equation}\label{4.59}
     \sum_{n \in \Z}|u(n)-u(n-1)|^2 \geq \sum_{n \in \Z}|u^*(n)-u^*(n-1)|^2.
\end{equation}
for all $u \in \ell^2(\Z)$. 
\end{theorem}

\begin{proof}
The proof of \eqref{4.58} follows from Corollary \ref{cor4.1} and the fact that $\ell^p(\Z)$ norm is preserved under decreasing rearrangement.

Next we prove inequality \eqref{4.59}. Polya-Szeg\H{o} inequality for decreasing rearrangement \eqref{4.20} along with the radiality of $u^\#$ gives us
\begin{align*}
    \sum_{n \in \Z}|u^\#(n)-u^\#(n-1)|^2 &= 2\sum_{n \in \Z^+}|u^\#(n)-u^\#(n+1)|^2 \\
    & \geq 2\sum_{n \in \Z^+}|u^*(n)-u^*(n+1)|^2\\
    &= \sum_{n \in \Z}|u^*(n)-u^*(n-1)|^2.
\end{align*}
The above inequality along with Polya-Szeg\H{o} inequality \eqref{4.34} for Fourier rearrangement completes the proof.
\end{proof}

\begin{remark}
Let $u$ be a function which produces equality in \eqref{4.59}. Then $u$ will produce equality in the Polya-Szeg\H{o} inequality for Fourier rearrangement \eqref{4.34} and $u^\#$ will produce equality in the Polya-Szeg\H{o} inequality for decreasing rearrangement \eqref{4.20}. This would imply that $|F(u)| = F(u)^*$ and $u^\# = u^*$. We couldn't derive any useful information from these two conditions. It would be nice to understand the equality case of \eqref{4.59} in a more direct way.   

\end{remark}

\subsection{A rearrangement proof of weighted Hardy inequality}\label{subsec:applications}
In this section, we apply weighted Polya-Szeg\H{o} inequality \eqref{4.21} to prove Hardy inequality with the weights $n^\alpha$. In what follows, $C_c(\Z^+)$ denotes the space of finitely supported functions on $\Z^+$.

\begin{theorem}
Let $1 <\alpha \leq 2$ and $u \in C_c(\Z^+)$ with the condition that $|u(0)|^2 = max\hspace{3pt}|u|$.Then the following Hardy inequality is true
\begin{equation}\label{4.60}
    \sum_{n=1}^\infty |u(n)-u(n-1)|^2 n^\alpha \geq (\alpha-1)^2/4 \sum_{n=1}^\infty |u|^2 n^{\alpha-2}.  
\end{equation}
\end{theorem}

\begin{proof}
Let $w(n):= n^{\alpha-2}$ on $\N$ and $w(0):=1$. Then using the Hardy-Littlewood inequality \eqref{4.17} and the weighted Polya-Szeg\H{o} inequality \eqref{4.21} for decreasing rearrangement $\widetilde{u}$ we get
\begin{align*}
    \sum_{n=1}^\infty|u|^2 n^{\alpha-2} + |u(0)|^2 =\sum_{n=0}^\infty|u|^2 w &\leq \sum_{n=0}^\infty |\widetilde{u}|^2 w(n) = \sum_{n=1}^\infty|\widetilde{u}|^2 n^{\alpha-2} + |\widetilde{u}(0)|^2,\\
    \sum_{n=1}^\infty |u(n)-u(n-1)|^2 n^\alpha &\geq \sum_{n=1}^\infty |\widetilde{u}(n)-\widetilde{u}(n-1)|^2 n^\alpha.
\end{align*}
Since $|u|$ is maximized at origin, it is sufficient to prove \eqref{4.60} for $\widetilde{u}$. Therefore from now on, we can assume that $u$ is a decreasing function. 

Let $Lu$ be the linear interpolation of $u$ on $\R$, i.e.,
\begin{align*}
 Lu(x) := u(n) + (x-n)(u(n)-u(n-1)) \hspace{5pt} \text{for} \hspace{5pt} x \in [n-1, n] . 
\end{align*}
Using the linearity of $Lu$ and the fact that $u$ is a decreasing function we can conclude that
\begin{equation}
    \sum_{n=1}^\infty|u|^2 n^{\alpha-2} \leq \int_0^\infty |Lu|^2 x^{\alpha-2} dx
\end{equation}
and
\begin{equation}
    \sum_{n=1}^\infty |u(n)-u(n-1)|^2 n^\alpha \geq \int_0^\infty |(Lu)'|^2 x^\alpha dx.
\end{equation}
We complete the proof by applying the weighted Hardy inequality on $Lu$ on the interval $(0, \infty)$.
\end{proof}

%% file: higher_Rearrangement_inequality/higher_Rearrangement_inequality.tex
\chapter{Rearrangement inequality on $d$-dimensional lattice graph}\label{ch:higher-rearrangement}

\section{Introduction}
This chapter concerns the questions raised in Chapter \ref{ch:1d-rearrangement} around Polya-Szeg\H{o} inequality on two dimensional lattice graph. To recall, we have two `natural' ways to rearrangement functions on $\Z^2$: \emph{spiral rearrangement} and \emph{Wang-Wang rearrangement}, which comes from the solution of edge and vertex isoperimetric problem on $\Z^2$ respectively.
\begin{figure}[h!]
\centering
\begin{minipage}[r]{.35\textwidth}
  \begin{tikzpicture}[scale=1.4]
  \node at (0,0) {1};
    \node at (0.5,0) {2};
    \node at (0.5,0.5) {3};
  \node at (0,0.5) {4};
  \node at (-0.5,0.5) {5};
  \node at (-0.5,0) {6};
  \node at (-0.5,-0.5) {7};
  \node at (0,-0.5) {8};
  \node at (0.5,-0.5) {9};
    \node at (1,-0.5) {10};
    \node at (1,0) {11};
    \node at (1,0.5) {12};
    \node at (1,1) {13};
        \node at (0.5,1) {14};
    \node at (0,1) {15};
    \node at (-0.5,1) {16};
  \draw [->] (1 ,-0.35) -- (1,-0.15);
    \draw [->] (1 ,0.15) -- (1,0.35);
    \draw [->] (1 ,1.3/2) -- (1,1.7/2);
    \draw [->] (1.7/2,2/2) -- (1.4/2,2/2);
        \draw [->] (0.7/2,2/2) -- (0.4/2,2/2);
    \draw [->] (-0.3/2,2/2) -- (-0.6/2,2/2);
        \draw [->] (-1.3/2,2/2) -- (-1.8/2,2/2);
  \draw [->] (0.15, 0) -- (0.35, 0);
    \draw [->] (1/2, 0.3/2) -- (1/2, 0.7/2);
  \draw [->] ( 0.7/2,1/2) -- (0.3/2,1/2);
 \draw [->] ( -0.3/2,1/2) -- (-0.7/2,1/2);
 \draw [->] ( -1/2,0.7/2) -- (-1/2,0.3/2);
 \draw [->] ( -1/2,-0.3/2) -- (-1/2,-0.7/2);
 \draw [->] ( -0.7/2,-1/2) -- (-0.3/2,-1/2);
  \draw [->] ( 0.3/2,-1/2) -- (0.7/2,-1/2);
  \draw [->] ( 1.3/2 ,-1/2) -- (1.7/2,-1/2);
    \end{tikzpicture}
    \vspace{0pt}
\end{minipage}
\begin{minipage}[l]{.35\textwidth}
\begin{tikzpicture}
\node at (-3,0) {};
\node at (0,0) {1};
\node at (0,0.5) {2};
\node at (0.5,0) {3};
\node at (-0.5,0) {4};
\node at (0,-0.5) {5};
\node at (0.5,0.5) {6};
\node at (-0.5,0.5) {7};
\node at (0,1) {8};
\node at (1,0) {9};
\node at (0.5,-0.5) {10};
\node at (-1,0) {11};
\node at (-0.5,-0.5) {12};
\node at (0,-1) {13};
\end{tikzpicture}
\end{minipage}
\caption{Left: Initial part of spiral labelling. Right: Initial segment of Wang-Wang labelling.} 
\label{fig:spiral and wang-wang}
\end{figure}
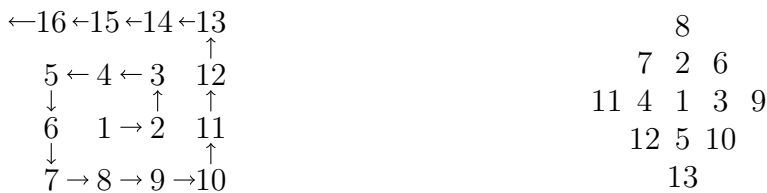

It is known that spiral rearrangement satisfy Polya-Szeg\H{o} for $p=1$, Wang-Wang satisfy Polya-Szeg\H{o} for $p=\infty$ and no rearrangement satisfty Polya-Szeg\H{o} for $p=2$ (see \cite{steinerberger} and \cite{hajaiej2}). These results motivates the search of a constant $c_p(\eta)$ (measuring the `quality' of a particular rearrangement) in 
\begin{equation}\label{5.1}
    \|\nabla f^*\|_p \leq c_p(\eta)\|\nabla f\|_p,
\end{equation}
with $f^*$ being the rearrangement of $f$ w.r.t. an arbitrary labelling $\eta$. More importantly, for a fixed $1\leq p \leq \infty$, amongst all rearrangements (labellings), which one gives the smallest possible constant $c_p(\eta)$ in \eqref{5.1}. Is there some kind of `phase transition' between spiral and Wang-Wang rearrangement as $p$ goes from $1$ to $\infty$: spiral being optimal for $p \in (1, p_0)$ and Wang-Wang being optimal for $p \in (p_0, \infty)$, for some $1< p_0< \infty$?

In this Chapter we find an explicit constant $c_p$ in \eqref{5.1}, for both spiral and Wang-Wang rearrangement. The method is quite robust and can be used to prove \eqref{5.1} for a large class of graphs and rearrangements on them (see Theorem \ref{thm5.3}). This abstract Theorem \ref{thm5.3} can further be used to derive results on $\Z^d$, for $d \geq 3$ (Corollary \ref{cor5.1}).

\begin{theorem}[Spiral Rearrangement]\label{thm5.1}
Let $f: \Z^2 \rightarrow \R$ be a function vanishing at infinity and let $f^*$ denote the spiral rearrangement of $f$. Then, for  all $p \geq 1$,
\begin{equation}\label{5.2}
    \left\|\nabla f^*\right\|_p \leq 4^{1+1/p} \left\|\nabla f\right\|_p.
\end{equation}
\end{theorem}
The constant $4^{1+1/p}$ is not sharp: it is known to be $1$ when $p=1$. It would be of interest to obtain improved bounds. A natural question is whether the optimality of the spiral rearrangement for $p=1$ extends beyond this endpoint: is it optimal in some range $p \in (1,p_0)$? 
\begin{figure}[h!]
\centering
  \begin{tikzpicture}[scale=1]
  \node at (2,0.5) {\includegraphics[width=0.25\textwidth]{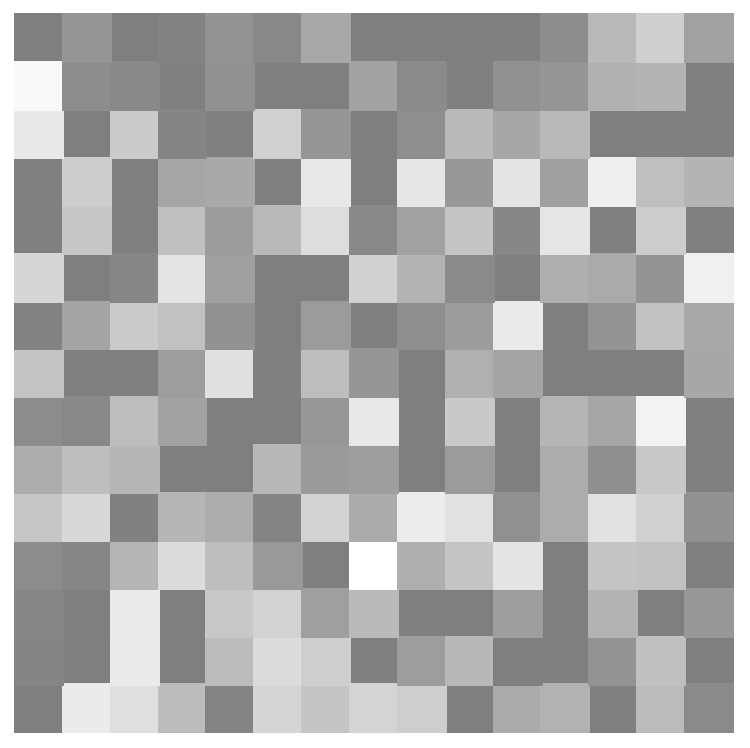}};
    \node at (7,0.5) {\includegraphics[width=0.25\textwidth]{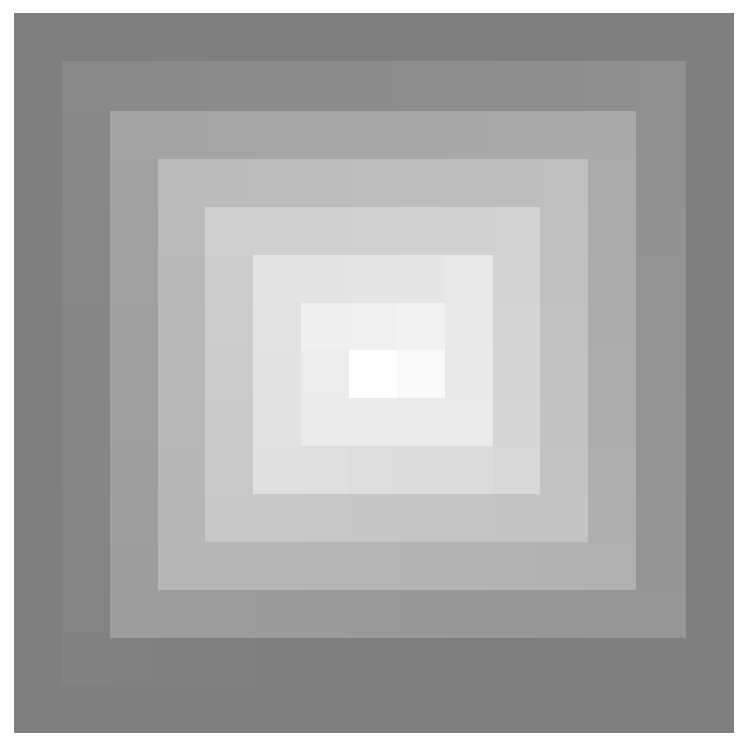}};
    \end{tikzpicture}
    \caption{Left: an example of a function $f:\mathbb{Z}^2 \rightarrow \mathbb{R}_{\geq 0}$ compactly supported around the origin (smaller values are brighter). Right: the same function rearranged using the spiral rearrangement.}
      \label{fig:example}
\end{figure}
 
\begin{theorem}[Wang-Wang rearrangement]\label{thm5.2} 
Let $f: \Z^2 \rightarrow \R_{}$ be a function vanishing at infinity and let $f^*$ denote the Wang-Wang rearrangement of $f$. Then, for  all $p \geq 1$, 
\begin{equation}
    \left\|\nabla f^*\right\|_p \leq 2^{1/p} \left\|\nabla f\right\|_p.
\end{equation}
\end{theorem}
There is no reason to believe that the constant $2^{1/p}$ is sharp if $p<\infty$. Just as it is interesting whether the spiral rearrangement is optimal for $p \in (1, p_0)$ one could wonder whether the Wang-Wang rearrangement is optimal for $p \in (p_1, \infty)$. A particularly interesting question is whether $p_0 = p_1$: this would correspond to the theory of rearrangements on $\mathbb{Z}^2$ having a particularly simple solution that interpolates between these two rearrangements. Or
are there other, yet unknown, rearrangements that are optimal in the intermediate range?
\begin{figure}[h!]
\centering
\begin{minipage}[l]{.35\textwidth}
\begin{tikzpicture}
\node at (-3,0) {};
\node at (0,0) {1};
\node at (0,0.5) {2};
\node at (0.5,0) {3};
\node at (-0.5,0) {4};
\node at (0,-0.5) {5};
\node at (0.5,0.5) {6};
\node at (-0.5,0.5) {7};
\node at (0,1) {8};
\node at (1,0) {9};
\node at (0.5,-0.5) {10};
\node at (-1,0) {11};
\node at (-0.5,-0.5) {12};
\node at (0,-1) {13};
\end{tikzpicture}
\end{minipage}
\begin{minipage}[r]{.35\textwidth}
\begin{center}
\begin{tabular}{c  c  c  c c  c c}
$n$ & 1 & 2 & 3 & 4 & 5 & 6\\
$\partial_V^n$ &  4  &  6 &  7 & 8 & 8 & 9\\
\end{tabular}
\end{center}
\end{minipage} 
\caption{Left: start of the Wang-Wang enumeration (Wang \& Wang \cite{wang}). Right: $\partial_V^n$, the size minimal vertex-boundary, for small $n$.} 
\label{fig:smaln}
\end{figure}

We note that the proof of Theorem \ref{thm5.2} is relatively concrete and there is some hope of generalizing it to higher dimensions (where a napkin computation would suggest the bound $d^{1/p}$ for $\mathbb{Z}^d$). Making this precise would require some additional combinatorial insights into the Wang-Wang enumeration of $(\mathbb{Z}^d, \ell^1)$.

Our proof of Theorem \ref{thm5.1} is very concrete in each step and illustrates the core of the main argument. Abstracting the proof of Theorem \ref{thm5.1}, we arrive at the following general result.
\begin{theorem}\label{thm5.3} Let $G=(V,E)$ be an infinite, connected graph with countable vertex set. Assume moreover that $\partial_V^{n+1} \geq \partial_V^{n} \geq \partial_V^1 \geq 2$ and that all vertices have uniformly bounded degree $\deg(v) \leq D$. Suppose now that $v_1, v_2, \dots$ is an enumeration of vertices with the property that, for some universal constant $c > 0$, we have
$$ \partial_V\left( \left\{v_1, \dots, v_n\right\} \right) \subseteq \left\{ v_{n+1}, \dots, v_{n + c \cdot \partial_V^n} \right\}.$$
Then for the rearrangement defined by this enumeration and all $1 \leq p \leq \infty$, we have
$$ \| \nabla f^*\|_{L^p} \leq  (c+1) \cdot D^{1/p} \cdot \| \nabla f\|_{L^p}.$$
\end{theorem}
This can be interpreted, at least philosophically, as an extension of  \cite[Theorem 3]{steinerberger} from $L^{\infty}$ to $L^p$. One nice byproduct is that it allows us to show boundedness of a suitable (and very large class) of rearrangements in all lattice graphs $(\mathbb{Z}^d, \ell^1)$: 
we quickly define a notion of rearrangement on $\mathbb{Z}^d$ for which Theorem \ref{thm5.3} implies uniform boundedness of the rearrangement. We say that an enumeration of the lattice points $\mathbb{Z}^d$
respects the $\ell^1-$norm if $\|v_i\|_{\ell^1} < \|v_j\|_{\ell^1}$ implies that $i < j$. This does not specify how to order lattice points with the same $\ell^1$ norm, for these any ordering is admissible. This defines a large number of possible enumerations.
\begin{corollary} \label{cor5.1} There exists a constant $c_d$ such that for every rearrangement respecting the $\ell^1-$norm and all $1 \leq p \leq \infty$
$$  \| \nabla f^*\|_{L^p(\mathbb{Z}^d)} \leq  c_d \cdot \| \nabla f\|_{L^p(\mathbb{Z}^d)}$$
\end{corollary}
This may look a priori like a strong result in so far as it applies uniformly to all $\ell^p-$spaces of functions as well as a very large number of possible enumerations of the vertex set. The price we pay is a lack of control on the constant $c_d$. It could also be interpreted as saying that our present approach is not fine enough to distinguish optimal rearrangements from merely very good rearrangements. As mentioned above, one could hope for a bound along the lines of $c_d \leq d^{1/p}$ from the Wang-Wang enumeration in higher dimensions and this is an interesting problem. We also mention that one could conceivably extend the definition of enumerations respecting the $\ell^1-$norm to enumerations respecting more general norms $\| \cdot \|_X$ on $\mathbb{R}^d$ without changing too much in the proof of Corollary \label{thm:cor} but we will not pursue this here.

The rest of the chapter is structured as follows: Theorem \ref{thm5.1} is, in a suitable sense, the most `generic' application of our framework (explicit constants, no magic simplifications): we will spend most of the paper building a framework that allows us to prove it. The other results tend to follow from the same framework and, using the framework, have shorter proofs. More concretely, the remainder of the chapter is structured as follows.
\begin{enumerate}
\item In Section \ref{sec:comparsiongraph}, we construct what we call the \emph{universal comparison graph} $G_c$. This graph is an infinite tree whose structure is built from the structure of solutions of the (vertex)-isoperimetric problem on $G$. There is way of mapping functions $f$ on $G$ to functions on $G_c$ in a way that decreases the $L^p-$norm of their gradient. 
\item Section \ref{sec:specific} discusses the universal comparison graph attached to $(\mathbb{Z}^2, \ell^1)$.
\item \S \ref{sec:mapping} is concerned with a mapping of edges in $(\mathbb{Z}^2, \ell^1)$ to short paths in its universal comparison graph $G_c$. We show that while this cannot be done bijectively, there is a mapping $\Psi$ of edges from the lattice graph to the comparison graph that has bounded multiplicity (in the sense of the cardinality of the pre-image of an edge in $G_c$ being uniformly bounded by a universal constant).
\item Section  \ref{sec:mainproof}  uses the results from the previous sections to prove Theorem \ref{thm5.1}.
\item Section \ref{sec:wang} is dedicated to proving Theorem \ref{thm5.2}.\footnote{See Appendix \ref{appendix:E} for a short and direct proof of Theorem \ref{thm5.2} in arbitrary dimension.} The proof can be seen as a particularly simple application of the framework
developed above (for example, $\Psi$ is mapping edges to edges instead of mapping edges to paths).
\item Section \ref{sec:last} gives a proof of Theorem \ref{thm5.3} and Corollary \ref{cor5.1}. The proof of Theorem \ref{thm5.3} is essentially identical to the proof of  Theorem \ref{thm5.1} and relies heavily on Section \ref{sec:comparsiongraph} while replacing Section \ref{sec:specific}  and Section \ref{sec:mapping} with more abstract conditions. Corollary \ref{cor5.1} follows quickly from Theorem \ref{thm5.3} (at the cost of giving no control on the constant).
\end{enumerate}

\section{The Universal Comparison graph}\label{sec:comparsiongraph}
The purpose of this section is to introduce the universal comparison graph: it appears to be a useful concept when studying rearrangements on graphs ( `universal' refers it being independent of the function and the rearrangement, it only depends on the graph itself). In this section, we develop a bit of abstract theory for general graphs with Lemma \ref{lem5.4} being the main goal. In the next section we will specify the behavior of the universal comparison graph when $G=(\mathbb{Z}^2, \ell^1)$, in that case the universal comparison graph is a completely explicit infinite tree without leaves (see Fig. \ref{fig:ucg0}).

Before introducing the universal comparison graph, we quickly introduce some of the relevant concepts. Let $G=(V, E)$ be a connected graph with countably infinite vertex set $V$. For any set of vertices $X \subseteq V$ the \emph{vertex boundary} of $X$ is defined as
\begin{align*}
    \partial_V(X) := \{z \in V\setminus X: x \sim z,   \text{for some}\hspace{3pt} x \in X \},
\end{align*}
and $|\partial_V(X)|$ is the \emph{vertex perimeter} of set $X$. We define \emph{isoperimetric number} $\partial_V^n$ as the solution of vertex isoperimetric problem on $G$ among all sets with $n$ vertices 
\begin{equation}
    \partial_V^n := \min_{X \subseteq V, \\ |X|=n}  |\partial_V(X)|,
\end{equation}
for $n \in \N$. We note that on $(\mathbb{Z}^2, \ell^1)$ the vertex-isoperimetric problem is completely solved, we refer to Wang \& Wang \cite{wang}. We also refer to 
\cite{bol1, bol2, gar, harper, lindsey} and references therein for related results. We are now ready to define the main object of this section. 
\begin{definition}\label{def5.1}
Let $G=(V, E)$ be a graph, let $\partial_V^n$ be its isoperimetric number and assume that $\partial_V^{n+1} \geq \partial_V^n$ for $n \geq 1$. Then the \emph{universal comparison graph} $G_c =(\N, E_c)$ is the unique graph on $\N$ satisfying
\begin{equation}\label{5.5}
    \{i \in \N: i > n, i \sim n\} = \{i \in \N: (n-1)+\partial_V^{n-1} < i \leq n+\partial_V^n\},
\end{equation}
for all $n \geq 1$ with the convention $\partial_V^0 = 1$. 
\end{definition}
We quickly explain the construction for the graph $(\mathbb{Z}^2, \ell^1)$ by appealing to results of Wang \& Wang \cite{wang}.
They construct a permutation of the vertices such that the first $n$ vertices corresponding to that permutation minimize the vertex perimeter (the number of adjacent vertices) among all subsets of size $n$ uniformly in $n$.  
Note that this is something special: for most graphs one cannot expect the solutions of the vertex-isoperimetric problem to be
nested, one would expect that they vary a great deal depending on the number of vertices under consideration. Appealing to the
definition above, we see that the universal comparison graph is going to be a graph on $\mathbb{N} = \left\{1,2,3,\dots,\right\}$. Plugging
in $n = 1$, we see that $1$ is adjacent to $\left\{2,3,4,5\right\}$. Plugging in $n=2$, we see that the vertex $2$ is adjacent to $\left\{5 < i \leq 8\right\}$.
The pattern continues, Figs. \ref{fig:ucg0} and \ref{fig:ucg} show the first few levels of the universal comparison tree.

\begin{lemma}\label{lem5.1}
The universal comparison graph $G_c$ is well defined, that is, there is exactly one graph on $\N$ satisfying \eqref{5.5}.
\end{lemma}

\begin{proof}
Let $n \in \N$, then condition \eqref{5.5} fixes all neighbours of $n$ which are greater than $n$. Next, we prove that \eqref{5.5} also fixes neighbours of $n$ which are smaller than $n$, thereby proving the uniqueness. In particular, we prove that for each vertex $n \geq 2$, there exists exactly one vertex $i <n$ such that $i \sim n$. Since $m+\partial_V^m$ is a strictly increasing sequence of integers, for each $n \geq 2$, there exists $1 \leq i<n$ such that $(i-1)+\partial_V^i < n \leq i+\partial_V^i$. Then condition \eqref{5.5} imply that $i \sim n$ and the same argument, monotonicity of $m+\partial_V^m$, establishes uniqueness. This proves that each vertex $n \geq 2$ has exactly one neighbour smaller than $n$.  
\end{proof}

\begin{figure}[h!]
\centering
\begin{tikzpicture}
\filldraw (0,0) circle (0.05cm);
\node at (0.2, -0.2) {1};
\filldraw (1,0) circle (0.05cm);
\node at (0.9, -0.2) {2};
\filldraw (0,1) circle (0.05cm);
\node at (0.2, 1-0.2) {3};
\filldraw (-1,0) circle (0.05cm);
\node at (-1+0.2, -0.2) {4};
\filldraw (0,-1) circle (0.05cm);
\node at (0.2, -1-0.2) {5};
\draw [thick] (-1,0) -- (1,0);
\draw [thick] (0,-1) -- (0,1);
\filldraw (2,1) circle (0.05cm);
\filldraw (2,0) circle (0.05cm);
\filldraw (2,-1) circle (0.05cm);
\node at (2.2, 1-0.2) {8};
\node at (2.2, 0-0.2) {7};
\node at (2.2, -1-0.2) {6};
\draw [thick] (1,0) -- (2,1);
\draw [thick] (1,0) -- (2,0);
\draw [thick] (1,0) -- (2,-1);
\filldraw (0.5,2) circle (0.05cm);
\filldraw (-0.5,2) circle (0.05cm);
\draw [thick] (0,1) -- (0.5, 2);
\draw [thick] (0,1) -- (-0.5, 2);
\node at (0.6, 1.8) {9};
\node at (-0.7, 1.8) {10};
\end{tikzpicture}
\caption{Initial segment of the universal comparison graph for $(\mathbb{Z}^2, \ell^1)$.}
\label{fig:ucg0}
\end{figure}
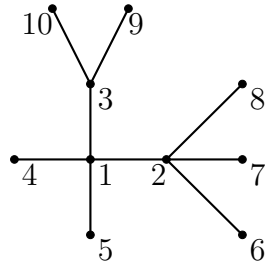

\begin{lemma}\label{lem5.2}
The universal comparison graph $G_c$ is a tree. Furthermore if $\partial_V^1 \geq 2$ then $G_c$ has no leaves, that is, each vertex has degree at least two.
\end{lemma}

\begin{proof}
First we prove that $G_c$ is connected. This follows from the argument above showing that each vertex $n$ is adjacent to exactly one vertex smaller than $n$. This induces a path to the vertex 1. Since each vertex is connected to the vertex 1, the graph is connected.
In Lemma \ref{lem5.1} we proved that each vertex $n \geq 2$ has exactly one neighbour smaller than $n$. This immediately implies that $G_c$ has no cycles: for any cycle, the largest vertex in the cycle say $l$ will have at least two neighbours which are smaller than $l$, which is not possible. Therefore $G_c$ is a tree.
Let $n \geq 2$, since $m+\partial_V^m$ is a strictly increasing sequence, condition \eqref{5.5} implies that $|\{i \in \N: i > n, i \sim n \}| \geq 1$. From Lemma \ref{lem5.1} we also know that each vertex $n \geq 2$ has exactly one neighbour smaller that $n$. This proves that degree of vertex $n$ is at least two, for $n \geq 2$. It is clear that  the degree of vertex $1$ is $\partial_V^1 \geq 2$. Therefore $G_c$ has no leaves.  
\end{proof}
In the next lemma, we compute the vertex boundary of first $n$ vertices of $G_c$.
\begin{lemma}\label{lem5.3}
Let $G$ be a graph and $G_c$ be its universal comparison graph. Then 
\begin{equation}\label{5.6}
    \partial_V(\{1,2,...,n\}) = \{n+1, n+2,..., n+\partial_V^n\}.
\end{equation}
\end{lemma}

\begin{proof}
We prove the result using induction on $n$. Let us assume that \eqref{5.6} holds true for $n \geq 1$. It is easy to see  that
$$\partial_V(\{1,2,...,n, n+1\}) = \partial_V(\{1,2,...,n\})\setminus\{n+1\} \cup \{i \in \N: i > (n+1), i \sim n+1\}.$$
Above identity along with \eqref{5.5} proves \eqref{5.6} for $n+1$. Identity \eqref{5.6} for $n=1$ follows from \eqref{5.5},
$$ \partial_V(\{1\}) = \{i \in \N: i >1, i \sim 1\} = \{2,..,1+\partial_V^1\}.$$
\end{proof}
We are now ready to prove the main result of this section. Let $f: V(G) \rightarrow \R_{\geq 0}$ be a function vanishing at infinity. We define a function $f_c: V_c \rightarrow \R$ on the vertices of its universal comparison graph $G_c = (V_c, E_c)$ as 
$$ f_c(k) := k^{th} \hspace{3pt} \text{largest value attained by} \hspace{3pt} |f|.$$
We will refer to $f_c$ as the \emph{comparison function} of $f$. We will now show that the comparison function has a smaller gradient in the sense of   
$$ \left\| \nabla f_c\right\|_{L^p(G_c)} \leq    \left\| \nabla f\right\|_{L^p(G)}.$$ 
\begin{lemma}[Comparison Lemma]\label{lem5.4}
Let $G= (V,E)$ be a graph and $f: V \rightarrow \R_{\geq 0}$ be a function vanishing at infinity. Then for $p \geq 1$, 
\begin{equation}
    \sum_{(x, y) \in E_c} |f_c(x)-f_c(y)|^p \leq \sum_{(x,y) \in E} |f(x)-f(y)|^p.
\end{equation}
\end{lemma}
The proof is based on a coarea formula already used in \cite{steinerberger} which we quickly explain for the completeness. The coarea formula on graphs is
$$ \| \nabla f\|_{L^p}^p =  \int_{0}^{\infty} \int_{ \partial_{E} \left\{ f \geq s \right\} } \left| \nabla f\right|^{p-1} dx ds,$$
where $ \partial_{E} \left\{ f \geq s \right\}$ is the set of edges that connect the vertex sets $\left\{ v \in V: f(v) \geq s \right\}$ and
$\left\{ v \in V: f(v) < s \right\}$. It is easily derived: the idea being each edge contributes $|f(v) - f(w)|^p$ to the
left-hand side and $|f(v) - f(w)|^{p-1}$ over an interval of length $|f(v) - f(w)|$ to the right-hand side. We will now use a small modification
of the idea: the advantage of this new formulation is that the values in $\left\{ f \geq s \right\}$ no longer show
up in the inner integral which is solely determined by $s$ and the values outside.

\begin{lemma}[Modified Coarea Formula, see \cite{steinerberger}]\label{lem5.5} Let $1 \leq p < \infty$. Then
$$ \| \nabla f\|_{L^p}^p =  p\int_{0}^{\infty} \int_{ \partial_{E} \left\{ f \geq s \right\} } \left| \nabla \min(f, s) \right|^{p-1} dx ds.$$
\end{lemma}
There is a quick proof: consider again a single edge $(v,w) \in E$. The contribution to the left-hand side is
$|f(v) - f(w)|^p$. 
The contribution to the right-hand side is
$$ \int_{\min\left\{f(v), f(w) \right\}}^{\max\left\{f(v), f(w) \right\}} \left(s - \min\left\{f(v), f(w) \right\}\right)^{p-1} ds = \frac{|f(v) - f(w)|^p}{p}.$$
\begin{proof}[Proof of Lemma \ref{lem5.4}]
We can assume, without loss of generality, that $||f||_\infty =1$. The modified coarea formula allows us to rewrite 
$$X =   \sum_{x \sim  y \in V} |f(x)-f(y)|^p$$
as
\begin{equation}\label{5.8}
  X =  p \int_0^{1} \sum_{(x,y) \in E(\{f>s\}, \{f>s\}^c)}|\min(s, f(x))-\min(s,f(y))|^{p-1} ds,
\end{equation}
where $E(X,Y)$ denotes the set of edges between $X, Y \subseteq V(G)$. We argue that the desired integral is monotone for each fixed $0 < s< 1$. Let us thus fix a value of $0 < s < 1$ and consider the quantity
$$ Y = \sum_{(x,y) \in E(\{f>s\}, \{f>s\}^c)}|\min(s, f(x))-\min(s,f(y))|^{p-1}.$$
There is a naturally associated integer $i$ defined via
$$ f_c(1) \geq f_c(2) \geq \dots \geq  f_c(i) \geq s > f_c(i+1) \geq f_c(i+2) \geq \dots$$
The sum $Y$ is then a sum running over all edges connecting $\left\{f \geq s \right\}$ and $\left\{f < s \right\}$. We note that $\left\{f \geq s \right\}$ is finite and has a number of neighbors is at least as big as $\partial_V^i$ (the smallest number of vertices that \textit{any} set of $i$ vertices is adjacent to). Therefore
$$ Y \geq  \sum_{y \in \partial_V(\{f>s\})}|s-f(y)|^{p-1}.$$
We do not have too much information about $ \partial_V(\{f>s\})$ but certainly the sum is smallest if the values are as close as possible to $s$. Thus,
$$ \sum_{y \in \partial_V(\{f>s\})}|s-f(y)|^{p-1} \geq  \sum_{j=1}^{\partial_V^i} |s-f_c(i+j)|^{p-1}.$$
Using \eqref{5.6} and the fact that $G_c$ is a tree we obtain 
$$\sum_{j=1}^{\partial_V^i} |s-f_c(i+j)|^{p-1} = \sum_{(x,y) \in E(\{f_c>s\}, \{f_c>s\}^c)}|\min(s, f_c(x)) -\min(s,f_c(y))|^{p-1}.
    $$
Integrating over all $s$ and applying the modified coarea formula once more
$$ p\int_0^1 \sum_{(x,y) \in E(\{f_c>s\}, \{f_c>s\}^c)}|\min(s, f_c(x)) -\min(s,f_c(y))|^{p-1} ds = \| \nabla f_c\|_{L^p(G_c)}^p.$$
\end{proof}

\section{The Universal Comparsion Graph of $(\mathbb{Z}^2, \ell^1)$} \label{sec:specific}
For the remainder of the paper, we will be interested in the universal comparison graph of the lattice graph on $(\Z^2, \ell^1)$. In the next lemma, we prove that universal comparison graph of the lattice graph is well defined by proving $ \partial_V^{n+1} \geq \partial_V^n$.
\begin{lemma}\label{lem5.6}
Let $G$ be the lattice graph and $\partial_V^n$ be its isoperimetric number. Then, for all $n \geq 1$, we have
$ \partial_V^{n+1} \geq \partial_V^n.$
Furthermore, if $n = 2k(k+1)+1$, then
\begin{equation}\label{5.9}
    \partial_V^n = 4k+4. 
\end{equation}
\end{lemma}
\begin{proof}
In \cite{wang} Wang and Wang studied the vertex isoperimetric problem on the lattice graph. They constructed an enumeration of the vertices such that set of vertices with label $\leq n$ has as few vertices as any set of $n$ elements: this correspond to the construction of a nested sequence of extremizers. We recall the labelling (which was already shown above) in Fig. \ref{fig:sain}. One notably feature is that the nested sequence fills up $\ell^1-$balls in a layer-by-layer fashion.
More precisely, if $x,y \in \mathbb{Z}^2$ and if $\|x\|_{\ell^1} < \|y\|_{\ell^1}$ then the label of $x$ is smaller than the label of $y$. This proves that if $n= 2k(k+1)+1$ for $k \in \N$ (this is the size of $\ell^1$ closed ball in $\Z^2$ of radius $k$) then
$$ \partial_V^n = |\partial_V(\{x \in \Z^2: \left\|x\right\|_{\ell^1} \leq k\})| = 4k+4.$$
It remains to prove that $\partial_V^{n+1} \geq \partial_V^n$. Let $2k(k+1)+1 \leq n < 2(k+1)(k+2)+1 $, for non-negative integer $k$. Assume that $\partial_V^{n+1} < \partial_V^n$. We take the set of all points with label $\leq n+1$. We start by noting that corners of $\ell^1$ ball of radius $k+1$ cannot lie in the set: if a corner lies in the set, then removing the corner does not increase the vertex boundary (see Fig. \ref{fig:sain}) which shows  $\partial_V^{n} \leq \partial_V^{n+1}$ contradicting $ \partial_V^{n+1} < \partial_V^n$. Let us now consider all those vertices whose $y-$coordinate is maximal among all the $n+1$ points in the set. Let $(x,y)$ denote one such vertex with minimal $x$-coordinate. Then removing $(x,y)$ from the set does not increase the vertex boundary since $(x,y)$ is adjacent to $(x,y+1)$ which is not adjacent to any other vertex in the set; removing $(x,y)$ removes this neighbor while only adding $(x,y)$ as a new neighbor. This contradicts $\partial_V^{n+1} < \partial_V^n$. 
\end{proof}
\begin{figure}[h!]
\centering
\begin{minipage}[l]{.43\textwidth}
\begin{tikzpicture}
\node at (-3,0) {};
\node at (0,0) {1};
\node at (0,0.5) {2};
\node at (0.5,0) {3};
\node at (-0.5,0) {4};
\node at (0,-0.5) {5};
\node at (0.5,0.5) {6};
\node at (-0.5,0.5) {7};
\node at (0,1) {8};
\node at (1,0) {9};
\node at (0.5,-0.5) {10};
\node at (-1,0) {11};
\node at (-0.5,-0.5) {12};
\node at (0,-1) {13};
\end{tikzpicture}
\end{minipage}
\begin{minipage}[r]{.25\textwidth}
\begin{tikzpicture}
\filldraw (0,0) circle (0.08cm);
\draw (-0.5,0) circle (0.08cm);
\draw (0,0.5) circle (0.08cm);
\filldraw (0.5,0) circle (0.08cm);
\filldraw (1,0) circle (0.08cm);
\draw (1.5,0.5) circle (0.08cm);
\draw (1.5,0) circle (0.08cm);
\draw (1,0.5) circle (0.08cm);
\draw (0.5,0.5) circle (0.08cm);
\draw (-0.5,0.5) circle (0.08cm);
\filldraw (0,-0.5) circle (0.08cm);
\end{tikzpicture}
\end{minipage} 
\caption{Nested minimizers of the vertex-isoperimetry problem (Wang \& Wang \cite{wang}) (left) and a step in the proof.} 
\label{fig:sain}
\end{figure}
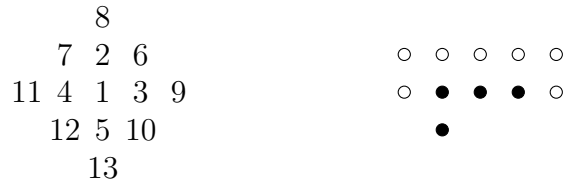
Lemma \ref{lem5.6} implies useful geometric information about the comparison graph of $(\mathbb{Z}^2, \ell^1)$. We use 
 $ S(r;G_C) :=\{i \in V_c : d_{G_c}(i,1) = r\},$
and $ B(r;G_C) := \{i \in V_c : d_{G_c}(i,1) \leq r\}$
to denote the sphere and closed ball of radius $r$ in $G_c$. 
\begin{lemma} \label{lem5.7}
Let $G = (\mathbb{Z}^2, \ell^1)$ and $G_c$ be its universal comparison graph. Then
\begin{equation}
  \forall~r \in \mathbb{N}_{\geq 1} \qquad   |S(r; G_c)| = 4r \hspace{9pt} \text{and} \hspace{9pt} |B(r;G_c)| = 1+2r(r+1).
\end{equation}
\end{lemma}
This follows immediately from the existence of nested minimizers (Wang \& Wang \cite{wang}), see also Fig. \ref{fig:ucg0}. It implies that 
$$\{i \in V_c : d_{G_c}(i,1) = r\} = \left\{v \in \mathbb{Z}^2: \|v\|_{\ell^1} = r \right\}$$
from which one deduces $|S(r; G_c)| = 4r$ and then, by summation,
$$  |B(r;G_c)| = 1 + \sum_{i=1}^{r} 4r = 1 + 2r(r+1).$$

\section{Embedding Edges into Universal Comparison graph}\label{sec:mapping}
For the purpose of this section we will assume again that $G= (\Z^2, \ell^1)$ is the lattice graph on $\Z^2$ and that $G_c$ is the associated universal comparison graph. In this section, we assume that vertices of $G$ are labelled using the spiral labelling on $\Z^2$, as defined in Figure \ref{fig:spiral and wang-wang}. The rest of this section is devoted to the study of the map
$$ \Psi : E_{(\mathbb{Z}^2, \ell^1)} \rightarrow \{\text{paths in}\hspace{3pt} G_c\},$$
defined as follows: Let $i<j$ and $(i,j) \in E_{(\mathbb{Z}^2, \ell^1)}$ be an edge in the lattice graph. Then $\Psi(i, j)$ is defined as the shortest path in the tree $G_c$ between the vertex $i$ and smallest vertex $k \in V_c$ such that
\begin{enumerate}
\item $k \geq j$ 
\item and $i$ lies in the path between $1$ and $k$ in $G_c$. 
\end{enumerate}
The second condition could also be phrased as follows: since $G_c$ is a tree with no leaves (Lemma \ref{lem5.2}), we may think of the vertex 1 as a root. Then the vertex $i$ has infinitely many descendants (among which we pick the smallest one, $k$, that is at least as big as $j$). We will show that $\Psi$ maps edges to paths with uniformly bounded length.
\begin{figure}[h!]
\centering
\begin{minipage}[l]{.3\textwidth}
\begin{tikzpicture}
\node at (-3,0) {};
\node at (0,0) {1};
\node at (0,0.5) {2};
\node at (0.5,0) {3};
\node at (-0.5,0) {4};
\node at (0,-0.5) {5};
\node at (0.5,0.5) {6};
\node at (-0.5,0.5) {7};
\node at (0,1) {8};
\node at (1,0) {9};
\node at (0.5,-0.5) {10};
\node at (-1,0) {11};
\node at (-0.5,-0.5) {12};
\node at (0,-1) {13};
\end{tikzpicture}
\end{minipage}
\begin{minipage}[l]{.53\textwidth}
\centering
\begin{tikzpicture}
\node at (-3,0) {};
\filldraw (0,0) circle (0.06cm);
\filldraw (0,0.5) circle (0.06cm);
\filldraw (-0.5,0) circle (0.06cm);
\filldraw (0,-0.5) circle (0.06cm);
\filldraw (0.5,0.5) circle (0.06cm);
\filldraw (-0.5,0.5) circle (0.06cm);
\filldraw (0,1) circle (0.06cm);
\filldraw (1,0) circle (0.06cm);
\filldraw (0.5,-0.5) circle (0.06cm);
\filldraw (-0.5,-0.5) circle (0.06cm);
\filldraw (-1,0) circle (0.06cm);
\filldraw (0,-1) circle (0.06cm);
\filldraw (0.5,0) circle (0.06cm);
\draw [thick] (-0.5, 0) -- (0.5, 0);
\draw [thick] (0,-1) -- (0,0.5);
\draw [thick] (-0.5,0.5) -- (0.5, 0.5);
\draw [thick] (0,0.5) -- (0,1);
\draw [thick] (0.5, 0) -- (1,0);
\draw [thick] (0.5, 0) -- (0.5,-0.5);
\draw [thick] (-0.5, 0) -- (-1,0);
\draw [thick] (-0.5, 0) -- (-0.5,-0.5);
\end{tikzpicture}
\end{minipage}
\caption{Nested minimizers of the vertex-isoperimetry problem (left) and the universal comparison graph over the same vertex set (right).} 
\label{fig:ucg}
\end{figure}
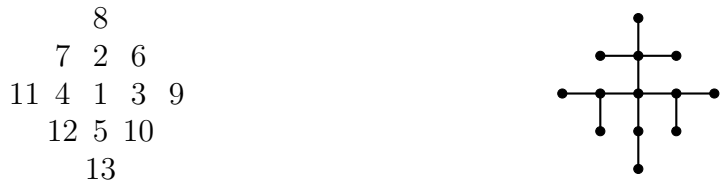
\begin{lemma} \label{lem5.8} Suppose $(m,n)$, with $m < n$, is an edge in $(\mathbb{Z}^2, \ell^1)$ (where $m,n$ are integers and the corresponding vertices are with respect to the spiral labeling). Then
$$ n \leq m + 7\sqrt{m}.$$
 \end{lemma}
\begin{proof}
The vertex $m$ is on the boundary of an $\ell \times \ell$ square where $\ell = \left\lfloor \sqrt{m} \right\rfloor$. The smallest integer at the boundary of an $\ell \times \ell$ square is $(\ell - 1)^2 + 1$. Thus $m \geq (\ell - 1)^2 + 1$. Any neighbor has to be contained in an $(\ell + 2) \times (\ell + 2)$ square. The largest number arising in such a square is $(\ell + 2)^2$. Thus $n \leq (\ell + 2)^2$. We have
$$ n \leq (\ell + 2)^2 \leq (\ell - 1)^2 + 1 + 7 \sqrt{(\ell - 1)^2 + 1 }$$
for all $\ell \geq 10$. The cases $\ell \leq 10$ can be verified by hand.\end{proof}
We note that the argument is tight because $(1,8)$ is indeed an edge in $\mathbb{Z}^2$ with respect to the spiral labeling. It seems that one could asymptotically improve the constant for larger $m$ but it is not entirely clear how this could be leveraged into a better result.
\begin{lemma}\label{lem5.9} $\Psi$ maps edges in $E_{(\mathbb{Z}^2, \ell^1)}$  to paths of length at most $4$. 
\end{lemma}
\begin{proof} Let us pick an edge $(m,n) \in (\mathbb{Z}^2, \ell^1)$ where $m < n$ and the numbering refers to the spiral labeling of $(\mathbb{Z}^2, \ell^1)$. We use
$r = d_{G_c}(1,m)$ to denote the distance between $m$ and $1$ in the universal comparison graph. Lemma \ref{lem5.7} implies that
$$ 1 + 2(r-1)(r+1) \leq m \leq 1 + 2r(r+1).$$
The upper bound implies  
$$r \geq \frac{\sqrt{2m-1}-1}{2} = X.$$
Appealing once more to Lemma \ref{lem5.7} shows that we expect $ \geq 4X+4, 4X+8, 4X+12$ vertices at distance $r+1$, $r+2$ and $r+3$, respectively. This means there are at least
$$12X + 24  = 6 \sqrt{2m-1} + 18 \geq  8\sqrt{m} + 15 >7\sqrt{m}$$
 vertices at distance $r+1 \leq r+3$. Lemma \ref{lem5.8} implies that
picking a descendant $k$ of $m$ at distance 4 ensures $k \geq n$. Thus $\Psi$ maps edges to a path of length at most 4. 
\end{proof}
\begin{lemma}\label{lem5.10}
Let $e \in E_c$ be an edge in $G_c$. Then there are at most 16 edges $(i, j) \in E$ in the lattice graph such that $e$ lies in the path $\Psi(i, j)$.
\end{lemma}
\begin{proof} 
Let $e \in E_c$ be an edge in the universal comparison graph (which is a tree). We recall that $\Psi$ maps edges $(i,j) \in E_{(\mathbb{Z}^2, \ell^1)}$ to a path from $i$ to $k$ where $k$ is a descendant of $i$ at distance at most 4. This means that if we take the shortest path from the edge $e$ to the root, we are bound to find the vertex $i$ at distance at most 4 from the edge $e$. Thus there are at most 4 different vertices $i$ that could be the source of a path containing $e$. Since each vertex $i \in \mathbb{Z}^2$ has 4 neighbors, there are at most 16 edges $(i,j) \in E$ that could be mapped to a path containing $e$.
\end{proof}

\section{Proof of  Theorem \ref{thm5.1}}\label{sec:mainproof}
\begin{proof}[Proof of Theorem \ref{thm5.1}]
Let $f^*$ be the rearrangement of $f$ with respect to the spiral labelling on $\Z^2$ and let $f_c$ be the comparison function of $f$ on the universal comparison graph of the lattice graph. Consider an edge $(i, j) \in E$ and $\Psi(i, j)= \{x_0,x_1,..,x_n\}$ be a path in $G_c$ with $x_0<x_1<...<x_n$ and $n \leq 4$ (see Lemma \ref{lem5.9}). Note that $x_0 = i$ and
$x_{n} \geq j$ from which we deduce $f^*(x_n) \leq f^*(j)$.
Jensen's inequality applied to $x \mapsto x^p$ for $p \geq 1$ implies that for $a_1, \dots, a_n > 0$
$$ \left( \frac{a_1 + \dots + a_n}{n} \right)^p \leq \frac{a_1^p}{n} + \dots +  \frac{a_n^p}{n}$$
and thus
$$ (a_1 + \dots + a_n)^p \leq n^{p-1} \left(a_1^p + \dots + a_n^p\right).$$
Therefore, using this together with $n\leq 4$ and the triangle inequality,
\begin{align*}
    |f^*(i)-f^*(j)|^p &\leq |f_c(x_0)-f_c(x_n)|^p\\
    &= \left| \sum_{k=0}^{n-1} f_c(x_k)-f_c(x_{k+1}) \right|^p \leq 4^{p-1} \sum_{k=0}^{n-1} |f_c(x_k)-f_c(x_{k+1})|^p.
\end{align*}
We now sum both sides of the inequality over all edges $(i,j) \in E_{(\mathbb{Z}^2, \ell^1)}$. Appealing to Lemma \ref{lem5.10}, we deduce that we end up summing over each edge in the universal comparison graph at most 16 times and thus, together with  Lemma \ref{lem5.4},
\begin{equation}\label{4.1}
    \left\|\nabla f^*\right\|_{L^p(\mathbb{Z}^2)}^p \leq 4^{p+1} \left\|\nabla f_c\right\|_{L^p(G_c)}^p \leq 4^{p+1}  \left\|\nabla f\right\|_{L^p(\mathbb{Z}^2)}^p.
\end{equation}
\end{proof}

\section{Proof of Theorem \ref{thm5.2}} \label{sec:wang}
\begin{proof} 
The proof is similar in style to the proof of Theorem \ref{thm5.1}, however, extremal properties of the Wang-Wang enumeration simplifies various steps.
 \begin{figure}[h!]
 \centering
\begin{minipage}[l]{.3\textwidth}
\begin{tikzpicture}
\node at (-3,0) {};
\node at (0,0) {1};
\node at (0,0.5) {2};
\node at (0.5,0) {3};
\node at (-0.5,0) {4};
\node at (0,-0.5) {5};
\node at (0.5,0.5) {6};
\node at (-0.5,0.5) {7};
\node at (0,1) {8};
\node at (1,0) {9};
\node at (0.5,-0.5) {10};
\node at (-1,0) {11};
\node at (-0.5,-0.5) {12};
\node at (0,-1) {13};
\end{tikzpicture}
\end{minipage}
\begin{minipage}[l]{.53\textwidth}
\begin{tikzpicture}
\node at (-3,0) {};
\filldraw (0,0) circle (0.06cm);
\filldraw (0,0.5) circle (0.06cm);
\filldraw (-0.5,0) circle (0.06cm);
\filldraw (0,-0.5) circle (0.06cm);
\filldraw (0.5,0.5) circle (0.06cm);
\filldraw (-0.5,0.5) circle (0.06cm);
\filldraw (0,1) circle (0.06cm);
\filldraw (1,0) circle (0.06cm);
\filldraw (0.5,-0.5) circle (0.06cm);
\filldraw (-0.5,-0.5) circle (0.06cm);
\filldraw (-1,0) circle (0.06cm);
\filldraw (0,-1) circle (0.06cm);
\filldraw (0.5,0) circle (0.06cm);
\draw [thick] (-0.5, 0) -- (0.5, 0);
\draw [thick] (0,-1) -- (0,0.5);
\draw [thick] (-0.5,0.5) -- (0.5, 0.5);
\draw [thick] (0,0.5) -- (0,1);
\draw [thick] (0.5, 0) -- (1,0);
\draw [thick] (0.5, 0) -- (0.5,-0.5);
\draw [thick] (-0.5, 0) -- (-1,0);
\draw [thick] (-0.5, 0) -- (-0.5,-0.5);
\end{tikzpicture}
\end{minipage}
\caption{Nested minimizers of the vertex-isoperimetry problem (left) and the universal comparison graph over the same vertex set (right).} 
\label{fig:ucg2}
\end{figure}
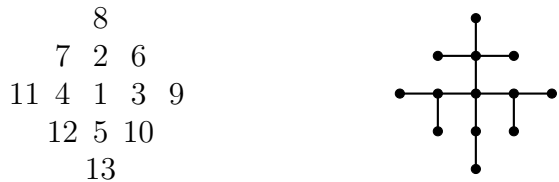

We map everything into the universal comparison tree and use Lemma \ref{lem5.4}
$$ \| \nabla f \|_{L^p(\mathbb{Z}^2)} \geq \| \nabla f_c\|_{L^p(G_c)}.$$
Note that the Wang-Wang construction uniformly minimizes the
vertex-boundary, the neighbors of $\left\{v_1, \dots, v_n\right\}$ in $\mathbb{Z}^2$
are exactly the same neighbors as the neighbors of $\left\{1, \dots, n\right\}$ in $G_c$:
they are given by 
\begin{equation}\label{5.12}
    \partial_V(\left\{v_1, \dots, v_n\right\}) = \left\{v_{n+1}, v_{n+2}, \dots, v_{n+\partial_V^n} \right\}
\end{equation}
in both cases. For $G_c$ this follows by construction (see Lemma \ref{lem5.3}), for $\mathbb{Z}^2$ with the
Wang-Wang enumeration this follows from the fact that the solutions are \textit{nested}: this means (see, for example, Bezrukov \& Serra \cite{bez}) that
\begin{enumerate}
\item there exists a nested sequence of sets of vertices
$$ A_1 \subset A_2 \subset A_3 \subset \dots$$
such that $\# A_i = i$ and $A_i$ is adjacent to as few vertices as possible for a set of vertices with $i$ elements and, moreover,
\item such that for all $i \in \mathbb{N}$ there exists $j \in \mathbb{N}$ such that $A_i \cup \partial_V(A_i) = A_j$.
\end{enumerate}

The main difference (see also Fig. \ref{fig:ucg2}) is that $\mathbb{Z}^2$ contains
some edges that are not contained in $G_c$. Using the Wang-Wang enumeration, it suffices to consider all edges $(i,j) \in \mathbb{Z}^2$ with $\|i\|_{\ell^1} + 1 = \|j\|_{\ell^1} =: r \geq 1$. There are two cases:
\begin{enumerate}
    \item $j$ is a corner point of $\ell^1$ ball of radius $r$. There is exactly one neighbour $i$ of $j$ in the $\ell^1$ ball of radius $r-1$ with $\|i\|_{\ell^1} = r-1$. Then using \eqref{5.12} we get,
    \begin{equation}\label{5.13}
        i+ \partial_V^{i-1} \leq j \leq i + \partial_V^i.
    \end{equation}
    Inequality \eqref{5.13} with the Definition \ref{def5.1} of comparison graph proves that $j$ is a neighbour of $i$ in $G_c$. 
    \item $j$ is not a corner point of $\ell^1$ ball of radius $\|j\|_{\ell^1}$. It is easy to see that $j$ will have exactly two neighbours in the ball of radius $r-1$ say $i_1, i_2$ with $\|i_1\|_{\ell^1} = \|i_2\|_{\ell^1} = r-1$. W.l.o.g. assume that $i_1 < i_2$. This shows
    \begin{equation}
        i_1 + \partial_V^{i_1-1} \leq j \leq i_1 + \partial_V^{i_1}
    \end{equation}
    Similarly, this proves that $j$ is connected to $i_1$ and not $i_2$ in $G_c$.   
\end{enumerate}
This proves that the comparison graph $G_c$ is obtained from $\Z^2$ by removing some of the edges between two consecutive $\ell^1$ spheres. 
The mapping 
$$ \Psi : E_{(\mathbb{Z}^2, \ell^1)} \rightarrow \{\text{paths in}\hspace{3pt} G_c\}$$
is now very simple: if $(i,j) \in E_{(\mathbb{Z}^2, \ell^1)}$, then we map $(i,j)$ to $(i,j)$ if that edge happens
to be in $G_c$. If not, then there exists exactly one $(k,j) \in G_c$ with $k < i$ and we map the edge to that.
In particular, comparing to the previous proof, $\Psi$ is mapping edges to edges and no application of Jensen's inequality is needed. 
Each edge in $G_c$ has a pre-image of cardinality at most 2 under $\Psi$ and thus, summing over all edges,
$$ \| \nabla f_c \|^p_{L^p(\mathbb{Z}^2, \ell^1)} \leq 2 \cdot \| \nabla f_c \|^p_{L^p(G_c)} \leq 2\cdot \| \nabla f \|^p_{L^p(\mathbb{Z}^2)}. $$ \end{proof}
\section{Proof of Theorem \ref{thm5.3} and Corollary \ref{cor5.1}}
\label{sec:last}
\begin{proof}[Proof of Theorem \ref{thm5.3}] The proof follows the same steps that we followed in the proof of the two-dimensional results. Much of the combinatorial complexity comes from explicitly bounding various constants which here are encapsulated in a more abstract condition. We start by returning to the notion of the universal comparison graph developed in Section \ref{sec:comparsiongraph}.
We can apply Lemma \ref{lem5.4} to deduce that
$$  \| \nabla f\|_{L^p(G)} \geq  \| \nabla f_c\|_{L^p(G_c)}$$
where $G_c$ is the universal comparison graph. Note that, as before, $G_c$ is an infinite tree with no leaves. It remains to compare edges in $G$ with short paths in $G_c$ just as we did before. As before, we now assume that we are given an enumeration of the vertices $v_1, v_2, \dots$ satisfying all assumptions of the Theorem.
We consider again the map
$$ \Psi : E_{G} \rightarrow \{\text{paths in}\hspace{3pt} G_c\},$$
defined in the same way as above: if $i<j$ and $(v_i,v_j) \in E(G)$ be an edge in the graph, then $\Psi(v_i,v_j)$ is defined as the shortest path in the tree $G_c$ between the vertex $v_i \in V_c$ and smallest vertex $v_k \in V_c$ such that
\begin{enumerate}
\item $k \geq j$ 
\item and $v_i$ lies in the shortest path between $v_1$ and $v_k$ in $G_c$. 
\end{enumerate}
Suppose now that $(v_i, v_j) \in E(G)$. The assumption in the Theorem says that the neighbors of $\left\{v_1, \dots, v_i\right\}$ are not too many, the set
is uniformly close to a vertex-isoperimetric set (up to a constant $c$). More precisely, since $v_j$ is adjacent to the set and $j > i$, we deduce that 
$$ j \leq i + c \cdot \partial_V^i.$$ 
Mentally fixing $v_1$ as the root of the infinite tree, we can now pick the vertex $v_i$. The distance between these two vertices is
$r = d_{G_c}(v_1, v_i)$. By construction of the universal comparison graph, we have that 
$$ \# \left\{v \in G_c: d(v, v_1) = r+1\right\} \geq \partial_V^{i}.$$
Using the monotonicity of $\partial_V^n$, we deduce that the same inequality holds for larger `spherical shells' $ \# \left\{v \in G_c: d(v, v_1) = r+\ell\right\}$ and any
$\ell \geq 1$. By summation,
 $$\# \left\{v \in G_c:  r+1 \leq d(v, v_1) \leq r+\ell\right\} \geq \ell \cdot \partial_V^i.$$
 This shows that there is a descendant $v_k$ of $v_i$ satisfying $d_{G_c}(v_i, v_k) \leq c+1$ as well as $k \geq j$. Thus $\Psi$ is mapping edges to paths with length bounded by $c+1$. This means, appealing to both the tree structure and the construction of $\Psi$ that each edge $e \in G_c$ can only appear as the image of $\Psi$
 of an edge that is adjacent to at most $c+1$ different vertices in $G$. Since each vertex in $G$ has degree bounded from above by $D$, there are at most $(c+1) \Delta$ edges in question. Jensen's inequality implies, as above,
\begin{align*}
    |f^*(i)-f^*(j)|^p \leq (c+1)^{p-1} \cdot \sum_{k=0}^{n-1} |f_c(x_k)-f_c(x_{k+1})|^p,
\end{align*}
since $n \leq c+1$. Summing again both sides over all edges, we deduce
$$ \| \nabla f^*\|^p_{L^p} \leq (c+1)^{p} \Delta \cdot  \| \nabla f\|^p_{L^p}.$$
\end{proof}
\begin{proof}[Proof of Corollary \ref{cor5.1}]
The argument appeals to Theorem \ref{thm5.3}. We first note that $\partial_V^{n+1} \geq \partial_V^n$ remains true and can be shown just as in $d=2$ (Lemma \ref{lem5.6}). We first recall a well-known isoperimetric inequality (see \cite{bol1, bol2, wang}) telling us that for some universal constant $\alpha_d>0$ depending only on the dimension
$ \partial_{(\mathbb{Z}^d,\ell^1)}^n \geq \alpha_d \cdot n^{\frac{d-1}{d}}.$
The second observation is that if $A \subset \mathbb{Z}^d$ is a set such that 
$$  \exists k \in \mathbb{N} ~\forall~a \in A \qquad \|a\|_{\ell^1} \in \left\{k, k+1 \right\} $$
then $ \partial_V(A) \leq C_d (\#A)^{\frac{d-1}{d}}$. This follows from the fact that one can control the volume of a $1-$neighborhood of a convex domain in $\mathbb{R}^n$ from above by a multiple of its surface area (provided the convex domain is not too small in volume) and a simple comparison argument. Since $\alpha_d$ and $C_d$ are both universal constants, we can apply Theorem \ref{thm5.3} to deduce the desired result.
\end{proof}

%% file: appendix.tex
\appendix
\chapter{Discrete Hardy inequality for antisymmetric functions}\label{appendix:A}
It is well known that Hardy inequality does not hold in $\R^2$, and it is natural to expect the same in the discrete setting. We start with proving the impossibility of discrete Hardy inequality in dimension two:

\begin{lemma}
There does not exist a real constant $c$ for which the following Hardy inequality holds
\begin{equation}\label{A1}
    \sum_{n \in \Z^2} \frac{|u(n)|^2}{|n|^2}\leq c\sum_{n\in \Z^2}\sum_{j=1}^2|u(n)-u(n-e_j)|^2, 
\end{equation}
for all $u \in C_c(\Z^2)$ with $u(0) = 0$.
\end{lemma}

\begin{proof}
Let $\|x\| := \text{max} |x_i|$ be the $\ell^\infty$ norm on $\R^d$. Let $\varphi_N: \N_0 \rightarrow \R$ be a function defined by
\begin{align*}
    \varphi_N(r) := 
    \begin{cases}
        1 \hspace{99pt} \text{if} \hspace{19pt} 1 \leq r \leq N.\\
        2-r/N \hspace{63pt}\text{if} \hspace{9pt} N+1 \leq r \leq 2N.\\
        0 \hspace{100pt} \text{if} \hspace{19pt} r \geq 2N \hspace{5pt} \text{or} \hspace{5pt} r=0.
    \end{cases} 
\end{align*}
Let $u_N: \Z^2 \rightarrow \R$ be a radial function with respect to $\ell^\infty$ norm given by $u_N(n) := \varphi_N(\|n\|)$. Some straightforward calculation gives 
\begin{align*}
    \sum_{n \in \Z^2} \frac{|u_N(n)|^2}{|n|^2} \geq \frac{1}{2}\sum_{n \in \Z^2} \frac{|u_N(n)|^2}{\|n\|^2} \geq 4 \sum_{r=1}^N \frac{1}{r}.
\end{align*}
On the other hand we have
\begin{align*}
    \sum_{n\in \Z^2}\sum_{j=1}^2|u_N(n)-u_N(n-e_j)|^2 \leq 4+ \frac{4}{N^2}\sum_{r=1}^{2N} (2r+1) = 20 + \frac{16}{N}.  
\end{align*}
Using the above two estimates we get
$$ \sum_{n} |u_N(n)|^2 |n|^{-2}/\sum_n\sum_{j=1}^2|u_N(n)-u_N(n-e_j)|^2 \geq 4 \frac{\sum_{r=1}^N r^{-1}}{(20+16/N)}  \rightarrow \infty,$$
as $N \rightarrow \infty$.

\end{proof}

In the continuum, the standard ways to get estimate an of the kind $\eqref{A1}$ is either by restricting the class of functions or by adding a magnetic potential in the operator \cite{laptev-weidl, antisymmetric}. One important and interesting class of functions for which Hardy inequality holds in $\R^2$ are antisymmetric functions:
Let $u: \R^d \rightarrow \R$ be a function, then it is called \emph{antisymmetric} if 
\begin{equation}\label{A2}
    u(\dots,x_j,\dots,x_i, \dots) = -u(\dots,x_i,\dots,x_j, \dots),
\end{equation}
for all $1 \leq i, j \leq d$. 

In \cite{hoffmann2021hardy}, authors proved a sharp Hardy inequality for functions satisfying \eqref{A2}:
\begin{theorem}[Hoffmann-Ostenhof, Laptev \cite{hoffmann2021hardy}]\label{thmA1}
Let $d \geq 2$ and $u \in C_0^\infty(\R^d)$ satisfying \eqref{A2}. Then 
\begin{equation}\label{A3}
    \int_{\R^d} |\nabla u(x)|^2 dx \geq \frac{(d^2-2)^2}{4} \int_{\R^d} \frac{|u(x)|^2}{|x|^2} dx.
\end{equation}
The constant in \eqref{A3} is sharp.
\end{theorem}
We note that the result is true for $d=2$ and the constant in \eqref{A3} is substantially better than the one in classical Hardy inequality \eqref{1.1} for large dimensions.

In the rest of this chapter we prove a discrete analogue of Theorem \ref{thmA1}. In particular, showing that inequality \eqref{A1} holds under the antisymmetric condition \eqref{A2}. We consider $d=2$ and $d \geq 3$ separately since we prove them using different techniques. 

\section{Discrete polar coordinates}
In this section, we give a natural notion for `polar-coordinates' for points in $\Z^2$. This notion uses $\ell^\infty$ norm ($\|x\| := \text{max} |x_i|$), as opposed to standard $\ell^2$ norm on $\R^d$.

Let $r \in \N$. Let $\mathbb{S}(r):=\{n=(n_1,n_2) \in \mathbb{Z}^2 : ||n||=r\}$ be the sphere of radius $r$. We define the anti-clockwise rotation map $U : \mathbb{S}(r) \rightarrow \mathbb{S}(r)$ as follows:
\begin{align*}
    U(n_1,n_2) := 
    \begin{cases}
        (n_1-1, n_2)   &\hspace{9pt} \text{if} \hspace{5pt} -r < n_1 \leq r \hspace{5pt} \text{and} \hspace{5pt} n_2 = r.\\
        (n_1, n_2-1) &\hspace{9pt} \text{if} \hspace{5pt} -r < n_2 \leq r \hspace{5pt} \text{and} \hspace{5pt} n_1 = -r.\\
        (n_1+1, n_2) &\hspace{9pt} \text{if} \hspace{5pt} -r \leq n_1 < r \hspace{5pt} \text{and} \hspace{5pt} n_2 = -r.\\
        (n_1,n_2+1) &\hspace{9pt} \text{if} \hspace{5pt} -r \leq n_2 < r \hspace{5pt} \text{and} \hspace{5pt} n_1 = r.\\
    \end{cases}
\end{align*}
Let $n=(n_1,n_2) \in \mathbb{Z}^2$ and let $m$ be the smallest non-negative integer such that $U^m(r,r) = (n_1, n_2)$ where $r := ||n||$. Then we define $(r,m)$ as the \emph{polar coordinates} of $(n_1,n_2)$.

It is natural at this point to study the spectrum of discrete Laplacian\footnote{discrete Laplacian on a graph is given by $\Delta u(x) := \sum_{y \sim x}(u(x)-u(y))$, where $y \sim x$ means there is an edge between $x$ and $y$} on the sphere $\mathbb{S}(r)$ (`discrete spherical Harmonics'). We take the standard graph structure on $\mathbb{S}(r)$: $x = (r, m_1)$ and $y=(r, m_2)$ are connected by an edge (denoted by $x \sim y$ or $m_1 \sim m_2$) if and only if $|m_1-m_2| =1 $ or $8r-1$.

\begin{lemma}\label{lemA2}
Let $\Delta$ be the discrete Laplacian acting on graph $\mathbb{S}(r)$. Then its eigenvalues and eigenfunctions are given by $\lambda_l := 4\sin^2(\frac{\pi l}{8r})$ and $\{\alpha e^{e^{\frac{2\pi i ml}{8r}}} + \beta e^{-\frac{2\pi i ml}{8r}}: \alpha, \beta \in \mathbb{\C}\}$ respectively for $0 \leq l \leq 4r$.
\end{lemma}
It can be proved using discrete Fourier transform. We note that the antisymmetric eigenfunctions corresponding to $\lambda_l$ are spanned by $\sin(\frac{2\pi ml}{8r})$ for $1 \leq l \leq 4r-1$. Thus spanning the space of antisymmetric functions on $\mathbb{S}(r)$.

\section{Hardy inequality for antisymmetric functions: $d=2$}

We prove an inequality equivalent to \eqref{A1}:
\begin{theorem}
Let $u \in C_c(\Z^2)$ satisfying the antisymmetric condition \eqref{A2}. Then
\begin{equation}\label{A4}
    4 \sin^2(\pi/8)\sum_{n \in \mathbb{Z}^2} \frac{|u(n)|^2}{||n||^2} \leq \sum_{n \in \mathbb{Z}^2} \sum_{j=1}^2|u(n)-u(n-e_j)|^2,
\end{equation}
where $e_j$ is the $j^{th}$ canonical basis of $\mathbb{R}^2$.
\end{theorem}

\begin{proof}
We begin with the observation that antisymmetric eigenfunctions of $\Delta$ on $\mathbb{S}(r)$ $\Big\{\varphi_l (m) := \sqrt{\frac{1}{4r}} \sin(\frac{2\pi ml}{8r})\Big\}_{l=1}^{4r-1}$ forms an orthonormal basis of space of antisymmetric functions on $\mathbb{S}(r)$. Thus for $r \geq 1$ we have 
\begin{equation}\label{A5}
    u(r, m) = \sum_{l=1}^{4r-1} \widehat{u}(r, l)\varphi_l(m),
\end{equation}
for all $0 \leq m \leq 8r-1$. Using the decomposition \eqref{A5} and the Parseval's identity we get
\begin{equation}\label{A6}
    \sum_{n \in \Z^2} \frac{|u(n)|^2}{\|n\|^2} = \sum_{r=1}^\infty \frac{1}{r^2}\sum_{m=0}^{8r-1} |u(r, m)|^2 =  \sum_{r=1}^\infty \frac{1}{r^2} \sum_{l=1}^{4r-1}|\widehat{u}(r, l)|^2.
\end{equation}
Ignoring the differences of $u$ between $\mathbb{S}(r)$ and $\mathbb{S}(r+1)$(`radial derivatives') we get 
\begin{equation}\label{A7}
    \begin{split}
        \sum_{n \in \Z^2} \sum_{j=1}^2 |u(n)-u(n-e_j)|^2 &\geq \frac{1}{2}\sum_{r=1}^\infty \sum_{m_1 \sim m_2} |u(r, m_1)-u(r, m_2)|^2\\
        &= \sum_{r=1}^\infty \sum_{m=0}^{8r-1} \Delta u(r,m) u(r, m)\\
        &= \sum_{r=1}^\infty \sum_{l=1}^{4r-1} \lambda_l |\widehat{u}(r, l)|^2.
    \end{split}
\end{equation}
In the last line we used the decomposition \eqref{A5} and the fact $(\lambda_l, \varphi_l)$ satisfy eigenvalue equation for $\Delta$ on $\mathbb{S}(r)$ (see Lemma \ref{lemA2}).

Using $\sin x \geq \frac{8 \sin (\pi/8)}{\pi} x$ for $0 \leq x \leq \pi/8$ we obtain for $1 \leq l \leq 4r-1$
$$ \lambda_l = 4\sin^2(\frac{\pi l}{8r}) \geq 4 \sin^2(\frac{\pi}{8r}) \geq 4 \sin^2(\pi/8)\frac{1}{r^2}.$$

Using this estimate in \eqref{A7} and \eqref{A6} completes the proof.
\end{proof}

\begin{remark}
It would be interesting to extend the proof in higher dimensions. This would involve inventing polar coordinates and studying the spectrum of Laplacian on a $d$ dimensional sphere $\mathbb{S}(r)$. This is an interesting problem in itself. 
\end{remark}

\begin{remark}
We believe that the major part of the optimal constant in \eqref{A4} comes from the radial part of the gradient which was ignored in \eqref{A7}. Moreover, we conjecture that the constant coming from spherical part of the gradient in \eqref{A7} scales linearly in $d$ as $d \rightarrow \infty$. Hence it would be a worthwhile effort to understand the radial component, as it might give the correct asymptotic behaviour of the optimal constant.
\end{remark}

\section{Hardy inequality for antisymmetric functions: $d \geq 3$}
The main result of this section is
\begin{theorem} \label{thmA3}
Let $d \geq 3$ and $u \in C_c(\Z^d)$ satisfying the antisymmetric condition \eqref{A2}. Then we have
\begin{equation}\label{A8}
    \frac{4d(d-2)^2 C_p(d)}{16dC_p(d)+(3d-2)(d-2)}\sum_{n\in \Z^d} |n|^{-2}|u(n)|^2 \leq \sum_{n \in \Z^d} |Du(n)|^2,
\end{equation}
where 
$$   C_p(d) := N(N-1)(2N-1)/3 + (3-(-1)^d)N^2/2,$$
and  $N := \lfloor d/2 \rfloor$. Here $Du$ is as defined in \eqref{3.1}.
\end{theorem}

\begin{remark}
We note that $C_p(d) \sim d^3$ and hence the constant in \eqref{A8} grows as $d^2$ as $d \rightarrow \infty$. This shows the constant in the discrete Hardy inequality improves by a significant amount under the antisymmetric condition (the constant in the discrete Hardy inequality grows linearly, see Corollary \ref{cor3.1}).
\end{remark}

\begin{remark}
Consider a non-zero antisymmetric function $u$ such that $u(n) = 0$ for $|n|^2 \neq C_p(d)$ (see proof of Lemma \ref{A3} for existence of such a function). Then we have $\sum_{n}|Du(n)|^2 / \sum_n |n|^{-2}|u(n)|^2 = 2d C_p(d) \sim d^4$ as $d \rightarrow \infty$. This along with Theorem \ref{thmA3} proves that there exists positive constants $c_1, c_2$ such that the sharp constant $C_{ass}(d)$ in \eqref{A8} satisfy $c_1 d^2 \leq C_{ass}(d) \leq c_2 d^4$ as $d \rightarrow \infty$. 
\end{remark}

The proof of the result is based on the ideas similar to ones discussed in Chapter \ref{ch:higher-hardy}. Let $u_j := \frac{n_j}{|n|^2} u(n)$. Then from Lemma \ref{lem3.1} (for $k=0$) we know that $\big(Q_d := (-\pi, \pi)^d$ and $\omega(x):= \sum_{j=1}^d \sin^2(x_j/2)\big)$
\begin{equation}\label{A9}
    \begin{split}
        \sum_{n \in \Z^d} |n|^{-2} |u(n)|^2 &= \int_{Q_d} |\nabla \psi(x)|^2 dx \\
        \sum_{n \in \Z^d} |Du(n)|^2 &= 4\int_{Q_d} |\Delta \psi(x)|^2 \omega(x) dx,
    \end{split}
\end{equation}
where $\psi$ is a smooth $2\pi$-periodic function with zero average satisfying $\widehat{u_j} = \partial_{x_j}\psi$. 

We claim that if $u$ is antisymmetric then so is $\psi$: writing the Fourier expansion of $\widehat{u}$ and $\psi$, we get
\begin{align*}
    \widehat{u_j} = (2\pi)^{-\frac{d}{2}}\sum_{n} \frac{n_j}{|n|^2} u(n) e^{-i n \cdot x} = (2\pi)^{-\frac{d}{2}}\sum_{n} -i n_j a(n) e^{-in \cdot x} = \partial_{x_j} \psi,
\end{align*}
where $\psi = (2\pi)^{-\frac{d}{2}} \sum_n a(n) e^{-i n \cdot x}$. This implies that $$ a(n) = i |n|^{-2} u(n).$$
Antisymmtry of $u$ implies that $a(n)$ is antisymmetric, which further imply that $\psi$ is antisymmetric. Thus, proving \eqref{A8} reduces to proving Hardy inequality on a torus for antisymmetric functions $\psi$:
$$ \int_{Q_d} \frac{|\psi|^2}{\omega} dx \leq c_d \int_{Q_d} |\nabla \psi|^2 dx,$$
since (consequence of integration by parts and H\"older's inequality)
$$ \int_{Q_d} |\nabla \psi|^2 dx \leq \Big(\int_{Q_d}|\Delta \psi|^2 w(x) dx\Big)^{1/2} \Big(\int_{Q_d}\frac{|\psi|^2}{\omega} dx\Big)^{1/2}.$$
In order to prove Hardy inequality for antisymmetric functions we need the following Poincar\'e inequality:
\begin{lemma}
Let $\psi \in C^\infty(\overline{Q_d})$ which is $2\pi$-periodic in each variable. Furthermore assume that $\psi$ satisfies the antisymmetric condition \eqref{A2}. Then we have
\begin{equation}\label{A10}
      C_p(d) \int_{Q_d} |\psi(x)|^2 dx \leq \int_{Q_d} |\nabla \psi(x)|^2 dx,
\end{equation}
where
\begin{equation}\label{A11}
    C_p(d) := N(N-1)(2N-1)/3 + (3-(-1)^d)N^2/2, 
\end{equation}
and $N := \lfloor d/2 \rfloor$. Moreover the constant is sharp: any non-zero antisymmetric function $\psi(x) = (2\pi)^{-d/2} \sum_{n} \widehat{\psi}(n) e^{-i n \cdot x}$ with $\widehat{\psi}(n) = 0$ for $|n|^2 \neq C_p(d)$ optimize \eqref{A10}.
\end{lemma}

\begin{proof}
Expanding $\psi(x) = (2\pi)^{-d/2} \sum_{n \in \Z^d} \widehat{\psi}(n) e^{-i n \cdot x}$ and using Parseval's identity we get
$$ \int_{Q_d} |\psi|^2 = \sum_{n \in \Z^d} |\widehat{\psi}(n)|^2, \hspace{19pt} \int_{Q_d} |\nabla \psi|^2 = \sum_{n\in \Z^d} |n|^2 |\widehat{\psi}(n)|^2.$$

We start with a simple observation about the Fourier coefficients of an antisymmetric function $\psi$ : Let $n = (n_1, \dots, n_d) \in \Z^d$ such that $n_i \neq n_j$ for $1\leq i \neq j \leq d$. Then it can be verified that
$$ |n|^2 = n_1^2+\dots+n_d^2 \geq N(N-1)(2N-1)/3 + (3-(-1)^d)N^2/2 =: C_p(d),$$
where $N:= \lfloor d/2 \rfloor$.

In other words, if $|n|^2 < C_p(d)$, there exist $i \neq j$ such that $n_i=n_j$. Since $\widehat{\psi}$ is antisymmetric we have 
$$ \widehat{\psi}(\dots, n_i, \dots, n_j, \dots) = \widehat{\psi}(\dots, n_j, \dots, n_i, \dots) = - \widehat{\psi}(\dots, n_i, \dots, n_j, \dots),$$
which imply $\widehat{\psi}(n) = 0$. Therefore we have 
$$ \int_{Q_d} |\nabla \psi|^2 = \sum_{n \in \Z^d} |n|^2 |\widehat{\psi}(n)|^2 = \sum_{|n|^2 \geq C_p(d)} |n|^2 |\widehat{\psi}(n)|^2 \geq C_p(d) \sum_{n \in \Z^d} |\widehat{\psi}|^2 = C_p(d)\int_{Q_d} |\psi|^2 dx.$$
Moreover the above inequality becomes an equality for any antisymmetric function satisying $\widehat{\psi}(n) = 0$ for $|n|^2 \neq C_p(d)$. Existence of such a function relies on an algebraic fact: any permutation of $\{1,\dots,d\}$ can be written as a product of either even or odd number of transpositions but not both. 
\end{proof}
We are now ready to prove Hardy inequality for antisymmetric functions on torus. 
\begin{lemma}
Let $\psi \in C^\infty (\overline{Q_d})$ all of whose derivatives are $2\pi$-periodic in each variable. Furthermore assume that $\psi$ satisfies the antisymmetric condition \eqref{A2}. Then  for $d \geq 3$ we have
\begin{equation}\label{A12}
    \frac{d(d-2)^2 C_p(d)}{16dC_p(d)+(3d-2)(d-2)}\int_{Q_d} \frac{|\psi(x)|^2}{\sum_{j=1}^d \sin^2(x_j/2)} dx \leq \int_{Q_d} |\nabla \psi(x)|^2 dx,
\end{equation}
where $C_p(d)$ is as defined by \eqref{A11}.
\end{lemma}
\begin{proof}
Let $F = (F_1, \dots, F_d)$ with $F_j := \frac{\sin x_j}{\omega}$. Exanding the sqaure $|\nabla \psi + \frac{(d-2)}{8}\psi F|^2$ and applying integration by parts we obtain (see proof of Theorem \ref{thm3.4} for details)
\begin{equation}\label{A13}
   \int_{Q_d} \frac{|\psi(x)|^2}{\sum_{j=1}^d \sin^2(x_j/2)} dx \leq \frac{16}{(d-2)^2} \int_{Q_d}|\nabla \psi(x)|^2  dx + \frac{3d-2}{d(d-2)}\int_{Q_d}|\psi(x)|^2 dx.
\end{equation}
Inequality \eqref{A13} along with the Poincar\'e inequality for antisymmetric functions \eqref{A10} completes the proof.
\end{proof}
\begin{proof}[Proof of Theorem \ref{thmA3}]
H\"older's inequality along with integration by parts gives
\begin{align*}
    \int_{Q_d} |\nabla \psi|^2 dx = -\int_{Q_d} \Delta \psi \overline{\psi} dx \leq \Big(\int_{Q_d}|\Delta \psi|^2 w(x) dx\Big)^{1/2} \Big(\int_{Q_d}\frac{|\psi|^2}{\omega(x)} dx\Big)^{1/2}.
\end{align*}
Next using Hardy inequality \eqref{A12} we obtain
\begin{equation}\label{A14}
    \frac{d(d-2)^2 C_p(d)}{16dC_p(d)+(3d-2)(d-2)}\int_{Q_d} |\nabla \psi|^2 dx \leq \int_{Q_d} |\Delta \psi|^2 \omega(x) dx.
\end{equation}
Inequality \eqref{A14} along with \eqref{A9} completes the proof.

\end{proof}

\begin{remark}
We note that the constant in \eqref{A12} grows as $d^2$ as $d \rightarrow \infty$ (cf. Corollary \ref{cor3.3}). We believe that there is still a room for improvement and the sharp constant should grow as $d^4$ as $d \rightarrow \infty$.   
\end{remark}

\chapter{Supersolution method for infinite graphs}\label{appendix:B}
In this chapter we assume $G = (V, E)$ is a graph with countably infinite vertex set $V$. We also assume that $G$ is connected and is \emph{locally finite}: each vertex has finitely many neighbours. For a function $\varphi: V \rightarrow \R$ we define its \emph{Laplacian} by 
$$ \Delta \varphi(x) := \sum_{y \sim x} (u(x)-u(y)),$$
where $x \sim y$ means $(x, y) \in E$. 

We prove a result similar to Lemma \ref{lem2.1} for graph $G$. 

\section{Generalization of Lemma \ref{lem2.1}}
\begin{lemma}\label{lemB1}
Let $G= (V, E)$ be a graph and let $\varphi: V \rightarrow \R$ be a positive function. Then we have the following Hardy inequality 
\begin{equation}
    \frac{1}{2}\sum_{x \sim y}|u(x)-u(y)|^2 \geq \sum_{x \in V} \frac{\Delta \varphi(x)}{\varphi(x)} |u(x)|^2,
\end{equation}
for all finitely supported functions $u$ on the vertex set $V$.
\end{lemma}
\begin{proof}
It can checked that 
\begin{equation}\label{B1}
    (a-t)^2 \geq (1-t)(a^2 -t),
\end{equation}
for $a \in \R$ and $t \geq 0$. Let $\psi(x) := u(x)/\varphi(x)$, $x \in V$. Assuming $\psi(y) \neq 0$ and applying \eqref{B1} for $a = \psi(x)/\psi(y)$ and $t = \varphi(y)/\varphi(x)$ we get
\begin{equation}\label{B2}
    |\varphi(x)\psi(x) - \varphi(y)\psi(y)|^2 \geq (\varphi(x)-\varphi(y))(\psi(x)^2 \varphi(x) - \psi(y)^2 \varphi(y)).
\end{equation}
The above inequality is true even when $\psi(y)=0$ since $\varphi(x) >0$. Using \eqref{B2} we obtain
\begin{align*}
    \sum_{x \sim y} |u(x)-u(y)|^2 &= \sum_{x \sim y}|\varphi(x)\psi(x) - \varphi(y)\psi(y)|^2 \\
    &\geq \sum_{x \sim y}(\varphi(x)-\varphi(y))(\psi(x)^2 \varphi(x) - \psi(y)^2 \varphi(y)) = 2 \sum_{x \in V} \frac{\Delta \varphi(x)}{\varphi(x)} |u(x)|^2.
\end{align*}
\end{proof}
\section{Hurdles in higher dimensions}
In this section, we restrict ourselves to the standard graph on $\Z^d$. We define $d$-\emph{dimensional lattice graph} as a graph whose vertex set is $\Z^d$, and $x, y \in \Z^d$ are connected by an edge if and only if $\|x-y\|_{\ell^1} := \sum_{i=1}^d |x_i-y_i| = 1$. We note that for lattice graphs the term in the LHS of \eqref{B1} becomes
$$ \frac{1}{2}\sum_{x \sim y}|u(x)-u(y)|^2 = \sum_{n \in \Z^d} \sum_{j=1}^d |u(n)-u(n-e_j)|^2,$$
with $e_j$ being the standard $j^{th}$ basis element of $\R^d$ and the Laplcian becomes
$$ \Delta \varphi(n) := \sum_{j=1}^d 2u(n)-u(n-e_j)-u(n+e_j).$$

With the aim of proving discrete Hardy inequality \eqref{2.4} we choose $\varphi(n) := |n|^\alpha$ for a real number $\alpha$ (to be chosen later) in Lemma \ref{lemB1}:
\begin{equation}
    \sum_{n \in \Z^d} \sum_{j=1}^d |u(n)-u(n-e_j)|^2 \geq \sum_{n \in \Z^d} w(n) |u(n)|^2,
\end{equation}
where $w(n):= \frac{\Delta |n|^\alpha}{|n|^\alpha}.$ Using Taylor's expansion of $\varphi$ we obtain 
\begin{align*}
    w(n) = -\alpha(\alpha+d-2) \frac{1}{|n|^2} + \text{lower order terms}. 
\end{align*}
We choose $\alpha = -(d-2)/2$ with the aim of maximizing $-\alpha(\alpha+d-2)$ and obtain 
\begin{equation}\label{B5}
    w(n) = \frac{(d-2)^2}{4} \frac{1}{|n|^2} + \text{lower order terms}.
\end{equation}
For $d=1$ it is not very hard to see that `lower order terms' are non-negative and hence we Hardy inequality \eqref{2.4} for $d=1$. However, for $d \geq 2$ the lower order terms starts to have both positive and negative terms and their analysis becomes very hard. Therefore, it is not straightforward to obtain Hardy inequality \eqref{2.4} for $d \geq 2$ along this route. In fact, for $d$ large, the lower order terms in \eqref{B5} are bound to be negative, since the sharp constant in \eqref{2.4} grows linearly as $d \rightarrow \infty$ (see Corollary \ref{cor3.1}).

\chapter{Integral identity}\label{appendix:C}
We prove identity \eqref{2.67} with $w_{ij}$ replaced with an arbitrary smooth $2\pi$ periodic funcition.

\begin{lemma}
Let $u, w \in C^\infty[-\pi, \pi]$ such that their derivatives satisfy $d^k u(-\pi) = d^k u(\pi)$ and $d^k w(-\pi) = d^k w(\pi)$, for all $k \in \N_0$. Then for non-negative integers $0 \leq i<j$ we have
\begin{equation}\label{C1}
    I(i, j, w) :=\text{Re} \int_{-\pi}^{\pi} d^i u(x) \overline{d^ju(x)} w(x) dx = \sum_{\sigma = i}^{\lfloor\frac{i+j}{2}\rfloor} \int_{-\pi}^{\pi} C_{\sigma, w}^{i,j}(x) |d^\sigma u|^2,
\end{equation}
where $C_{\sigma,w}^{i,j}$ is given by
\begin{align*}
    C_{\sigma, w}^{i,j}(x) = {j-\sigma -1 \choose \sigma -i-1}(-1)^{j-\sigma}d^{i+j-2\sigma}w(x) + \frac{1}{2}{j-\sigma-1 \choose \sigma-i}(-1)^{j-\sigma}d^{i+j-2\sigma}w(x).
\end{align*}
\end{lemma}

\begin{proof}
We prove the result using induction on the parameter $k:= j-i$. Let us assume that \eqref{C1} is true for all $0 \leq i < j$ such that $3 \leq j-i \leq k $. Consider non-negative integers $i<j$ such that $j-i = k+1$. Then integration by parts yields
\begin{align*}
    I(i, j, w) &= \text{Re} \int_{-\pi}^{\pi} d^i u(x) \overline{d^ju(x)} w(x) dx\\
    &= -\text{Re} \int_{-\pi}^{\pi} d^i u(x) \overline{d^{j-1}u(x)} w'(x) dx - \text{Re} \int_{-\pi}^{\pi} d^{i+1} u(x) \overline{d^{j-1}u(x)} w(x) dx\\
    &= -I(i, j-1, w') - I(i+1, j-1, w).
\end{align*}
Further using induction hypothesis we get
\begin{equation}\label{C2}
    \begin{split}
        I(i, j, w) = &\sum_{\sigma=i+1}^{\lfloor (i+j-1)/2 \rfloor} \int_{-\pi}^\pi \big(-C_{\sigma, w'}^{i, j-1} - C_{\sigma, w}^{i+1,j-1}\big) |d^\sigma u|^2 dx - \int_{-\pi}^\pi C_{i, w'}^{i, j-1}|d^{i}u|^2 dx \\
        &- \delta(i+j)\int_{-\pi}^\pi C_{\lfloor(i+j)/2 \rfloor ,w}^{i+1, j-1}|d^{\lfloor(i+j)/2 \rfloor}u|^2 dx,   
    \end{split}
\end{equation}
where $\delta(\text{odd numbers}) := 0$ and $\delta(\text{even numbers}) := 1$. Using identity ${n \choose r} + {n \choose r-1} = {n+1 \choose r}$ we obtain
\begin{equation}\label{C3}
    -C_{\sigma, w'}^{i, j-1} - C_{\sigma, w}^{i+1,j-1} = C_{\sigma, w}^{i, j}.
\end{equation}
It can also be checked that $-C_{i, w'}^{i, j-1} = C_{i, w}^{i, j} = \frac{1}{2}d^{j-1} w$ as well as $-C_{\lfloor(i+j)/2 \rfloor ,w}^{i+1, j-1} = C_{\lfloor(i+j)/2 \rfloor ,w}^{i, j} = (-1)^{j-i} w$ (for even $i+j$). These observations along with \eqref{C3} and \eqref{C2} proves \eqref{C1} for $3 \leq j-i = k+1$.\\

The base cases $j-i \in \{1, 2, 3\}$ can be checked by hand (it's a consequence of iterative integration by parts).
\end{proof}

\chapter{Fourier rearrangement in higher dimensions}\label{appendix:D}
In this chapter, we extend the notion of Fourier rearrangement defined on integers in Chapter \ref{ch:1d-rearrangement} to $\Z^d$ for $d \geq 1$. The main result of the chapter is an extension of Theorem \ref{thm4.5} to higher dimensional integers lattices.

Let $u \in \ell^2(\Z^d)$ and $Q_d := (-\pi, \pi)^d$. Then the Fourier transform $F(u) \in L^2((-\pi,\pi)^d)$ of function $u$ is given by 
\begin{align*}
    F(u)(x) := (2\pi)^{-\frac{d}{2}}\sum_{n \in \mathbb{Z}^d} u(n) e^{-in\cdot x}, \hspace{19pt} x \in Q_d,
\end{align*}
and the inverse of $F$ is given by
\begin{align*}
    F^{-1}(u)(n) = (2\pi)^{-\frac{d}{2}}\int_{Q_d} u(x) e^{i n \cdot x} dx. 
\end{align*}

We define the \emph{Fourier rearrangement} $u^\#$ of function $u \in \ell^2(\mathbb{Z}^d)$ as
\begin{equation}
    u^\# := F^{-1}(F(u)^{e_1e_2...e_d}), 
\end{equation}
where $e_i$ is the $i^{th}$ standard basis of $\R^d$ and $f^{e_i}$ is the Steiner symmetrization of $f$ in the direction $e_i$ (see \cite[Chapter 3]{liebloss}).

Some basic properties of $u^\#$:
\begin{enumerate}
    \item $u^\#$ is symmetric in the direction $e_i$. Namely 
    \begin{equation}
        u^\#(n_1,..,n_i,..n_d) = u^\#(n_1,..,-n_i,..n_d).
    \end{equation}
    Proof: Doing change of variable $x_i \mapsto -x_i$ we obtain\\
    \begin{align*}
        u^\#(n_1,..,-n_i,..,n_d) &= (2\pi)^{-\frac{d}{2}} \int_{Q_d}F(u)^{e_1..e_d}(x_1,..,-x_i,..,x_d) e^{i n\cdot x}dx
        \\
        &= (2\pi)^{-\frac{d}{2}} \int_{Q_d}F(u)^{e_1..e_d}(x) e^{i n\cdot x}dx = u^\#(n_1,..,n_i,..,n_d). 
    \end{align*}
    In the last step we used the symmetry of $f^{e_1, \dots, e_d}$ in the direction $e_i$.
    \item $u^\#$ is real valued follows directly from the fact that $F(u)^{e_1,..,e_d}$ is real valued and symmetric in every direction $e_i$.
    
    \item Parseval's identity along with the fact that Steiner symmetrization preseves $L^2$ norm gives
    \begin{align*}
        \sum_{n \in \mathbb{Z}^d}|u(n)|^2 = \sum_{n \in \mathbb{Z}^d}|u^\#(n)|^2.
    \end{align*}
    
    \item (Hardy-Littlewood inequality) The following Hardy-Littlewood inequality follows from parseval's identity and Hardy-littlewood inequality for steiner symmetrization.
    \begin{align*}
        \Big|\sum_{n \in \mathbb{Z}^d} u(n)\overline{v(n)}\Big| \leq \sum_{n \in \mathbb{Z}^d} u^\#(n) v^\#(n).
    \end{align*}
    \item Parseval's identity along with $L^2$ contraction property of steiner symmetrization gives
    \begin{align*}
        \sum_{n \in \mathbb{Z}^d}|u^\#(n)-v^\#(n)|^2 \leq \sum_{n \in \mathbb{Z}^d} |u(n)-v(n)|^2.
    \end{align*}
\end{enumerate}

\begin{definition}
Let $u: \Z^d \rightarrow \C$ be a function. Then its gradient is given by
$$ Du(n) := (D_1u(n), \dots, D_d u(n)),$$
where $D_iu(n):= u(n)-u(n-e_i)$ and its Laplacian is given by
$$ \Delta u(n) := \sum_{i=1}^d (2u(n)-u(n-e_i)-u(n+e_i)).$$ 
\end{definition}

\begin{theorem}
Let $u \in \ell^2(\mathbb{Z}^d)$ and $k \geq 0$ then we have
\begin{equation}\label{D3}
    \sum_{n \in \mathbb{Z}^d}|\Delta^k u(n)|^2 \geq \sum_{n \in \mathbb{Z}^d}|\Delta^k u^\#(n)|^2,
\end{equation}
and 
\begin{equation}\label{D4}
    \sum_{n \in \mathbb{Z}^d}|D(\Delta^k u)(n)|^2 \geq \sum_{n \in \mathbb{Z}^d}|D(\Delta^k u^\#)(n)|^2.
\end{equation}
\end{theorem}

\begin{proof}
We begin with computing the Fourier transform of $D_j u$ and $\Delta u$:
\begin{align*}
    D_j u(n)&= u(n)-u(n-e_j) = (2\pi)^{-\frac{d}{2}}\int_{Q_d} F(u)(1-e^{-ix_j})e^{i n \cdot x} dx,\\
    \Delta u(n) &= \sum_{j=1}^d 2u(n)-u(n+e_j)-u(n-e_j) = (2\pi)^{-\frac{d}{2}} \int_{Q_d} F(u)\sum_{j=1}^d [2-e^{ix_j}-e^{-ix_j}] e^{i n \cdot x}dx.
\end{align*}
Therefore we have 
\begin{align*}
    F(D_ju) &= F(u)(1-e^{-ix_j}),\\
    F(\Delta u ) &= F(u)\sum_{j=1}^d[2-e^{ix_j}-e^{-ix_j}] = 4F(u)\sum_{j=1}^d\sin^2(x_j/2). 
\end{align*}
Using above expressions and Parseval's identity we obtain
\begin{align*}
    \sum_{n \in \mathbb{Z}^d} |\Delta^k u(n)|^2 = \int_{Q_d} |F(\Delta^k u)|^2 dx = 4^{2k}\int_{Q_d} |F(u)|^2 \Big(\sum_{j=1}^d \sin^2(x_j/2)\Big)^{2k} dx,
\end{align*}
and 
\begin{align*}
    \sum_{n \in \mathbb{Z}^d} |D(\Delta^k u)(n)|^2 = \int_{Q_d} |F(D\Delta^k u)|^2 dx  = 4^{2k+1} \int_{Q_d}|F(u)|^2 \Big(\sum_{j=1}^d\sin^2(x_j/2) \Big)^{2k+1} dx.
\end{align*}
Therefore proving \eqref{D3} and \eqref{D4} reduces to proving 
\begin{equation}\label{D5}
    \int_{Q_d}|F(u)|^2 \Big(\sum_{j=1}^d \sin^2(x_j/2)\Big)^{2k} \geq \int_{Q_d}|F(u)^{e_1..e_d}|^2 \Big(\sum_{j=1}^d \sin^2(x_j/2)\Big)^{2k} dx,
\end{equation}
and
\begin{equation}\label{D6}
    \int_{Q_d}|F(u)|^2 \Big(\sum_{j=1}^d \sin^2(x_j/2)\Big)^{2k+1} \geq \int_{Q_d}|F(u)^{e_1..e_d}|^2 \Big(\sum_{j=1}^d \sin^2(x_j/2)\Big)^{2k+1} dx
\end{equation}
respectively. Now we prove an extension of inequality \eqref{4.39} which implies inequalities \eqref{D5} and \eqref{D6}. 

\emph{Claim}: Let $f$ be a non-negative measurable function  vanishing at infinity and $g$ be a non-negative measurable function which is radially increasing with respect to each variable variable. Then we have
\begin{equation}\label{D7}
    \int_{\mathbb{R}^d} f(x) g(x) dx \geq \int_{\mathbb{R}^d} f^{e_1e_2...e_d}(x) g(x) dx.
\end{equation}

\begin{proof}
From \eqref{4.39} we know
\begin{align*}
    \int_{\mathbb{R}}f(x_1,..,x_d)g(x_1,..,x_d) dx_i \geq \int_{\mathbb{R}}f^{e_i}(x_1,..,x_d)g(x_1,..,x_d) dx_i. 
\end{align*}
Integrating with respect to other variables and using Fubini's theorem yields 
\begin{equation}\label{D8}
    \int_{\mathbb{R}^d} f(x) g(x) dx \geq \int_{\mathbb{R}^d} f^{e_i}(x) g(x) dx.
\end{equation}
Finally applying \eqref{D8} iteratively from $e_1$ to $e_d$ proves \eqref{D7}.
\end{proof}

For the proof of \eqref{D5} we choose $f:= |F(u)|^2$ on $Q_d$ and zero on the complement and define $g := \Big(\sum_{i=1}^d \sin^2(x_i/2)\Big)^{2k}$ on $Q_d$ and extend $g$ in the complement of $Q_d$ such that its radially increasing in each variable. Now applying \eqref{D7} with this choice of $f$ and $g$ proves \eqref{D5}. One can similarly prove \eqref{D6}.

\end{proof}

\chapter{A direct proof of Wang-Wang rearrangement inequality}\label{appendix:E}
In this chapter we give a short proof of Theorem \ref{thm5.2} in arbitrary dimensions. This proof does not use the framework of universal comparison graph developed in Chapter \ref{ch:higher-rearrangement}.  

\begin{lemma}
Let $f: \Z^d \rightarrow \R$ be a function vanishing at infinity and $f^*$ be its Wang-Wang rearrangement. Then
\begin{equation}
    \|\nabla f^*\|_p \leq d^{1/p} \|\nabla f\|_p.
\end{equation}
\end{lemma}

\begin{proof}
Coarea formula for $f$ (see Lemma \ref{lem5.5}) gives
$$ \|\nabla f\|_p^p = p \int_0^\infty X(f)(s) ds,$$
where $$X(f)(s) := \sum_{(x,y) \in \partial_E(\{f>s\})} |s-f(y)|^{p-1}.$$
Let $f^*(1)\geq f^*(2) \geq \dots f^*(i) \geq s \geq f^*(i+1) \dots$. Consider
\begin{align*}
    X(f)(s) &\geq \sum_{y \in \partial_V(\{f>s\})} |s-f(y)|^{p-1} \geq \sum_{j=1}^{\partial_V^i} |s-f(i+j)|^{p-1}.
\end{align*}
On the other hand we have
\begin{align*}
    X(f^*)(s) &=\sum_{(x,y) \in \partial_E(\{f^*>s\})} |s-f^*(y)|^{p-1}\\
    & \leq d \sum_{y \in \partial_V(\{f^*>s\})}|s-f^*(y)|^{p-1} = d \sum_{j=1}^{\partial_V^i} |s-f^*(i+j)|^{p-1}.
\end{align*}

First we used that each vertex in $\partial_V(\{f^* > s\})$ can have at most $d$ neighbours in $\{f^* > s\}$ (its a consequence of the fact that Wang-Wang enumeration fills up $\ell^1$ balls in $\Z^d$). Secondly the fact that $\partial_V(\{1, \dots, i\}) = \{i+1, \dots, i+\partial_V^i\}$ (see \cite{bez}). 
\end{proof}